\documentclass[a4paper, 12pt, reqno]{amsart}

\usepackage[utf8]{inputenc}
\usepackage[T1]{fontenc}
\usepackage{indentfirst}
\usepackage{textcomp,lscape,rotating}
\usepackage[english]{babel}
\usepackage{enumerate}
\usepackage{amsmath,amsfonts,amssymb,amsopn,amscd,amsthm}
\usepackage{mathrsfs,bbm}
\usepackage{graphicx}
\usepackage[dvipsnames]{xcolor}
\usepackage[colorlinks=true,citecolor=Red,linkcolor=Blue]{hyperref}
\usepackage{epigraph}
\usepackage{tikz,tkz-graph,tkz-berge}
\usetikzlibrary{arrows,shapes,matrix}
\usepackage{pgfplots}

\title[Approximation of discrete measures and harmonic analysis on the torus]{Mod-$\phi$ convergence: \\ Approximation of discrete measures\\ and harmonic analysis on the torus}

\author{Reda Chhaibi}
\address{Institut de Math\'ematiques de Toulouse -- Universit\'e Paul Sabatier -- 118, route de Narbonne, 31062 Toulouse cedex 9, France}
\email{reda.chhaibi@math.univ-toulouse.fr}

\author{Freddy Delbaen}
\address{ETH Z\"urich, Department of Mathematics and Universit\"at Z\"urich, Institut für Mathematik -- Winterthurerstrasse 190, CH-8057 Z\"urich, Switzerland}
\email{delbaen@math.ethz.ch}

\author{Pierre-Lo\"ic M\'eliot}
\address{Laboratoire de Math\'ematiques, B\^atiment 425 -- Facult\'e des sciences d'Orsay -- Universit\'e Paris-Sud, 91400 Orsay, France}
\email{pierre-loic.meliot@math.u-psud.fr}

\author{Ashkan Nikeghbali}
\address{Universit\"at Z\"urich -- Institut f\"ur Mathematik --
 Winterthurerstrasse 190, CH-8057 Z\"urich, Switzerland}
\email{ashkan.nikeghbali@math.uzh.ch}

\def\half{\frac{1}{2}}
\def\l{\lambda}
\newcommand{\card}{\mathrm{card}\,}
\newcommand{\scal}[2]{\left\langle #1\vphantom{#2}\,\right |\left.#2 \vphantom{#1}\right\rangle}
\newcommand{\normA}[1]{\left\| #1 \right\|_{\Ac} }
\newcommand{\eps}{\varepsilon}
\newcommand{\sym}{\mathfrak{S}}

\def\N{{\mathbb N}}
\def\Z{{\mathbb Z}}

\def\R{{\mathbb R}}
\def\C{{\mathbb C}}
\def\T{{\mathbb T}}

\def\P{{\mathbb P}}
\newcommand{\frakp}{\mathfrak{p}}
\newcommand{\frake}{\mathfrak{e}}
\newcommand{\lle}{\left[\!\left[} 
\newcommand{\rre}{\right]\!\right]} 
\newcommand{\id}{\mathrm{id}}
\newcommand{\I}{\mathrm{i}}
\newcommand{\E}{\mathrm{e}}
\newcommand{\dloc}{d_{\mathrm{L}}}
\newcommand{\dkol}{d_{\mathrm{K}}}
\newcommand{\dtv}{d_{\mathrm{TV}}}
\def\esper{{\mathbb E}}
\def\proba{{\mathbb P}}
\newcommand{\leb}{\mathscr{L}}
\newcommand{\Ac}{\mathscr{A}}
\newcommand{\Sym}{\mathrm{Sym}}
\newcommand{\comment}[1]{}

\renewcommand{\Re}{\mathrm{Re}}
\newtheorem{lemma}{Lemma}[section]
\newtheorem{corollary}[lemma]{Corollary}

\newtheorem{proposition}[lemma]{Proposition}

\newtheorem{theorem}[lemma]{Theorem}
\newtheorem{definition}[lemma]{Definition}
\theoremstyle{remark}
\newtheorem{example}[lemma]{Example}
\newtheorem{remark}[lemma]{Remark}

\begin{document}

\begin{abstract}
In this paper, we relate the framework of mod-$\phi$ convergence to the construction of approximation schemes for lattice-distributed random variables. The point of view taken here is that of Fourier analysis in the Wiener algebra, allowing the computation of asymptotic equivalents in the local, Kolmogorov and total variation distances. By using signed measures instead of probability measures, we are able to construct better approximations of discrete lattice distributions than the standard Poisson approximation. This theory applies to various examples arising from combinatorics and number theory: number of cycles in (possibly coloured) permutations, number of prime divisors (possibly within different residue classes) of a random integer, number of irreducible factors of a random polynomial, \emph{etc.} One advantage of the approach developed in this paper is that it allows us to deal with approximations in higher dimensions as well. In this setting, we can explicitly see the influence of the correlations between the components of the random vectors in our asymptotic formulas.
\end{abstract}


\maketitle

\hrule
\tableofcontents
\hrule
\bigskip
\bigskip


\section{Introduction}

\subsection{Poisson approximation of lattice-valued random variables} \label{subsec:poissonapprox}
Consider a sequence $(B_i)_{i \geq 1}$ of independent Bernoulli random variables, with $\proba[B_i=1]=p_i$ and $\proba[B_i=0]=1-p_i$. We set 
$$X_{n}=\sum_{i=1}^{n} B_i.$$
The \emph{Poisson approximation} ensures that if the $p_i$'s are small but their sum $\lambda_n =\sum_{i=1}^n p_i$ is large, then the distribution of $X_n$ is close to the distribution of a Poisson random variable with parameter $\lambda_n = \sum_{i=1}^n p_i$. A first quantitative result in this direction is due to Prohorov and Kerstan (see \cite{Pro53,Ker64}): if $p_i=\frac{\lambda}{n}$ for all $i \in \lle 1,n\rre$, then
$$\dtv(X_n,\mathcal{P}(\lambda)):=\sum_{k \in \N} \left|\proba[X_n =k] - \E^{-\lambda}\,\frac{\lambda^k}{k!}\right| \leq \frac{2\lambda}{n}.$$
More generally, with parameters $p_i$ that can be distinct and with $\lambda_n=\sum_{i=1}^n p_i$, Le Cam showed that
$$\sum_{k \in \N} \left|\proba[X_n =k] - \E^{-\lambda_n}\,\frac{(\lambda_n)^k}{k!}\right| \leq 2\,\sum_{i=1}^n (p_i)^2.$$
This is an immediate consequence of the inequality on total variation distances 
$$\dtv(\mu_1*\mu_2,\nu_1*\nu_2) \leq \dtv(\mu_1,\nu_1) + \dtv(\mu_2,\nu_2)$$ 
which holds for any probability measures $\mu_1,\mu_2,\nu_1,\nu_2$ on $\Z$ (\emph{cf.}~\cite{LeCam60}). By using arguments derived from Stein's method for the Gaussian approximation, Chen and Steele obtained improved versions of this inequality, \emph{e.g.},
$$\sum_{k \in \N} \left|\proba[X_n =k] - \E^{-\lambda_n}\,\frac{(\lambda_n)^k}{k!}\right| \leq 2\,(1-\E^{-\sum_{i=1}^n p_i})\,\frac{\sum_{i=1}^n (p_i)^2}{\sum_{i=1}^n p_i},$$
see \cite{Chen74,Chen75,Steele94}. We also refer to \cite{AGG89} for an extension to possibly dependent Bernoulli random variables, and to \cite{BHS92} for a survey of the theory of Poisson approximations.\medskip

More generally, consider random variables $X_{n\geq 1}$ whose distributions are supported by the lattice $\Z^d \subset \R^d$, and which stem from a common probabilistic model. In many cases, the scaling properties of the model imply that when $n$ is large, $X_n$ can be approximated by a discrete infinitely divisible law $\nu_n$, the Lévy exponents of these reference laws being all proportional:
$$\widehat{\nu_n}(\xi) := \sum_{k \in \Z^d} \nu_n(k)\,\E^{\I\scal{\xi}{k}}=\E^{\lambda_n\phi(\xi)}, \quad \lambda_n \to +\infty.$$
To go beyond the classical Poisson approximation, one can try in this setting to write bounds on the total variation distance $\dtv(X_n,\nu_n)$, or on another metric which measures the convergence in law (the local distance, the Kolmogorov distance, the Wasserstein metric, \emph{etc.}). A convenient framework for this program is the notion of \emph{mod-$\phi$ convergence} developed by Barbour, Kowalski and Nikeghbali in \cite{BKN09}, see also the papers \cite{KN10,DKN15,JKN11,FMN16,FMN17}. The main idea is that, given a sequence of random variables $(X_{n})_{n \in \N}$ with values in a lattice $\Z^{d}$, and a reference infinitely divisible law $\phi$ on the same lattice, if one has sufficiently good estimates on the Fourier transform of the law $\mu_n$ of $X_n$ and on the ratio
$$ \frac{\widehat{\mu}_n(\xi)}{\widehat{\nu}_n(\xi)} = \esper[\E^{\I\xi X_n}]\,\E^{-\lambda_n \phi(\xi)},$$
then one can deduce from these estimates the asymptotics of the distribution of $X_n$: central limit theorems \cite{BKN09,JKN11}, local limit theorems \cite{KN10,DKN15}, large deviations \cite{FMN16}, speed of convergence \cite{FMN17}, \emph{etc.} In this paper, we shall use the framework of mod-$\phi$ convergence to compute the precise asymptotics of the distance between the law of $X_n$ and the reference infinitely divisible law, for various distances.\medskip

When approximating lattice-distributed random variables, another objective that one can pursue consists in finding better approximations than the one given by an infinitely divisible law. A general principle is that, if one allows signed measures instead of positive probability measures, then simple deformations of the infinitely divisible reference law can be used to get smaller distances, \emph{i.e.} faster convergences. In the setting of the classical Poisson approximation of sums of independent integer-valued random variables, this idea was used in \cite[Theorem 5.1]{BC02}. We also refer to \cite{Pre83,Kru86,BP96,Cek97,Cek98,CM99,BX99,JKK08} for other applications of the signed compound Poisson approximation (SCP). For mod-Poisson convergent random variables $(X_n)_{n \in \N}$, an unconditional upper bound on the distance between the law of $X_n$ and a signed measure $\nu_n$ defined by means of Poisson--Charlier polynomials was proven in \cite[Theorem 3.1]{BKN09}. In this paper, we shall set up a general approximation scheme for sequences of lattice-valued random variables, which will allow us:
\begin{itemize}
 	\item to obtain signed measure approximations whose distances to the laws of the variables $X_n$ are arbitrary negative powers of the parameter $\l_n$;
 	\item to compute the asymptotics of these distances (instead of an unconditional upper bound).
 \end{itemize} 
Thus, this paper can be regarded as a complement to \cite{BKN09}, with an alternative approach. In \cite[Proposition 4.1.1]{FMN16} and \cite{FMN17b}, we also used signed measures in order to approximate mod-$\phi$ convergent random variables, but with respect to a continuous infinitely divisible distribution $\phi$, and using a first-order deformation of the Gaussian distribution.
\medskip

In the remainder of this introductory section, we recall the main definitions from the theory of mod-$\phi$ convergence, and we explain the general approximation scheme. In subsequent sections, this approximation scheme will yield the main hypotheses of our Theorems (Sections \ref{sec:wiener} and \ref{sec:asymptoticdistance}), and we shall apply it to various one-dimensional and multi-dimensional examples coming from probability, analytic number theory, combinatorics, \emph{etc.} (Sections \ref{sec:oneexample} and \ref{sec:multidimensional}). We also present in this introduction the main tool that we shall use in this paper, namely, harmonic analysis in the Wiener algebra.\bigskip

\subsection{Infinitely divisible distributions and distances between probability measures}
Let us start by presenting the reference infinitely divisible laws and the distances between probability distributions that we shall work with. Fix a dimension $d\geq 1$. If $X$ is a random variable with values in $\Z^d$, we denote $\mu_X$ its probability law
$$\mu_X(k_1,k_2,\ldots,k_d)=\proba[X = (k_1,k_2,\ldots,k_d)],$$
and $\widehat{\mu}_X$ its Fourier transform, which is defined on the torus $\T^d=(\R/2\pi \Z)^d$:
$$\widehat{\mu}_X(\xi) = \esper[\E^{\I \scal{\xi}{X}}] = \sum_{k_1,\ldots,k_d \in \Z} \mu_X(k_1,\ldots,k_d)\,\,\exp\!\left(\I\sum_{j=1}^d k_j\xi_j\right).$$
Assume that the law of $X$ is infinitely divisible. Then the Fourier transform of $X$ writes uniquely as 
$$\widehat{\mu}_X(\xi) = \E^{\phi(\xi)},$$
where $\phi$ is a function that is $(2\pi\Z)^d$-periodic:
$$\forall (n_{1},\ldots,n_{d}) \in \Z^{d},\,\,\,\phi(\xi_{1}+2\pi n_{1},\ldots,\xi_{d}+2\pi n_{d})=\phi(\xi_{1},\ldots,\xi_{d}).$$
Moreover, the law $\mu_X$ is supported on $\Z^{d}$ and none of its sublattices if and only if the L\'evy--Khintchine exponent $\phi$ is periodic with respect to $(2\pi\Z)^{d}$ and none of its sublattices. We refer to \cite{Sato99} and \cite[Chapter 2]{SVH04} for details on lattice-distributed infinitely divisible laws (in the case $d=1$); see also the discussion of \cite[\S3.1]{FMN16}. Throughout this paper, we make the following assumptions on a reference infinitely divisible law $\mu_X$:
\begin{itemize}
 	\item $\mu_X$ has moments of order $2$;
 	\item and $\mu_X$ is not supported on any sublattice of $\Z^d$.
 \end{itemize}   
Then, the corresponding L\'evy--Khintchine exponent $\phi$ is twice continuously differentiable and one has the Taylor expansion around zero
$$ \phi(\xi) = \I \scal{m}{\xi} - \frac{\xi^t \Sigma \xi}{2} + o(|\xi|^2),$$
where $m=\esper[X]$, $\xi^t$ is the transpose of the column vector $\xi$, and $\Sigma$ is the covariance matrix of $X$, which is non-degenerate. Moreover, $\Re(\phi(\xi))$ admits a unique non-degenerate global maximum on $[0,2\pi]^d$ at $\xi=0$.\medskip

\begin{remark}
Let $(\gamma \in \R^d,A \in \mathrm{M}(d\times d,\R), \Pi)$ be the triplet of the Lévy--Khintchine representation of the Fourier transform of an infinitely divisible random variable $X$ (\emph{cf.} \cite[Theorem 8.1]{Sato99}). Then, $X$ is supported on $\Z^d$ and has a moment of order $2$ if and only if:
\begin{enumerate}
	\item $A=0$ and $\gamma \in \Z^d$;
	\item the Lévy measure $\Pi$ is supported on $\Z^d$ and has a second moment.
\end{enumerate} 
\end{remark}
\medskip

\begin{example}
With $d=1$, suppose that $X$ follows a Poisson distribution $\mathcal{P}_{(\lambda)}$ with parameter $\lambda$. In this case,
$$ \phi(\xi) = \lambda \left( \E^{\I \xi} - 1 \right)=\I\lambda \xi - \lambda\, \frac{\xi^{2}}{2}+o(\xi^{2}).$$ 
\end{example}
\medskip

\begin{example}
More generally, fix a random variable $Z$ with values in $\Z$, and consider the compound Poisson distribution
$ X=\sum_{i=1}^{\mathcal{P}_{(\l)}}Z_{i},$
where the $Z_{i}$'s are independent copies of $Z$, and $\mathcal{P}_{(\l)}$ is an independent Poisson variable. One obtains a new infinitely divisible law with values on the lattice $\Z$, and the corresponding L\'evy--Khintchine exponent is
$$\phi(\xi) = \lambda\left( \esper[\E^{\I \xi Z}]-1\right) = \I\lambda\, \esper[Z]\,\xi - \lambda\, \esper[Z^{2}]\, \frac{\xi^{2}}{2}+o(\xi^{2}).$$
\end{example}
\medskip

\begin{example}
The previous example is generic if one restricts oneself to infinitely divisible random variables with values in $\N = \{0,1,2,3,\dots\}$; see \cite{Kat67,Sam75}. Thus, a random variable $X$ with values in $\N$ is infinitely divisible if and only it admits a representation $$X=\sum_{i=1}^{\mathcal{P}_{(\lambda)}} Z_i $$ as a compound Poisson distribution. In this representation, the $Z_i$'s are independent copies of a random variable $Z$ with law $$\proba[Z=j]=\frac{\l_j}{\l} \quad \text{for all } j>0;$$ 
and $\mathcal{P}_{(\l)}$ is an independent Poisson variable with parameter $\l = \sum_{j=1}^\infty \lambda_j$. Then, another representation of $X$ is $X = \sum_{j=1}^\infty j\,U_j$, where the $U_j$'s are independent Poisson variables with parameters $\l_j$. These parameters $\l_j$ are the solutions of the system of equations
$$k\,\proba[X=k] = \sum_{j=1}^k\l_j\,j\,\proba[X=k-j], $$
and this system provides a numerical criterion of infinite divisibility: $X$ is infinitely divisible if and only if the solutions $\l_j$ are all non-negative.
\end{example}
\medskip

Given a random variable $X$ on $\Z^{d}$, the general question that we want to tackle is: how close is the law $\mu=\mu_X$ of $X$ to an infinitely divisible law $\nu$ with exponent $\phi$? For that purpose, one can choose different distances between probability measures, the typical ones being:\vspace{2mm}
\begin{itemize}
\item the \emph{local distance}:
$$ \dloc(\mu, \nu) := \sup_{k \in \Z^{d}} \left|\mu(\{k\})-\nu(\{k\})\right|.$$
\item the \emph{total variation distance}:
$$ \dtv(\mu, \nu) := 2\,\sup_{A \subset \Z^{d}} \left|\mu(A)-\nu(A)\right| = \sum_{k \in \Z^d} \left|\mu(\{k\})-\nu(\{k\})\right|.$$
\item and in dimension $1$, the \emph{Kolmogorov distance}:
$$ \dkol(\mu, \nu) := \sup_{k \in \Z} \left|\mu(\lle -\infty,k\rre)-\nu(\lle - \infty,k \rre)\right|=\sup_{k \in \Z} \left|\mu(\lle k,+\infty\rre)-\nu(\lle k,+\infty \rre)\right|,$$
where $\lle a,b\rre = \{a,a+1,a+2,\ldots,b\}$ denotes an integer interval.\vspace{2mm}
\end{itemize}
Note that we multiplied the total variation distance by $2$ in comparison to the standard definition. It is well known and immediate to check that
$$\dloc(\mu,\nu) \leq 2\,\dkol(\mu,\nu) \leq \dtv(\mu,\nu)$$
for any pair of signed measures $(\mu,\nu)$. The hardest distance to estimate is usually the total variation distance. Note that all the previous quantities metrize the convergence of signed measures on $\Z^{d}$ with respect to the weak or strong topology. As recalled before in the Poisson case, many approaches for the estimation of $\dloc$, $\dkol$ and $\dtv$ can be found in the literature, the most popular ones being Stein's method, coupling methods and semi-group methods, see \cite{DP86,Steele94}. The use of characteristic functions has long been considered not so effective, until the remarkable series of papers of Hwang \cite{Hwa96,Hwa98,Hwa99} (see also \cite[Section IX.2]{FSe09}). In \cite{Hwa99}, Hwang explained how an effective Poisson approximation can be achieved assuming the analyticity of characteristic functions. In a similar spirit, but with much weaker hypotheses, \cite{BKN09} explored this question by using the notion of mod-$\phi$ convergence. Our paper revisits this notion, with the purpose of showing that the phenomenons at stake find their source in the harmonic analysis of the torus. While \cite{BKN09} was devoted to the proof of \emph{unconditional} upper bounds on the distances $\dloc$, $\dkol$, $\dtv$, here we shall be interested in the \emph{asymptotics} of these distances.
\bigskip

\subsection{Mod-\texorpdfstring{$\phi$}{phi} convergent sequences of random variables}\label{subsec:modphi}
We consider an infinitely divisible law $\nu$ of exponent $\phi$ on $\Z^d$, and a sequence of random variables $(X_n)_{n \in \N}$ on $\Z^{d}$. Following \cite{JKN11,DKN15,FMN16}, we say that $(X_n)_{n \in \N}$ converges mod-$\phi$ with parameters $\lambda_n \to +\infty$ and limiting function $\psi$ if 
$$ \esper\!\left[ \E^{\I\scal{ \xi}{ X_n}} \right] = \E^{\lambda_n \phi(\xi) }\, \psi_n(\xi)$$
with 
$$\lim_{n \to \infty}\psi_n(\xi) =  \psi(\xi).$$
The precise hypotheses on the convergence $\psi_{n} \to \psi$ will be made explicit when needed; typically, we shall assume it to occur in some space $\mathscr{C}^{r}(\T^{d})$ endowed with the norm
$$\|f\|_{\mathscr{C}^r}=\sup_{|\alpha|\leq r} \sup_{\xi \in \T^d} |(\partial^\alpha\! f)(\xi)|.$$
Let us provide a few examples which should enhance the understanding of this notion, as well as show its large scope of applications.

\begin{example}\label{ex:toymodel}
Consider as in Section \ref{subsec:poissonapprox} a sum of independent random variables
$$X_{n}=\sum_{i=1}^{n} B_i,$$
with $B_i$ following a law $\mathcal{B}(p_i)$: $\proba[B_i=1]=1-\proba[B_i=0]=p_i$. Let us assume that $\sum_{i=1}^{\infty} p_{i}=+\infty$ and $\sum_{i=1}^{\infty} (p_{i})^{2}<+\infty$. Then, setting $\lambda_{n}=\sum_{i=1}^{n}p_{i}$ and $\mu_n = \mu_{X_n}$,
\begin{align*}
\widehat{\mu}_{n}(\xi)&=\prod_{i=1}^{n} \left(1+p_{i}(\E^{\I \xi}-1)\right)=\E^{\lambda_{n}(\E^{\I \xi}-1)}\,\prod_{i=1}^{n} \left(1+p_{i}(\E^{\I \xi}-1)\right)\E^{-p_{i}(\E^{\I \xi}-1)}\\
&=\E^{\lambda_{n}(\E^{\I \xi}-1)}\,\prod_{i=1}^{n} \left(1-\frac{(p_{i}(\E^{\I \xi}-1))^{2}}{2}\,(1+o(1))\right) = \E^{\lambda_{n}(\E^{\I \xi}-1)}\,\psi_{n}(\xi)
\end{align*}
with $\psi_{n}(\xi) \to \psi(\xi)=\prod_{i=1}^{\infty} \left(1+p_{i}(\E^{\I \xi}-1)\right)\E^{-p_{i}(\E^{\I \xi}-1)}$. This infinite product converges uniformly on the circle because of the hypothesis $\sum_{i=1}^{\infty} (p_{i})^{2} < \infty$. So, one has mod-Poisson convergence with parameters $\lambda_{n}$.
\end{example}
\medskip

The following two examples will later be generalized to the multidimensional setting.
\begin{example}\label{ex:sigman}
If $\sigma$ is a permutation of the integers in $[\![1,n]\!]$, denote $\ell(\sigma)$ its number of disjoint cycles, including the fixed points. We then set $\ell_n=\ell(\sigma_n)$, where $\sigma_n$ is taken at random uniformly among the $n!$ permutations of $\sym_n$. Using Feller's coupling (\emph{cf.} \cite{ABT03}), one can show that $\ell_n$ admits the following representation in law:
$$\ell_n = \sum_{i=1}^n \,\mathcal{B}\!\left(\frac{1}{i}\right),$$
where the Bernoulli random variables are independent. This representation will also be made clear by the discussion of \S\ref{subsec:colouredpermutation} in the present paper. By the discussion of the previous example, $(\ell_n)_{n \in \N}$ converges mod-Poisson with parameters $H_n = \sum_{i=1}^n \frac{1}{i}$ and limiting function
$$\prod_{i=1}^{\infty} \left(1+\frac{\E^{\I \xi}-1}{i}\right)\E^{-\frac{\E^{\I \xi}-1}{i}} = \frac{1}{\Gamma(\E^{\I\xi})}\,\E^{\gamma\,(\E^{\I\xi}-1)},$$
where $\gamma$ is the Euler--Mascheroni constant $\gamma=\lim_{n \to \infty} (H_n-\log n)$ (see \cite{Art64} for this infinite product representation of the $\Gamma$-function, due to Weierstrass). Thus, one can also say that $(\ell_n)_{n \in \N}$ converges mod-Poisson with parameters $\l_n=\log n$ and limiting function 
$$\psi(\xi) =\frac{1}{\Gamma(\E^{\I \xi})}.$$
\end{example}
\medskip

\begin{example}\label{ex:omegan}
As pointed out in \cite{KN10,JKN11}, mod-$\phi$ convergence is a common phenomenon in probabilistic number theory. For instance, denote $\omega(k)$ the number of distinct prime divisors of an integer $k\geq 1$, and $\omega(N_n)$ the number of distinct prime divisors of a random integer $N_n$ smaller than $n$, taken according to the uniform law. Then, it can be shown by using the Selberg--Delange method (\emph{cf.} \cite[\S II.5]{Ten95}) that
\begin{align*}
\esper[\E^{z\omega(N_n)}] &=  \frac{1}{n} \sum_{k=1}^{n} \E^{z\omega(k)} =   \E^{(\log \log n) \left( \E^{z}- 1 \right) }\left( \Psi(\E^{z}) + O\!\left(\frac{1}{\log n}\right) \right)
\end{align*}
with
$$ \Psi(z) = \frac{1}{\Gamma(z)}\,\, \prod_{p \textrm{ prime}} \left( 1 + \frac{z-1}{p} \right) \E^{-\frac{z-1}{p}},$$
and where the remainder $O(\frac{1}{\log n})$ is uniform for $z$ in a compact subset of $\C$. Therefore, one has mod-Poisson convergence at speed $\lambda_{n}=\log\log n$, and with limiting function $\psi(\xi)=\Psi(\E^{\I \xi})$. This limiting function involves two factors: the limiting function $\frac{1}{\Gamma(\E^{\I \xi})}$ of the number of cycles of a random permutation (``geometric'' factor), and an additional ``arithmetic'' factor $\prod_{p \in \P} ( 1 + \frac{\E^{\I \xi}-1}{p})\, \exp(-\frac{\E^{\I\xi}-1}{p})$. This apparition of two limiting factors is a common phenomenon in number theory \cite{KN10,JKN11}.
\end{example}
\medskip

Let $(X_n)_{n \in \N}$ be a sequence of random variables that converges mod-$\phi$ with parameters $\l_n$. Informally, $\psi_n(\xi) = \esper[\E^{\I \scal{\xi}{ X_n}}]\,\E^{-\l_n\phi(\xi)}$ measures a deconvolution residue, and mod-$\phi$ convergence means that this residue stabilizes, allowing the computation of equivalents of $d(\mu_{n},\nu_n)$, where
\begin{align*}
\mu_{n}&=\text{law of }X_{n};\\
\nu_{n}&=\text{law of the infinitely divisible law with exponent }\lambda_{n}\phi
\end{align*}
and $d$ is one of the distances introduced in the previous paragraph. This setting will be called the \emph{basic approximation scheme for a mod-$\phi$ convergent sequence}.
\bigskip

\subsection{The Wiener algebra and the general scheme of approximation}
As explained before, we shall be interested in more general approximation schemes, with signed measures $\nu_n$ that are closer to $\mu_n$ than the basic scheme. When $d=1$, the apparition of signed measures is natural in the setting of the Wiener algebra of absolutely convergent Fourier series (see \cite{Kah70,Kat04}):
 \begin{definition}
The Wiener algebra $\Ac = \Ac(\T)$ is the algebra of continuous functions on the circle whose Fourier series converges absolutely. It is a Banach algebra for the pointwise product and the norm $$ \normA{f} := \sum_{n \in \Z} |c_n(f)|,$$
where $c_{n}(f)$ denotes the Fourier coefficient $\int_{0}^{2\pi}f(\E^{\I \theta})\,\E^{-\I n \theta}\,\frac{d\theta}{2\pi}$.
\end{definition}
Note that the characteristic function of any signed measure $\mu$ on $\Z$ belongs to the Wiener algebra, with
$$ \normA{\widehat{\mu}}=\sum_{n \in \Z} |c_{n}(\widehat{\mu})|=\sum_{n \in \Z} |\mu(n)| = \|\mu\|_{\mathrm{TV}}$$
equal to the total variation norm of the measure. Thus, $\Ac$ is the right functional space in order to use harmonic analysis tools when dealing with (signed) measures. To the best of our knowledge, this interpretation of the total variation distance as the norm of $\Ac$ has not been used before. On the other hand, another important property of the Wiener algebra is:
\begin{proposition}[Wiener's $\frac{1}{f}$ theorem]\label{prop:wienerinverse}
Let $f \in \Ac$. Then $f$ never vanishes on the circle if and only if $\frac{1}{f} \in \Ac$.
\end{proposition}
\noindent We refer to \cite{New75} for a short proof of the Wiener theorem. As a consequence, if $X_n$ is a $\Z$-valued random variable and if $\nu$ is an infinitely divisible distribution with exponent $\phi$, then the deconvolution residue
$$\psi_n(\xi) = \esper[\E^{\I \xi X_n}]\,\E^{-\lambda_n \phi(\xi)}$$
always belongs to $\Ac$, so it is a convergent Fourier series $\psi_n(\xi)=\sum_{k=-\infty}^{\infty} a_{k,n}\,\E^{\I k \xi}$. The idea is then to replace this residue $\psi_n$ by a simpler residue $\chi_n\in \Ac$, which after reconvolution by $\E^{\lambda_n \phi(\xi)}$ yields a signed measure approximating the law of $X_n$. We are thus lead to:

\begin{definition}\label{def:approxscheme}
Let $(X_n)_{n \in \N}$ be a sequence of random variables in $\Z^d$ that is mod-$\phi$ convergent with parameters $(\lambda_n)_{n \in\N}$. A \emph{general approximation scheme} for $(X_n)_{n \in \N}$ is given by a sequence of discrete signed measures $(\nu_n)_{n \in \N}$ on $\Z^d$, such that 
\begin{align*}
\widehat{\mu_n}(\xi) &= \esper[\E^{\I\scal{\xi}{X_n}}] = \E^{\lambda_n\phi(\xi)}\,\psi_n(\xi);\\
\widehat{\nu_n}(\xi) &= \E^{\lambda_n\phi(\xi)}\,\chi_n(\xi),
\end{align*}
with $\lim_{n \to +\infty}\psi_n(\xi) = \psi(\xi)$ and $\lim_{n \to +\infty} \chi_n(\xi)=\chi(\xi)$. Here, the residues $\psi_n$, $\psi$, $\chi_n$ and $\chi$ are functions on the torus $\T^d = (\R/2\pi \Z)^d$ that have absolutely convergent Fourier series, and with $\chi_n(0)=\chi(0)=1$ (hence, $\nu_n(\Z^d)=1$). The convergence of the residues $\psi_n \to \psi$ and $\chi_n \to \chi$ is assumed to be at least uniform on the torus, that is in the space of continuous functions $\mathscr{C}^0(\T)$.
\end{definition}\medskip

The basic approximation scheme is the case when $\chi_n(\xi)=\chi(\xi)=1$. The residues $\chi_n$ and $\chi$ will typically be trigonometric polynomials, and they will enable us to enhance considerably the quality of our discrete approximations. \medskip

\begin{example}\label{ex:approximationorderr}
An important case of approximation scheme in the sense of Definition \ref{def:approxscheme} is the \emph{approximation scheme of order $r\geq 1$}. Suppose that $d=1$ and that the functions $\psi_n(\xi)$ can be represented on the torus as absolutely convergent series
$$
\psi_n(\xi) = 1+\sum_{k=1}^\infty b_{k,n}\,(\E^{\I\xi}-1)^k + \sum_{k=1}^n c_{k,n}\,(\E^{-\I\xi}-1)^k.
$$
The coefficients $a_{k,n}$ of $\psi_n(\xi) = \sum_{k=-\infty}^\infty a_{k,n}\,\E^{\I k\xi}$ are related to the coefficients $b_{k,n}$ and $c_{k,n}$ by the equations
\begin{align*}
b_{k,n} = \sum_{l \geq k} \binom{l}{k}\,a_{l,n}\quad;\quad c_{k,n} = \sum_{l \geq k} \binom{l}{k}\,a_{-l,n} 
\end{align*}
for all $k\geq 1$. Set then
$$\chi^{(r)}_n(\xi)=P^{(r)}_n(\E^{\I \xi})=1+\sum_{k=1}^{r} b_{k,n}\,(\E^{\I\xi}-1)^k + \sum_{k=1}^{r} c_{k,n}\,(\E^{-\I\xi}-1)^k.$$
The $\chi_n^{(r)}(\xi)$ are the Laurent polynomials of degree $r$ that approximate the residues $\psi_n$ around $0$ at order $r\geq 1$. Then, if $X_n$ is a random variable with law $\mu_n$ with characteristic function $\widehat{\mu_n}(\xi)=\E^{\l_n\phi(\xi)} \,\psi_n(\xi)$, the approximation scheme of order $r$ of $\mu_n$ is given by the signed measures $\nu_n^{(r)}$, with 
$$\widehat{\nu_n^{(r)}}(\xi)=\E^{\l_n\phi(\xi)}\,\chi_n^{(r)}(\xi).$$
We shall prove in Section \ref{sec:asymptoticdistance} that $\dloc(\mu_n,\nu_n^{(r)})$, $\dkol(\mu_n,\nu_n^{(r)})$ and $\dtv(\mu_{n},\nu_{n}^{(r)})$ get smaller when $r$ increases; in particular, $\nu_n^{(r\geq 1)}$ is asymptotically a better approximation of $\mu_n$ than the basic scheme $\nu_n^{(0)}$.
\end{example}\medskip

One can give a functional interpretation to the approximation schemes of order $r \geq 1$. Denote as before $P^{(r)}_n$ the Laurent approximation of order $r$ of $\psi_n$, and introduce the shift operator $S$ on functions $f : \Z \to \C$, defined by 
$$ (Sf)(k) = f(k+1).$$
 If $\nu_n^{(r)}$ is the signed measure with Fourier transform $\widehat{\nu_n^{(r)}}(\xi) = \E^{\l_n\phi(\xi)}\,P^{(r)}_n(\E^{\I \xi})$, then for any square-integrable function $f : \Z \to \C$, if $\widehat{f}(\xi) = \sum_{k \in \Z} f(k)\,\E^{\I k\xi}$, then
\begin{align*}
\nu_n^{(r)}(f) &= \sum_{k \in \Z} \nu_n^{(r)}(\{k\})\,f(k) = \frac{1}{2\pi}\int_{0}^{2\pi}  \widehat{\nu_{n}^{(r)}}(\xi) \overline{\widehat{f}(\xi)} \,d\xi \\
& = \frac{1}{2\pi}\int_{0}^{2\pi}  \widehat{\nu_{n}^{(0)}}(\xi)\, P^{(r)}_n(\E^{\I \xi})\overline{\widehat{f}(\xi)} \,d\xi = \frac{1}{2\pi}\int_{0}^{2\pi}  \widehat{\nu_{n}^{(0)}}(\xi)\, \overline{\widehat{(P^{(r)}_n(S)f)}(\xi)} \,d\xi \\
&= \nu_n^{(0)}(P^{(r)}_n(S)f).
\end{align*}
Moreover, the operator $P^{(r)}_n(S)$ is a linear combination of discrete difference operators:
$$
P_n^{(r)}(S)=\id+\sum_{k=1}^r b_{k,n}\,(\Delta_{+})^k + \sum_{k=1}^r c_{k,n}\,(\Delta_{-})^k$$
with $(\Delta_{+}^k(f))(j)=\sum_{l=0}^k (-1)^{k-l} \binom{k}{l} f(j+l)$ and $(\Delta_{-}^k(f))(j)=\sum_{l=0}^k (-1)^{k-l} \binom{k}{l} f(j-l)$. Therefore:
\begin{proposition}
In the previous setting, if $(\nu_n^{(r)})_{n \in \N}$ is the approximation scheme of order $r\geq 1$ of a sequence of probability measures $(\mu_n)_{n \in \N}$ that is mod-$\phi$ convergent, then for any square-integrable function $f$, $$\nu_n^{(r)}(f) = \esper[(P^{(r)}_n(S)(f))(Y_n)],$$
where $Y_n$ follows an infinitely divisible law with exponent $\l_n\phi$. Thus, to estimate at order $r$ the expectation $\mu_n(f)=\esper[f(X_n)]$:\vspace{2mm}
\begin{enumerate}
\item one replaces $X_n$ by an infinitely divisible random variable $Y_n$;
\vspace{2mm}
\item and one also replaces the function $f$ by $P^{(r)}_n(S)(f)$, which is better suited for discrete approximations.
\end{enumerate}
\end{proposition}

\begin{remark}
Our notion of approximation scheme should be compared to the one of \cite{Hwa99}, which is the particular case where the reference infinitely divisible law is Poissonian, and the residues $\chi_n$ are constant equal to one. One of the main interest of our approach is that we are able to construct better schemes of approximation, by allowing quite general residues $\chi_n$ in the Fourier transform of the laws $\nu_n$ that approximate the random variables $X_n$. On the other hand, we consider only three distances among those studied in \cite{Hwa99}, but there should be no difficulty in adapting our results to the other distances, such as the Wasserstein metric and the Hellinger/Matusita metric.
\end{remark}

\begin{remark}
The approximation scheme of order $r \geq 1$ can be considered as a discrete analogue of the Edgeworth expansion in the central limit theorem, which is with respect to the Gaussian approximation, and thus gives results for a different scale of fluctuations \cite[Chapter VI]{Petrov75}.
\end{remark}
\bigskip

\subsection{Outline of the paper}
Placing ourselves in the setting of a general approximation scheme, one basic idea in order to evaluate the distances between the distributions $\mu_n$ and $\nu_n$ is to relate them to the distances between the two residues $\psi_n(\xi)$ and $\chi_n(\xi)$ in the Wiener algebra $\Ac(\T^d)$. In Section \ref{sec:wiener}, we prove various concentration inequalities in this algebra, and we explain how to use them in order to compute distances between distributions. In Section \ref{sec:asymptoticdistance}, we apply these results to obtain asymptotic estimates for the distances in the general approximation scheme (see our main Theorems \ref{thm:mainloc}, \ref{thm:mainkol} and \ref{thm:maindtv}). In these two sections, we shall restrict ourselves to the one-dimensional setting, postponing the more involved computations of the higher dimensions $d\geq 2$ to Section \ref{sec:multidimensional}. \bigskip

\noindent In Section \ref{sec:oneexample}, we apply our theorem to various one-dimensional examples of mod-Poisson convergent sequences:\vspace{2mm}
 \begin{itemize}
 	\item In Section \ref{subsec:poissonbernoulli}, we explain how to use the general theory with the toy-model of sums of independent Bernoulli random variables (classical Poisson approximation). The combinatorics of the approximation schemes of order $r \geq 1$ can be encoded in the algebra of symmetric functions, and we explain this encoding in \S\ref{subsec:alphabet}, by using the theory of formal alphabets. The formal alphabets provide a simple description of the approximation schemes for all the examples hereafter, although none of them (except the toy-model) come from independent Bernoulli variables.\vspace{2mm}
 	\item In Section \ref{subsec:disjointcycles}, we study the number of disjoint cycles in a model of random permutations which generalises the uniform and the Ewens measure (Example \ref{ex:sigman}). The mod-Poisson convergence of this model was proven in \cite{NZ13}, and we compute here the approximation schemes of this model.\vspace{2mm}
 	\item In Section \ref{subsec:combinatorialexamples}, we show more generally how to use \emph{generating series with algebraico-logarithmic singularities} to construct general approximation schemes of statistics of random combinatorial objects. We study with this method the number of irreducible factors of a random monic polynomial over the finite field $\mathbb{F}_q$ (counted with or without multiplicity), and the number of connected components of a random functional graph.\vspace{2mm}
 	\item In \S\ref{subsec:erdoskac}, we consider the number of distinct prime divisors of a random integer (Example \ref{ex:omegan}). We explain how to use the form of the residue $\psi(\xi)$ of mod-Poisson convergence to construct explicit approximation schemes of the corresponding probability measures, using again the formalism of symmetric functions.\vspace{2mm}
 \end{itemize}
 Our approach also allows us to measure the gain of the Poissonian approximation in comparison to the Gaussian approximation, when one has a sequence of random integers that behaves asymptotically like a Poisson random variable with a large parameter $\l_n$.
\bigskip

In Section \ref{sec:multidimensional}, we extend our results to the multi-dimensional setting. One of the advantages of the Fourier approach to approximation of probability measures is that it allows one to deal with higher dimensions  with the exact same techniques, although the computations are more involved. An important difference with the one-dimensional setting is that the dependence between the coordinates of a mod-$\phi$ convergent sequence $(X_n)_{n \in \N}$ in $\Z^d$ can be read on the asymptotics of the distances from the approximation schemes to the laws of the random variables. Note that this happens even when the reference infinitely divisible distribution corresponds to independent coordinates. We provide two examples of this phenomenon, stemming respectively from the combinatorics of the wreath products $\sym_n\wr(\Z/d\Z) $ and from the number theory of residue classes of prime numbers.
\bigskip
\bigskip

\section{Concentration inequalities in the Wiener algebra}\label{sec:wiener}

For convenience, until Section \ref{sec:multidimensional}, we shall focus on the one-dimensional case ($d=1$) and the corresponding torus  $\T = \R / 2 \pi \Z$, which we view as the set of complex numbers of modulus $1$. Thus, a function on $\T$ will be a function of $\E^{\I\xi}$ with $\xi \in [0,2\pi)$. For $p \in [1,+\infty)$, the space of complex-valued functions on $\T$ whose $p$-powers are Lebesgue integrable will be denoted $\leb^{p}=\leb^{p}(\T)$; it is a Banach space for the norm 
$$\|f\|_{p}:=\left(\int_{0}^{2\pi} |f(\E^{\I \xi})|^{p}\,\frac{dt}{2\pi}\right)^{\frac{1}{p}}.$$
For $p=+\infty$, $\leb^{\infty}(\T)$ is the space of essentially bounded functions on the torus, with Banach space norm
$$\|f\|_{\infty}:=\text{ess-sup}_{\xi \in [0,2\pi]}\left|f(\E^{\I \xi})\right|.$$
In the sequel, we abbreviate sometimes the Haar integral $\int_{0}^{2\pi}f(\E^{\I\xi})\,\frac{d\xi}{2\pi}$ by $\int_{\T}f(\E^{\I\xi})$, or simply $\int_{\T}f$.
\bigskip

\subsection{Deconvolution residues in the Wiener algebra}
Recall that the Wiener algebra $\Ac(\T)$ is the complex algebra of absolutely convergent Fourier series, endowed with the norm $\|f\|_\Ac = \sum_{n \in \Z}|c_n(f)| $. We shall use the following property of $\Ac(\T)$, which is akin to Poincar\'e's inequality for Sobolev spaces (see \cite[\S6.2]{Kat04}):

\begin{proposition}
\label{prop:H_in_A}
There is a constant $C_H = \frac{\pi}{\sqrt{3}} = 1.814\ldots$ such that
$$ \normA{f} \leq |c_0(f)| + C_H \|f'\|_{\leb^2}$$
for all $f \in \Ac$.
\end{proposition}
\begin{proof} Combining the Cauchy--Schwarz inequality and the Parseval identity, we obtain
\begin{align*}
\normA{f} &= |c_0(f)| + \sum_{n \neq 0} \frac{1}{n}\, |n\,c_n(f)|\\ 
& \leq |c_0(f)| + \sqrt{ \sum_{n \neq 0} \frac{1}{n^2} } \,\sqrt{ \sum_{n \neq 0} |n \,c_n(f)|^2 }\\
& \leq |c_0(f)| + \sqrt{ \frac{\pi^2}{3} }\, \|f'\|_{\leb^2}.
\end{align*}
Note that the inequality is sharp, and it implies that the Sobolev space $\mathscr{W}^{1,2}(\T)$ of $\leb^2$ functions on $\T$ with weak derivative in $\leb^2$ is topologically included in the Wiener algebra.
\end{proof}\medskip

Another important tool for the computation of total variation distances is the following Lemma \ref{lem:fundamental_ineq}. Call \emph{deconvolution residue} of a signed measure $\mu$ on $\Z$ by another signed measure $\nu$ on $\Z$ the function
$$ \delta(\xi) := \widehat{\mu}(\xi)\, (\widehat{\nu}(\xi))^{-1} .$$ 
By Wiener's theorem, if $\widehat{\nu}$ never vanishes, which is for instance the case when it is the Fourier transform of an infinitely divisible law $\nu$ \cite[Lemma 7.5]{Sato99}, then $\widehat{\nu}^{-1} \in \Ac$ and the previous deconvolution happens in the Wiener algebra: $\delta \in \Ac$.
Now, in order to measure the distance between the law $\mu$ of $X$ and a reference law $\nu$, one can adopt the following point of view, which is common in signal processing. For the deconvolution residue to be considered as a small noise, $\delta$ has to be close to the constant function $1$ (the Fourier transform of the Dirac distribution at zero). In particular, suppose that we are given a general scheme of approximation of a sequence of random variables $(X_n)_{n \in \N}$ by a sequence of laws $(\nu_n)_{n \in\N}$. Then, with the notation of Definition \ref{def:approxscheme}, one has a sequence of deconvolution residues
$$\delta_n(\xi) = \frac{\widehat{\mu_n}(\xi)}{\widehat{\nu_n}(\xi)} ,$$
and the quality of our scheme of approximation will be related to the speed of convergence of $(\delta_n)_{n \in \N}$ towards the constant function $1$. More precisely:

\begin{lemma}[Fundamental inequality for deconvolution residues]
\label{lem:fundamental_ineq}
Let $\mu$ and $\nu$ be two signed measures on $\Z$ and $\delta = \frac{\widehat{\mu}}{\widehat{\nu}}$ be the deconvolution residue (we assume that $\widehat{\nu}$ does not vanish on $\T$).  For any $c \in \Ac$,
$$   \dtv(\mu, \nu)
\leq \normA{ c (\delta - 1) \widehat{\nu} } + \normA{ \delta - 1 } \normA{ (1-c) \widehat{\nu} }.
$$
\end{lemma}
\begin{proof}
\begin{align*}
 \dtv(\mu, \nu) & = \normA{ \widehat{\mu} - \widehat{\nu} } = \normA{ (\delta - 1) \widehat{\nu} }\\
& \leq \normA{ c (\delta - 1) \widehat{\nu} } + \normA{ (1-c) (\delta - 1) \widehat{\nu} }\\
& \leq \normA{ c (\delta - 1) \widehat{\nu} } + \normA{ \delta - 1 } \normA{ (1-c) \widehat{\nu} }.\qedhere
\end{align*} 
\end{proof}
\medskip

In practice, $c \in \Ac$ will be a carefully chosen cut function concentrated on regions where $\widehat{\nu}$ is large. So, informally, $\nu$ is a good approximation of $\mu$ if the deconvolution residue $\delta$: \vspace{2mm}
\begin{itemize} 
\item is close to $1$ where $\widehat{\nu}$ is large; \vspace{2mm}
\item has a reasonable norm in the Wiener algebra.\vspace{2mm} 
\end{itemize}
Note that the fundamental inequality {\it does} rely on the algebra norm of $\Ac(\T)$. We now quantify this idea.\bigskip

\subsection{Estimates of norms in the Wiener algebra}
In this paragraph, we fix two discrete measures $\mu$ and $\nu$ on $\Z$, both being convolutions of an infinitely divisible law with exponent $\l\phi$ by residues $\psi$ and $\chi$:
\begin{align*} \widehat{\mu}(\xi) &= \E^{\lambda \phi(\xi)}\,\psi(\xi); \\
 \widehat{\nu}(\xi) &= \E^{\lambda \phi(\xi)}\,\chi(\xi).
\end{align*}
In Section \ref{sec:asymptoticdistance}, we shall recover the setting of a general approximation scheme by adding indices $n$ to this situation. We denote $m$ and $\sigma^2$ the first two coefficients of the Taylor expansion of $\phi$ around $0$:
$$ \phi(\xi) = m\I\xi - \frac{\sigma^2 \xi^2}{2} + o(\xi^2).$$
We then fix an integer $r \geq 0$ such that:\vspace{2mm}
\begin{enumerate}
 	\item the residues $\psi$ and $\chi$ are assumed to be $(r+1)$ times continuously differentiable on $[-\eps, \eps]$ for a certain $\eps >0$;\vspace{2mm}
 	\item their Taylor expansion coincides at $0$ up to the $r$-th order:
$$ \forall s \in \lle 0,r\rre,\,\,\, \left( \psi - \chi \right)^{(s)}(0) = 0.$$	
 \end{enumerate} 
In this setting, we define the non-negative quantities
\begin{align*}
 \beta_{r+1}(\eps) &= \sup_{\xi \in [-\eps, \eps] } \left| (\psi-\chi)^{(r+1)}(\xi) \right|;\\
 \gamma(\eps) &= \sup_{\xi \in [-\eps, \eps] } \left| \phi''(\xi) + \sigma^2   \right|;\\
 M &=-\sup_{\xi \in [-\pi,\pi]} \left(\frac{\Re(\phi(\xi))}{\xi^2} \right).
 \end{align*}
The strict positivity of $M$ as soon as $\phi$ is not the constant distribution concentrated at $0$ is proven in \cite[Section 3]{FMN16}. Note that for any $\xi \in [-\eps,\eps]$, one has
\begin{align*}
\Re(\phi(\xi)) &= \int_{\theta=0}^{\xi} (\xi-\theta)\,\Re(\phi''(\theta))\,d\theta =- \frac{\sigma^2 \xi^2}{2} + \int_{\theta=0}^{\xi} (\xi-\theta)\,\Re(\phi''(\theta)+\sigma^2)\,d\theta  \\
 &\leq \frac{(\gamma(\eps)-\sigma^2)\,\xi^2}{2}.
\end{align*}
On the other hand, outside the interval $[-\eps,\eps]$, one can use the inequality $\Re(\phi(\xi)) \leq -M\xi^2$.
\begin{theorem}[Norm estimate]
\label{thm:concentrationinequality}
Take $\eps$ small enough so that $ \gamma(\eps)   \leq \frac{\sigma^2}{2} $, and suppose that $\lambda$ is large enough so that $\lambda \geq \frac{2}{\sigma^{2}} $. Then:
\begin{align*}
\normA{ \widehat{\mu} - \widehat{\nu} } &\leq \normA{ \psi - \chi } \left( 1 + C_H \left( \sqrt{\frac{2}{\pi\eps}}  + \lambda \|\phi'\|_\infty \right) \right) \E^{-\frac{\lambda M \eps^2}{4}} \\
&\quad + \frac{C_{r+1}\,\beta_{r+1}(\eps)\,(\eps^{-1}+\sqrt{5(r+1)})}{(\frac{\sigma^{2}}{2}\,\lambda)^{\frac{r}{2}+\frac{1}{4}} }, 
\end{align*}
where $C_{r+1}$ is a constant depending only on $r$:
$$ C_{r+1} =  \frac{1}{(r+1)!}\,\sqrt{\frac{2\pi}{3}\,\Gamma\!\left(r+\frac{3}{2}\right)}.$$
\end{theorem}

In order to prove this theorem, we introduce the cut function $\xi \in \T \mapsto c(\xi)$ that is piecewise linear, vanishes outside of $[-\eps, \eps]$, and is equal to $1$ on $[-\frac{\eps}{2},\frac{\eps}{2}]$, see Figure \ref{fig:cutfunction}.
\begin{center}
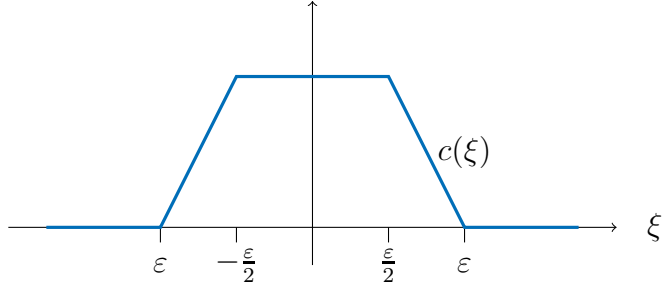
\begin{figure}[ht]

\begin{tikzpicture}[scale=1]
\draw [->] (-4,0) -- (4,0);
\draw [->] (0,-0.5) -- (0,3);
\draw (-1,-0.2) -- (-1,0);
\draw (-1,-0.5) node {$-\frac{\eps}{2}$};
\draw (1,-0.5) node {$\frac{\eps}{2}$};
\draw (-2,-0.5) node {$\eps$};
\draw (2,-0.5) node {$\eps$};
\draw (4.5,0) node {$\xi$};
\draw (2,1) node {$c(\xi)$};
\draw (1,-0.2) -- (1,0);
\draw (-2,-0.2) -- (-2,0);
\draw (2,-0.2) -- (2,0);
\draw [very thick,NavyBlue] (-3.5,0) -- (-2,0) -- (-1,2) -- (1,2) -- (2,0) -- (3.5,0);
\end{tikzpicture}

\caption{The cut-function $c(\xi)$.\label{fig:cutfunction}}
\end{figure}
\end{center}

\begin{lemma}\label{lem:concentration1}
Under the assumptions of Theorem \ref{thm:concentrationinequality},
$$
\normA{ (1-c) \,\E^{\lambda \phi} } \leq \left( 1 + C_H\left(  \sqrt{\frac{2}{\pi\eps}} + \lambda \|\phi'\|_\infty \right ) \right) \E^{-\frac{\lambda M \eps^2}{4}}.$$
\end{lemma}

\begin{proof}
Note that $|(1-c(\xi))\,\E^{\lambda\phi(\xi)}|$ vanishes on $[-\frac{\eps}{2},\frac{\eps}{2}]$, and is smaller than $\E^{-\lambda M\xi^{2}}\leq \E^{-\frac{\lambda M \eps^2}{4}}$ everywhere else. We then use Proposition \ref{prop:H_in_A}:
\begin{align*}
\normA{ (1-c)\, \E^{\lambda \phi} } 
& \leq \int_\T |1-c|\, \E^{\lambda \Re(\phi)} 
+ C_H \left\| -c' \E^{\lambda \phi} + (1-c) \,\lambda \phi'\, \E^{\lambda \phi} \right\|_{\leb^2}\\
& \leq  \E^{-\frac{\lambda M \eps^2}{4}}  + C_H \,\left\| c' \E^{\lambda \phi} \right\|_{\leb^2}
     + C_H \,\lambda \|\phi'\|_\infty \, \left\| (1-c)\, \E^{\lambda \phi} \right\|_{\leb^2}\\
 &\leq  \E^{-\frac{\lambda M \eps^2}{4}}
     + C_H \,\|c'\|_{\leb^2}\,\E^{-\frac{\lambda M \eps^2}{4}}
     + C_H\, \lambda \|\phi'\|_\infty \,\E^{-\frac{\lambda M \eps^2}{4}}\\
&\leq \left( 1 + C_H\left( \sqrt{\frac{2}{\pi\eps}}  + \lambda \|\phi'\|_\infty \right ) \right) \E^{-\frac{\lambda M \eps^2}{4}}.\qedhere
\end{align*}
\end{proof}

\begin{lemma}\label{lem:concentration2}
Under the assumptions of Theorem \ref{thm:concentrationinequality}, we have:
$$\normA{ c \left( \psi - \chi \right) \E^{\lambda \phi} } \leq \frac{C_{r+1}\,\beta_{r+1}(\eps)\,(\eps^{-1}+\sqrt{5(r+1)})}{(\frac{\sigma^{2}}{2}\,\lambda)^{\frac{r}{2}+\frac{1}{4}} }.
$$
\end{lemma}

\begin{proof}
Note that the norm $\normA{\cdot}$ is invariant after multiplication by $\E^{\I k\xi}, k \in \Z$. In particular,
$$ \normA{ c \left( \psi - \chi \right) \E^{\lambda \phi} } = \normA{ c(\xi) \left( \psi(\xi) - \chi(\xi) \right) \E^{\lambda \phi(\xi) - \I \xi \lfloor \lambda m \rfloor} },$$
where $\lfloor \cdot \rfloor$ denotes the integer part of a real number. Using again Proposition \ref{prop:H_in_A}, we obtain
\begin{align*}
\normA{ c \left( \psi - \chi \right) \E^{\lambda \phi} } & \leq  \int_\T |c \,(\psi-\chi) \,\E^{\lambda \phi}|
+ C_H \left\| \left( c(\xi)\, (\psi(\xi)-\chi(\xi))\, \E^{\lambda \phi(\xi)-\I \xi \lfloor \lambda m \rfloor } \right)' \right\|_{\leb^2}\\
&\leq \int_\T |c\, (\psi-\chi)\, \E^{\lambda \phi}|  + C_H \left\|c' \,(\psi-\chi)\, \E^{\lambda \phi} \right\|_{\leb^2}\\
&\quad + C_H \left\|c\, (\psi-\chi)' \,\E^{\lambda \phi} \right\|_{\leb^2}
+ C_H  \left\|c \,(\psi-\chi)\, (\lambda \phi' - i\lfloor \lambda m \rfloor) \,\E^{\lambda \phi} \right\|_{\leb^2}\\[3mm]
&\leq A + B + C + D.
\end{align*}
Let us bound separately each term. Recall that 
$ \int_{-\infty}^\infty |t|^r \E^{-t^2}\, dt= \int_{0}^{\infty}u^{\frac{r-1}{2}}\,\E^{-u}\,du= \Gamma\!\left( \frac{r+1}{2} \right)$
for $r \geq 0$.\vspace{2mm}

\begin{enumerate}[(A)]
\item For $\xi \in [-\eps,\eps]$, 
\begin{align*}
|\psi(\xi) - \chi(\xi)| &\leq \frac{\beta_{r+1}(\eps)\, |\xi|^{r+1}}{(r+1)!}; \\
\E^{\lambda\Re(\phi(\xi))} &\leq \E^{-\frac{\lambda(\sigma^{2}-\gamma(\delta))\,\xi^2}{2}} \leq \E^{-\frac{\lambda\sigma^{2}\,\xi^{2}}{4}}
\end{align*}
the last inequality following from the assumptions on $\eps$. Consequently,
\begin{align*}
A &\leq \frac{\beta_{r+1}(\eps)}{2\pi \,(r+1)!}\,\int_{-\eps}^{\eps} \E^{-\frac{\lambda\sigma^{2}\,\xi^{2}}{4}}\,|\xi|^{r+1}\,dt = \frac{\beta_{r+1}(\eps)}{2\pi \,(r+1)!\,(\lambda\,\frac{\sigma^{2}}{4})^{\frac{r}{2}+1}}\,\,\Gamma\!\left( \frac{r}{2}+1 \right)\\
&\leq \frac{\beta_{r+1}(\eps)}{2\pi \,(r+1)!\,(\lambda\,\frac{\sigma^{2}}{4})^{\frac{r}{2}+1}}\,\,\sqrt{\Gamma\!\left( \frac{r}{2}+\frac{3}{4} \right)\Gamma\!\left( \frac{r}{2}+\frac{5}{4} \right)}\\
&\leq \frac{\beta_{r+1}(\eps)}{(r+1)!\,(\lambda\,\frac{\sigma^{2}}{2})^{\frac{r}{2}+1}}\,\left(\frac{2}{\pi}\right)^{\!\frac{1}{4}}\,\sqrt{\frac{1}{2\pi}\,\Gamma\!\left( r+\frac{3}{2} \right)}.
\end{align*}
by using the log-convexity of the Gamma function, and the duplication formula $\Gamma(2z)=\frac{2^{2z-1}}{\sqrt{\pi}}\,\Gamma(z)\,\Gamma(z+\half)$.
\vspace{5mm}

\item Since $|c'(\xi)|\leq \frac{2}{\eps}$ on $[-\eps,\eps]$ and vanishes outside $[-\eps,\eps]$, one obtains 
\begin{align*}
B &\leq \frac{2\,C_{H}\,\beta_{r+1}(\eps)}{\eps\,(r+1)!}\, \left(\frac{1}{2\pi}\int_{-\infty}^{\infty} \E^{-\frac{\lambda\sigma^{2}\xi^{2}}{2}}\,|\xi|^{2r+2} \,d\xi\right)^{\half}\\
&\leq \frac{\beta_{r+1}(\eps)}{(r+1)!\,(\lambda\,\frac{\sigma^{2}}{2})^{\frac{r}{2}+\frac{3}{4}}}\,\frac{2\,C_{H}}{\eps}\,\sqrt{\frac{1}{2\pi}\,\Gamma\!\left(r+\frac{3}{2}\right)}.
\end{align*}
\vspace{2mm}

\item For $\xi \in [-\eps,\eps]$,
$$
|(\psi - \chi)'(\xi)| \leq \frac{\beta_{r+1}(\eps)\, |\xi|^{r}}{r!}, 
$$
so as before,
\begin{align*}
C &\leq \frac{C_{H}\,\beta_{r+1}(\eps)}{r!}\, \left(\frac{1}{2\pi}\int_{-\infty}^{\infty} \E^{-\frac{\lambda\sigma^{2}\xi^{2}}{2}}\,|\xi|^{2r}\,d\xi\right)^{\half}=\frac{C_{H}\,\beta_{r+1}(\eps)}{r!\,(\lambda\,\frac{\sigma^{2}}{2})^{\frac{r}{2}+\frac{1}{4}}}\,\sqrt{\frac{1}{2\pi}\,\Gamma\!\left(r+\half\right)}\\
&\leq \frac{\beta_{r+1}(\eps)}{(r+1)!\,(\lambda\,\frac{\sigma^{2}}{2})^{\frac{r}{2}+\frac{1}{4}}}\,\frac{C_{H}\,(r+1)}{\sqrt{r+\half}}\,\sqrt{\frac{1}{2\pi}\,\Gamma\!\left(r+\frac{3}{2}\right)}.
\end{align*}
\vspace{2mm}

\item Last, for $\xi \in [-\eps,\eps]$, $ |\phi'(\xi) - \I\mu| \leq \gamma(\delta)\, |\xi| \leq \frac{\sigma^{2}|\xi|}{2}$, and then, 
$$|(\lambda\phi'(\xi)- \I \lfloor \lambda \mu \rfloor)|\leq \frac{\lambda\sigma^{2}\,|\xi|}{2}+|\lambda\mu-\lfloor \lambda \mu \rfloor| \leq \frac{\lambda\sigma^{2}\,|\xi|}{2}+1.$$
Therefore,
\begin{align*}D &\leq \frac{C_H\,\beta_{r+1}(\eps)}{(r+1)!} \left(\left(\frac{1}{2\pi}\int_{-\infty}^{\infty}\E^{-\frac{\lambda\sigma^{2}\xi^{2}}{2}}\,|\xi|^{2r+2}\,d\xi \right)^{\frac{1}{2}}+\frac{\lambda\,\sigma^{2}}{2} \left(\frac{1}{2\pi}\int_{-\infty}^{\infty}\E^{-\frac{\lambda\sigma^{2}\xi^{2}}{2}}\,|\xi|^{2r+4}\,d\xi\right)^{\frac{1}{2}}\right)\\
&\leq \frac{C_H\,\beta_{r+1}(\eps)}{(r+1)!} \left(\frac{1}{(\lambda\,\frac{\sigma^{2}}{2})^{\frac{r}{2}+\frac{3}{4}}}\,\sqrt{\frac{1}{2\pi}\,\Gamma\!\left(r+\frac{3}{2}\right)}+\frac{1}{(\lambda\,\frac{\sigma^{2}}{2})^{\frac{r}{2}+\frac{1}{4}}}\,\sqrt{\frac{1}{2\pi}\,\Gamma\!\left(r+\frac{5}{2}\right)}\right).
\end{align*}\vspace{2mm}
\end{enumerate}

\noindent To conclude, notice that since $\lambda\,\frac{\sigma^{2}}{2}\geq 1$, one can take in each denominator the smallest power of this quantity, namely, $(\lambda\,\frac{\sigma^{2}}{2})^{\frac{r}{2}+\frac{1}{4}}$. One thus obtains:
\begin{align*}
A+B+C+D &\leq \frac{\beta_{r+1}(\eps)}{(r+1)! \,(\lambda\,\frac{\sigma^{2}}{2})^{\frac{r}{2}+\frac{1}{4}}}\,\sqrt{\frac{1}{2\pi}\,\Gamma\!\left(r+\frac{3}{2}\right)}\\
&\quad\times\left(\left(\frac{2}{\pi}\right)^{\frac{1}{4}}+C_H\left(\frac{2}{\eps}+\frac{r+1}{\sqrt{r+\frac{1}{2}}} + 1 + \sqrt{r+\frac{3}{2}}\right)\right)\\
&\leq  \frac{\beta_{r+1}(\eps)}{(r+1)! \,(\lambda\,\frac{\sigma^{2}}{2})^{\frac{r}{2}+\frac{1}{4}}}\,\sqrt{\frac{2\pi}{3}\,\Gamma\!\left(r+\frac{3}{2}\right)}\left(\frac{1}{\eps}+ \sqrt{5(r+1)}  \right)\\
\end{align*}
by using on the last line $C_H = \frac{\pi}{\sqrt{3}}$ and 
\begin{equation*}
2\sqrt{5(r+1)} \geq \frac{r+1}{\sqrt{r+\frac{1}{2}}} + \sqrt{r+\frac{3}{2}}+1+\frac{\sqrt{3}}{\pi}\left(\frac{2}{\pi}\right)^{\frac{1}{4}}.\qedhere
\end{equation*}
\end{proof}

\begin{proof}[Proof of Theorem \ref{thm:concentrationinequality}]
We have
$$\normA{\widehat{\mu}-\widehat{\nu}} \leq \normA{\psi-\chi}\,\normA{(1-c)\,\E^{\lambda\phi}} + \normA{c\,(\psi-\chi)\,\E^{\lambda\phi}},$$
and Lemmas \ref{lem:concentration1} and \ref{lem:concentration2} allow us to control these two terms.
\end{proof}
\bigskip

\noindent Note that if $\eps$ is fixed and $\lambda$ goes to infinity, then the dominant term in the norm inequality is the second one, because the first one decreases exponentially fast. In this case,
$$\dtv(\mu,\nu) = \normA{\widehat{\mu}-\widehat{\nu}} = O\!\left(\frac{1}{\lambda^{\frac{r}{2}+\frac{1}{4}}}\right).$$
The constant hidden in the $O(\cdot)$ depends only on $\phi$, $r$, $\|\psi-\chi\|_{\Ac}$ and the parameter $\beta_{r+1}(\eps)$ introduced before. In particular, if $\phi$ is fixed and we look at families of residues $(\psi_n)_{n \in \N}$ and $(\chi_n)_{n \in \N}$ such that the corresponding quantities $\|\psi_n-\chi_n\|_{\Ac}$ and $\beta_{r+1,n}(\eps)$ stay bounded, then one can take a uniform constant in the $O(\cdot)$.\medskip

\begin{remark}
One can broaden slightly the scope of Theorem \ref{thm:concentrationinequality}, if one notes that regarding the residues $\psi$ and $\chi$, one only used the bounds
\begin{align*}
|\psi(\xi)-\chi(\xi)| &\leq \frac{\beta_{r+1}(\eps)\,|\xi|^{r+1}}{(r+1)!} \\
\text{and }\,\,|\psi'(\xi)-\chi'(\xi)| &\leq \frac{\beta_{r+1}(\eps)\,|\xi|^{r}}{r!}
\end{align*}
for any $\xi \in [-\eps,\eps]$. In particular, if one assumes that $\psi$ and $\chi$ belong to $\mathscr{C}^1(\T)$ and that
$$\psi'(\xi) - \chi'(\xi) = \frac{\beta_{r+1}\,\xi^{r}\,(1+o_{\xi}(1))}{r!}$$
in a neighborhood $[-\eps,\eps]$ of $0$, then the previous bounds are satisfied and the conclusions of Theorem \ref{thm:concentrationinequality} hold. We shall use this important remark in \S\ref{subsec:totalvariationdistance}.
\end{remark}

\bigskip
\bigskip

\section{Asymptotics of distances for a general scheme of approximation}\label{sec:asymptoticdistance}

We now consider a general scheme of approximation of a sequence of $\Z$-valued random variables $(X_n)_{n \in \N}$ by a sequence of laws $(\nu_n)_{n \in \N}$, with
\begin{align*}
\ \widehat{\mu_{n}}(\xi)&=\esper[\E^{\I\xi X_n}] = \E^{\lambda_n\phi(\xi)}\,\psi_n(\xi) \ ,\\
\ \widehat{\nu_{n}}(\xi)&= \E^{\lambda_n\phi(\xi)}\,\chi_n(\xi) \ ,
\end{align*}
$\lim_{n \to \infty} \psi_n(\xi)=\psi(\xi)$ and $\lim_{n \to \infty} \chi_n(\xi)=\chi(\xi)$ as in Definition \ref{def:approxscheme}.  We can immediately deduce from Theorem \ref{thm:concentrationinequality}:
\begin{proposition}
Consider a general scheme of approximation such that the convergences $\psi_n \to \psi$ and $\chi_n \to \chi$ occur in $\mathscr{C}^1(\T)$. Then, $\dtv(\mu_n,\nu_n) \to 0$ (and as well for $\dloc(\mu_n,\nu_n)$ and $\dkol(\mu_n,\nu_n)$).
\end{proposition}
\begin{proof}
Taking $r=0$ in Theorem \ref{thm:concentrationinequality}, we have 
$$\dtv(\mu_n,\nu_n) = \normA{\mu_n-\nu_n} = O\!\left(\frac{1}{(\lambda_n)^{\frac{1}{4}}}\right),$$
where the constant in the $O(\cdot)$ depends on $\|\psi_n-\chi_n\|_\Ac$ and $\|\psi_n-\chi_n\|_{\mathscr{C}^1}$. By Proposition \ref{prop:H_in_A}, both quantities are bounded by a function of $\|\psi_n-\chi_n\|_{\mathscr{C}^1}$, which is itself bounded since $\psi_n - \chi_n \to \psi - \chi$ in $\mathscr{C}^1(\T)$.\medskip

\noindent For the other distances, we use the inequality $\dloc \leq 2 \,\dkol \leq \dtv$.
\end{proof}\medskip

If we want to improve on the rate of convergence of the distances to $0$, then we need an additional assumption similar to the hypothesis of Theorem \ref{thm:concentrationinequality}, namely:
\begin{align*}
\forall n \in \N,\,\,\, &\psi_n(\xi)-\chi_n(\xi) = \beta_n\,(\I\xi)^{r+1}\,(1+o_{\xi}(1))\\
\text{and}\,\,\, &\psi(\xi)-\chi(\xi) = \beta\,(\I\xi)^{r+1}\,(1+o_{\xi}(1))\tag{H1}\label{eq:H1}
\end{align*}
with $\lim_{n \to \infty} \beta_n = \beta$, and the $o_\xi(1)$ that converges to $0$ as $\xi$ goes to $0$, uniformly in $n$. We denote this condition by \eqref{eq:H1}. It is satisfied if, for instance, the convergences $\psi_n \to \psi$ and $\chi_n \to \chi$ occur in $\mathscr{C}^{r+1}(\T)$, and if the scheme of approximation is the scheme of order $r$ as defined in Example \ref{ex:approximationorderr}. However, this new condition \eqref{eq:H1} is a bit more general, as we do not assume that the residues $\chi_n$ and $\chi$ are Laurent polynomials of degree $r$ in $\E^{\I\xi}$.
\bigskip

In \S\ref{subsec:localdistance}, \S\ref{subsec:kolmogorovdistance} and \S\ref{subsec:totalvariationdistance}, we shall consider the setting described above, and we shall establish some exact asymptotic formulas for $\dloc(\mu_n,\nu_n)$, $\dkol(\mu_n,\nu_n)$ and $\dtv(\mu_n,\nu_n)$. These formulas generally follow from an application of the Laplace method to an integral representation of the distance, but we shall also use Theorem \ref{thm:concentrationinequality} in order to get rid of certain non dominant terms in the asymptotics. In \S\ref{subsec:derived}, we introduce the notion of \emph{derived approximation schemes} which involve simpler residues, and under appropriate assumptions, we extend the results of \S\ref{subsec:localdistance}-\ref{subsec:totalvariationdistance} to these approximation schemes.
\bigskip

\subsection{The local distance}\label{subsec:localdistance}
Until the end of this section, $(X_n)_{n \in \N}$ denotes a sequence that is mod-$\phi$ convergent with parameters $(\l_n)_{n\in \N}$, and $(\nu_n)_{n \in \N}$ is a scheme of approximation of it, which satisfies the hypothesis \eqref{eq:H1}. The main tool in the computation of the local distance $\dloc(\mu_n,\nu_n)$ is:

\begin{proposition}\label{prop:localestimate}
The error term being uniform in $k \in \Z$, one has
$$\mu_n(\{k\}) - \nu_n(\{ k\}) =\frac{(-1)^{r+1}\,\beta}{\sqrt{2 \pi} \left( \sigma^2 \lambda_n \right)^{\frac{r}{2}+1}}\, \left.\frac{\partial^{r+1}}{\partial \alpha^{r+1}}\left( \E^{-\frac{\alpha^2}{2}}\right)\right|_{\alpha = \frac{k-\lambda_nm}{\sigma \sqrt{\lambda_n}}}
 + o\!\left( \frac{1}{(\lambda_n)^{\frac{r}{2}+1}}\right).$$
\end{proposition}

\begin{proof}
The computations hereafter are very similar to those of \cite[Section 3]{FMN16}, but they are performed at a different scale for $k$. We combine the Fourier inversion formula
$$ \mu_n(\{k\}) - \nu_n(\{ k\})  = \int_{\T} \left( \psi_n(\xi) - \chi_n(\xi) \right)\,\E^{\lambda_n \phi(\xi)} \, \E^{-\I k \xi}\,\frac{d\xi}{2\pi}$$
with the Laplace method. Note that \eqref{eq:H1} implies that $\psi_n-\chi_n$ is bounded on the whole torus $\T$ by some constant $C$. As a consequence, for any $\eps>0$, since $\Re(\phi(\xi)) \leq -M \xi^2$ with $M>0$, one has
\begin{align*}
\frac{1}{2 \pi} \left| \int_{(-\pi, \pi) \setminus (-\eps, \eps)} \!\left(\psi_n(\xi) - \chi_n(\xi) \right)\,\E^{\lambda_n \phi(\xi)} \, \E^{-\I k \xi}\,d\xi\right|
&\leq  C\, \exp\left( \lambda_n \,{\textstyle \sup_{\xi \in (-\pi, \pi) \setminus (-\delta, \delta)} \Re( \phi(\xi) ) }\right)\\
&\leq C\,\exp(-\lambda_n\,M\eps^2)\\
&= o\!\left( \frac{1}{(\lambda_n)^{\frac{r}{2}+1}}\right),
\end{align*}
this being uniform in $k$. Then, for $\eps>0$ small enough, one has the estimate:
\begin{align*}
&\int_{-\eps}^{\eps}  \left( \psi_n(\xi) - \chi_n(\xi) \right)\,\E^{\lambda_n \phi(\xi)} \, \E^{-\I k \xi}\,d\xi \\
&= \beta_n\, (1 + o_{\eps}(1)) \int_{-\eps}^{\eps} (\I \xi)^{r+1}\, \E^{\lambda_n \phi(\xi)-\I k \xi}\,d\xi\\
&= \beta\, (1 + o_{\eps}(1)) \int_{-\eps}^{\eps} (\I \xi)^{r+1}\,\E^{\lambda_n (\phi(\xi) - \I m \xi)}\, \E^{-\I (k-\lambda_n m)\xi}\,d\xi\\
&= \frac{\beta}{(\sigma^2 \lambda_n)^{\frac{r}{2}+1}}\,(1 + o_{\eps}(1))
   \int_{-\eps\,\sigma\sqrt{\lambda_n}}^{\eps\,\sigma \sqrt{\lambda_n}}(\I t)^{r+1}\,
 \E^{\lambda_n \left(\phi\!\left(\frac{t}{\sigma \sqrt{\lambda_n}}\right) - \frac{\phi'(0)\,t}{\sigma \sqrt{\lambda_n}}\right)}\,\E^{-\I t\frac{k-\lambda_n m}{\sigma \sqrt{\lambda_n}}}\,dt\\
&= \frac{\beta}{(\sigma^2 \lambda_n)^{\frac{r}{2}+1}} \,(1 + o_{\eps}(1))
    \int_{-\infty}^\infty (\I t)^{r+1}\,\E^{-\frac{t^{2}}{2}}   \,\E^{-i t\frac{k-\lambda_n m}{\sigma \sqrt{\lambda_n}}}\,dt.
\end{align*}
The last step uses the dominated convergence theorem. Indeed, we have the pointwise convergence:
$$ \lambda_n \left(\phi\!\left(\frac{t}{\sigma \sqrt{\lambda_n}}\right) - \frac{\phi'(0)\,t}{\sigma \sqrt{\lambda_n}}\right) \longrightarrow_{n \to \infty} -\frac{t^2}{2}$$
and if $\eps$ is small enough so that $\Re(\phi''(\xi))< -\frac{ \sigma^2}{2}$ for all $\xi \in [-\eps,\eps]$, then for every $t \in [-\eps\,\sigma \sqrt{\lambda_n}, \eps\,\sigma \sqrt{\lambda_n}]$:
$$
\left|\E^{\lambda_n\left(\phi\!\left(\frac{t}{\sigma \sqrt{\lambda_n}}\right) - \frac{\phi'(0)\,t}{\sigma \sqrt{\lambda_n}}\right)}\right| =  \E^{\lambda_n \,\Re\!\left(\phi\!\left(\frac{t}{\sigma \sqrt{\lambda_n}}\right) - \frac{\phi'(0)\,t}{\sigma \sqrt{\lambda_n}}\right)} \leq  \E^{\left(\sup_{\xi \in [-\eps, \eps]} \!\Re(\phi''(\xi))\right)\, \frac{t^2}{2\sigma^2} } \leq  \E^{-\frac{t^2}{4}}.$$
The integrand is therefore dominated by $\E^{-\frac{t^{2}}{4}}\, |t|^{r+1}$. Hence, 
$$ \mu_n(\{k\}) - \nu_n(\{ k\})  = \frac{\beta}{2 \pi (\sigma^2 \lambda_n)^{\frac{r}{2}+1}}\left(\int_{-\infty}^\infty (\I t)^{r+1}\,\E^{-\frac{t^{2}}{2}}   \,\E^{-\I t\frac{k-\lambda_n m}{\sigma \sqrt{\lambda_n}}}\,dt\right)+o\!\left( \frac{1}{(\lambda_n)^{\frac{r}{2}+1}}\right),
$$
with a remainder uniform in $k \in \Z$. Last, we use the fact that
$$ \int_{-\infty}^\infty \E^{ -\frac{t^2}{2}}\,\E^{-\I t \alpha}\,\frac{dt}{\sqrt{2\pi}}
   = \E^{-\frac{ \alpha^2}{2}},$$ 
is the Fourier transform of the standard Gaussian distribution evaluated at $\xi=-\alpha$, and by taking the $(r+1)$-th derivative with respect to $\alpha$, we get:
$$ \int_{-\infty}^\infty (\I t)^{r+1}\,\E^{ -\frac{t^2}{2}}\,\E^{-\I t \alpha}\,\frac{dt}{\sqrt{2\pi}}
   =(-1)^{r+1} \frac{\partial^r}{\partial \alpha^{r+1}}\left( \E^{-\frac{ \alpha^2}{2}}\right),$$ 
hence the claimed result.
\end{proof}
\bigskip

The computation of the local distance amounts then to find the maximum in $k \in \Z$ of the previous quantity. Recall that the $r$-th Hermite polynomial $H_{r}(\alpha)$ is defined by
$$G_{r}(\alpha)= \frac{\partial^r}{\partial \alpha^{r}}\left( \E^{-\frac{ \alpha^2}{2}}\right) = (-1)^{r}\,H_{r}(\alpha)\,\E^{-\frac{ \alpha^2}{2}}.$$
The local extremas of $G_{r}$ correspond to the zeros of $H_{r+1}$ since
$$G_{r}'(\alpha) = (-1)^{r}\,\left(H_{r}'(\alpha)-\alpha\,H_{r}(\alpha)\right)\E^{-\frac{ \alpha^2}{2}} = (-1)^{r+1}\,H_{r+1}(\alpha)\,\E^{-\frac{ \alpha^2}{2}}=G_{r+1}(\alpha) .$$
Denote $z_{r+1}$ the smallest absolute value of a zero of $H_{r+1}$; it is $0$ when $r$ is even, and it can be shown that in any case it corresponds to the global extrema of $|G_{r}|$; see \cite[Chapter 6]{Sze39}, and Figure \ref{fig:hermite} for an illustration. 
\begin{center}
\begin{figure}[ht]
\begin{tikzpicture}[scale=0.8]
\draw [->] (-6,0) -- (6,0);
\draw [->] (0,-3.5) -- (0,3.5);
\draw (0.48,-0.2) -- (0.48,2.85);
\draw (0.5,-0.4) node {$z_{r+1}$};
\draw (6.25,0) node {$\alpha$};
\draw (3.6,0.8) node {$G_r(\alpha)$};
\draw[very thick,domain=-5.2:5.2,smooth,NavyBlue] plot (\x,{(9.45+\x*\x*(-12.6+\x*\x*(3.78+\x*\x*(-0.36+0.01*\x*\x))))*\x*exp(-\x*\x/2)});
\end{tikzpicture}
\caption{The Hermite function $G_{9}$; it attains its global extremas at the smallest zeroes of $H_{10}$.\label{fig:hermite}}
\end{figure}
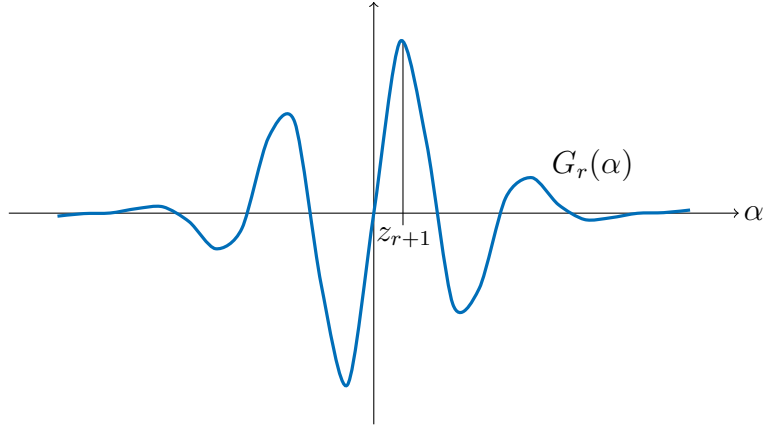
\end{center}

\begin{theorem}\label{thm:mainloc}
Under the assumptions stated at the beginning of this section (general scheme of approximation with condition \eqref{eq:H1}), one has 
$$\dloc(\mu_{n},\nu_{n}) = \frac{|\beta|\,|G_{r+1}(z_{r+2})|}{\sqrt{2 \pi}\, (\sigma^2 \lambda_n)^{\frac{r}{2}+1}}+o\!\left( \frac{1}{(\lambda_n)^{\frac{r}{2}+1}}\right) .$$
\end{theorem}

\begin{proof}
Denote $Z_n=\{\frac{k-\lambda_n m}{\sigma\sqrt{\lambda_n}},\,\,k\in\Z\}$. The previous discussion shows that
$$\dloc(\mu_n,\nu_n) = \frac{|\beta|}{\sqrt{2 \pi}\, (\sigma^2 \lambda_n)^{\frac{r}{2}+1}}\,\sup_{\alpha \in Z_n} |G_{r+1}(\alpha)|+o\!\left( \frac{1}{(\lambda_n)^{\frac{r}{2}+1}}\right).$$
As $n$ goes to infinity, the set $Z_n$ becomes dense in $\R$, and since $G_{r+1}$ is a Lipschitz function, we have in fact
\begin{align*}
\dloc(\mu_n,\nu_n) &= \frac{|\beta|}{\sqrt{2 \pi}\, (\sigma^2 \lambda_n)^{\frac{r}{2}+1}}\,\sup_{\alpha \in \R} |G_{r+1}(\alpha)|+o\!\left( \frac{1}{(\lambda_n)^{\frac{r}{2}+1}}\right)\\
 &= \frac{|\beta|\,|G_{r+1}(z_{r+2})|}{\sqrt{2 \pi}\, (\sigma^2 \lambda_n)^{\frac{r}{2}+1}}+o\!\left( \frac{1}{(\lambda_n)^{\frac{r}{2}+1}}\right).\qedhere
\end{align*}
\end{proof}

\begin{remark}
The first values of $M_{r}=|G_{r}(z_{r+1})|$ are
\begin{align*}
M_{0} &= |G_0(0)| = 1;\\
 M_{1} &= |G_1(1)| = \E^{-\frac{1}{2}} = 0.60653\ldots;\\
M_{2} &= |G_2(0)| = 1;\\
M_{3} &= \left|G_3\!\left(\sqrt{3-\sqrt{6}}\right)\right| = \E^{-\frac{3-\sqrt{6}}{2}}\,\left(3\sqrt{6}-6\right)= 1.38012\ldots;\\
M_4 &= |G_4(0)| = 3.
\end{align*}
In particular, $M_{2r}=(2r-1)!! = \frac{(2r)!}{2^r\,r!}$.
\end{remark}
\bigskip

\subsection{The Kolmogorov distance}\label{subsec:kolmogorovdistance}
To evaluate the Kolmogorov distance between $\mu_n$ and $\nu_n$, we use the classical integral formula
$$\mu_n(\lle k,+\infty\rre) - \nu_n(\lle k,+\infty\rre) = \int_{\T} \left( \psi_n(\xi) - \chi_n(\xi) \right)\,\E^{\lambda_n \phi(\xi)} \, \frac{\E^{-\I k \xi}}{1-\E^{-\I \xi}}\,\frac{d\xi}{2\pi}.$$
 Indeed, for any $l\geq k$, one has
\begin{align*}
\mu_n(\lle k,l\rre) - \nu_n(\lle k,l\rre) &= \sum_{j=k}^l\int_{\T} \left( \psi_n(\xi) - \chi_n(\xi) \right)\,\E^{\lambda_n \phi(\xi)} \, \E^{-\I j \xi}\,\frac{d\xi}{2\pi}\\
&= \int_{\T} \left( \psi_n(\xi) - \chi_n(\xi) \right)\,\E^{\lambda_n \phi(\xi)} \, \frac{\E^{-\I k \xi}-\E^{-\I (l+1)\xi}}{1-\E^{-\I\xi}}\,\frac{d\xi}{2\pi}.
\end{align*}
Since $\psi_n(\xi)-\chi_n(\xi) = \beta_n\,(\I\xi)^{r+1}\,(1+o(1))$, the quotient $\frac{\psi_n(\xi)-\chi_n(\xi)}{1-\E^{\I \xi}}$ is actually a continuous function on the torus, hence bounded. The same holds for $\E^{\lambda_n\phi(\xi)} = \esper[\E^{\I \xi Y_n}]$, where $Y_n$ is an infinitely divisible random variable with exponent $\lambda_n\phi$. Thus, 
$$f(\xi)= \frac{\psi_n(\xi)-\chi_n(\xi)}{1-\E^{\I \xi}}\,\E^{\lambda_n\phi(\xi)}$$ 
is a bounded continuous function, hence in $\leb^1(\T)$, and by the Riemann--Lebesgue lemma, one concludes that
$$\lim_{l \to +\infty} \int_{\T} f(\xi)\,\E^{-\I(l+1)\xi} \,\frac{d\xi}{2\pi} = 0 .$$
Since $\lim_{l \to +\infty} \mu_n(\lle k,l\rre) - \nu_n(\lle k,l\rre) = \mu_n(\lle k,+\infty\rre) - \nu_n(\lle k,+\infty\rre)$, the integral formula is indeed shown. It leads to the analogue of Proposition \ref{prop:localestimate} for probabilities of half-lines of integers $\lle k,+\infty\rre$:

\begin{proposition}
The error term being uniform in $k \in \Z$, one has
$$\mu_n(\lle k,+\infty\rre) - \nu_n(\lle k,+\infty\rre) =\frac{(-1)^{r}\,\beta}{\sqrt{2 \pi} \left( \sigma^2 \lambda_n \right)^{\frac{r+1}{2}}}\, \left.\frac{\partial^{r}}{\partial \alpha^{r}}\left( \E^{-\frac{\alpha^2}{2}}\right)\right|_{\alpha = \frac{k-\lambda_nm}{\sigma \sqrt{\lambda_n}}}
 + o\!\left( \frac{1}{(\lambda_n)^{\frac{r+1}{2}}}\right).$$
\end{proposition}

\begin{proof}
One has the Taylor expansion
$$\frac{\psi_n(\xi)-\chi_n(\xi)}{1-\E^{-\I\xi}} = \frac{\beta_n(\I\xi)^{r+1}}{\I\xi}\,(1+o_{\xi}(1)) = \beta_n (\I\xi)^r\,(1+o_{\xi}(1)).$$
Combined with the aforementioned integral formula and with the same Laplace method as in Proposition \ref{prop:localestimate}, it yields
$$\mu_n(\lle k,+\infty\rre) - \nu_n(\lle k,+\infty\rre) = \frac{\beta}{(\sigma^2 \lambda_n)^{\frac{r+1}{2}}} \,(1 + o(1))\,\int_{-\infty}^\infty (\I t)^{r}\,\E^{-\frac{t^{2}}{2}}   \,\E^{-i t\frac{k-\lambda_n m}{\sigma \sqrt{\lambda_n}}}\,dt.$$
Therefore, one has the same asymptotics as in Proposition \ref{prop:localestimate}, but with $r$ replacing $r+1$.
\end{proof}\bigskip

The same argument as in \S\ref{subsec:localdistance} gives:
\begin{theorem}\label{thm:mainkol}
Under the assumptions stated at the beginning of this section (general scheme of approximation with condition \eqref{eq:H1}), one has 
$$\dkol(\mu_{n},\nu_{n}) = \frac{|\beta|\,|G_{r}(z_{r+1})|}{\sqrt{2 \pi}\, (\sigma^2 \lambda_n)^{\frac{r+1}{2}}}+o\!\left( \frac{1}{(\lambda_n)^{\frac{r+1}{2}}}\right) .$$
\end{theorem}
\medskip

\begin{example}\label{ex:distancebasic}
Consider the basic scheme of approximation of a mod-$\phi$ convergent sequence of random variables $(X_n)_{n \in \N}$:
$$\widehat{\nu_n}(\xi) = \E^{\l_n \phi(\xi)}.$$
Hence, $\nu_n$ is the law of an infinitely divisible random variable $Y_n$ with L\'evy--Khintchine exponent $\l_n\phi$. Then, if $\widehat{\mu_n}(\xi) = \E^{\l_n \phi(\xi)}\,\psi_n(\xi)$ and $\psi_n(\xi) \to \psi(\xi)$ in $\mathscr{C}^{1}(\T)$, the hypothesis \eqref{eq:H1} is satisfied with $r=0$ and $\beta_n\to \beta = \psi'(0)$. Therefore, in this setting,
\begin{align*}
\dloc(X_n,Y_n) &= \frac{|\psi'(0)|\, \E^{-\frac{1}{2}}}{\sqrt{2\pi}\,\sigma^2\lambda_n} \,+ o\!\left(\frac{1}{\lambda_n}\right) ;\\
\dkol(X_n,Y_n) &= \frac{|\psi'(0)|}{\sqrt{2\pi\,\sigma^2\lambda_n}} + o\!\left(\frac{1}{\sqrt{\lambda_n}}\right).
\end{align*}
The second result should be compared with a computation in \cite[Chapter 4]{FMN16}, which ensures that 
$$\dkol(X_n,Y_n) = \frac{|\psi'(0)|}{\sqrt{2\pi\,\lambda_n}} + o\!\left(\frac{1}{\sqrt{\lambda_n}}\right)$$
if $(X_n)_{n \in \N}$ is a sequence of random real numbers that converges mod-Gaussian with parameters $\l_n$, and if $(Y_n)_{n \in \N}$ is a sequence of random Gaussian variables with means $0$ and variances $\l_n$.
\end{example}
\bigskip

\subsection{The total variation distance}\label{subsec:totalvariationdistance}
Before we estimate the total variation distance between the laws $\mu_n$ and $\nu_n$, let us make some observations with Proposition \ref{prop:localestimate}. Assume again that one has a general scheme of approximation which satisfies the hypothesis \eqref{eq:H1}. For any fixed interval $I=[a,b]$, one can write
\begin{align}
&\dtv(\mu_n,\nu_n) = \sum_{k \in \Z} |\mu_n(\{k\}) - \nu_n(\{k\})| \geq \sum_{k = \lfloor\l_n m + a \,\sigma \sqrt{\l_n}\rfloor}^{\lfloor\l_n m + b\, \sigma \sqrt{\l_n}\rfloor} |\mu_n(\{k\}) - \nu_n(\{k\})| \nonumber\\
& \geq \frac{|\beta_n|}{\sqrt{2\pi}\,(\sigma^2 \l_n)^{\frac{r}{2}+1}} \left(\sum_{k = \lfloor\l_n m + a\, \sigma \sqrt{\l_n}\rfloor}^{\lfloor\l_n m + b\, \sigma \sqrt{\l_n}\rfloor} \left|\left(H_{r+1}(\alpha)\,\E^{-\frac{\alpha^2}{2}}\right)_{\alpha = \frac{k-\l_n m}{\sigma\sqrt{\l_n}}}\right|\right)+ o\!\left(\frac{(b-a)}{(\sigma^2 \l_n)^{\frac{r+1}{2}}}\right) \tag{RS}\label{eq:riemannsum}\\
& \geq \frac{|\beta_n|}{\sqrt{2\pi}\,(\sigma^2 \l_n)^{\frac{r+1}{2}}}\,\int_a^b \left|H_{r+1}(\alpha)\,\E^{-\frac{\alpha^2}{2}}\right|\, d\alpha + o\!\left(\frac{(b-a)}{(\sigma^2 \l_n)^{\frac{r+1}{2}}}\right), \nonumber
\end{align}
by identifying a Riemann sum in \eqref{eq:riemannsum}. Since this is true for any $a$ and $b$, and since $H_{r+1}(\alpha)\,\E^{-\frac{\alpha^2}{2}}$ is integrable, we conclude that in general, one has
$$\liminf_{n \to \infty} \left(\dtv(\mu_n,\nu_n)\,(\sigma^2 \l_n)^{\frac{r+1}{2}}\right) \geq |\beta|\,\int_{\R} |H_{r+1}(\alpha)|\,\E^{-\frac{\alpha^2}{2}}\,\frac{d\alpha}{\sqrt{2\pi}}.$$
The goal of this paragraph is to show that, under a slightly stronger hypothesis than \eqref{eq:H1}, this inequality is in fact an identity, and that the liminf above is actually a limit. To this purpose, it is convenient to introduce a third sequence $(\rho_n)_{n\in \N}$ of signed measures, defined by their Fourier transforms
$$\widehat{\rho_n}(\xi) = \E^{\l_n \phi(\xi)}\,\left(\psi_n(\xi)-\beta_n\,(\E^{\I\xi}-1)^{r+1}\right).$$
These signed measures have their values given by:

\begin{lemma}\label{lem:modifmeasure}
Denote $\nu_n^{(0)}$ the infinitely divisible discrete law on $\Z$ with exponent $\l_n\phi$. The signed measure $\rho_n$ whose Fourier transform is $\widehat{\rho_n}(\xi) =  \E^{\l_n\phi(\xi)}\, (\psi_n(\xi)-\beta_n\,(\E^{\I \xi} - 1)^{r+1})$ is given by
$$ \rho_n(\{k\}) = \mu_n(\{k\}) - \beta_n \sum_{l=0}^{r+1} (-1)^{r+1-l}\, \binom{r+1}{l}\, \nu_n^{(0)}(\{k-l\}).$$
\end{lemma}

\begin{proof} One expands the Fourier transform $\widehat{(\mu_n - \rho_n)}(\xi)$ in powers of $\E^{\I \xi}$:
\begin{align*}
\widehat{(\mu_n - \rho_n)}(\xi) & = \beta_n \, \left( \E^{\I \xi} - 1 \right)^{r+1} \,\sum_{m=-\infty}^\infty  \nu_n^{(0)}(\{m\})\,\E^{\I m \xi}\\
& = \beta_n\,\sum_{l=0}^{r+1}  \sum_{m=-\infty}^\infty (-1)^{r+1-l} \binom{r+1}{l}\, \nu_n^{(0)}(\{m\}) \,\E^{\I (l+m)\xi}\end{align*}
\begin{align*}
& =  \sum_{k=-\infty}^\infty \E^{\I k\xi }\, \left(\beta_n\,\sum_{l=0}^{r+1} (-1)^{r+1-l}  \binom{r+1}{l}\, \nu_n^{(0)}(\{k-l\})\right)\\
&= \sum_{k=-\infty}^{\infty} \E^{\I k \xi}\, \left(\mu_n(\{k\}) - \rho_n(\{k\})\right).
\end{align*}
The result follows by identification of the Fourier coefficients.
\end{proof}
\medskip

\begin{remark}
Suppose in particular that $\phi(\xi) = \E^{\I \xi}-1$ is the exponent of the Poisson law $\mathcal{P}_{(1)}$. We then have
\begin{align*}
\mu_n(\{k\}) - \rho_n(\{k\}) &= \beta_n\,\sum_{l=0}^{(r+1 )\wedge k} (-1)^{r+1-l}\binom{r+1}{l}\,\frac{(\l_n)^{k-l}\,\E^{-\l_n}}{(k-l)!} \\
&= \beta_n \,\nu_n^{(0)}(\{k\})\,\sum_{l=0}^{(r+1 )\wedge k} (-1)^{r+1-l}\binom{r+1}{l}\,\frac{(\l_n)^{-l}\,k!}{(k-l)!}\\
&= \beta_n\,\nu_n^{(0)}(\{k\})\,c(r+1,k,\l_n),
\end{align*}
where $c(r+1,k,\l_n)$ is a Poisson--Charlier polynomial, see \cite[\S2.8.1]{Sze39}. 
\end{remark}
\bigskip

We now assume until the end of this paragraph that all the residues $\psi_n$, $\psi$, $\chi_n$ and $\chi$ are in $\mathscr{C}^1(\T)$, and that
\begin{align*}
\forall n \in \N,\,\,\, &\psi_n'(\xi)-\chi_n'(\xi) = \I(r+1)\,\beta_n\,(\I\xi)^{r}+\I(r+2)\,\gamma_n\,(\I\xi)^{r+1}\,(1+o_{\xi}(1))\\
\text{and}\,\,\, &\psi'(\xi)-\chi'(\xi) = \I(r+1)\,\beta\,(\I\xi)^{r}+\I(r+2)\,\gamma\,(\I\xi)^{r+1}\,(1+o_{\xi}(1))\tag{H2}\label{eq:H2}
\end{align*}
with $\lim_{n \to \infty} \beta_n = \beta$ and $\lim_{n \to \infty} \gamma_n = \gamma$. As in Equation \eqref{eq:H1}, the $o_\xi(1)$ tend to $0$ as $\xi \to 0$, uniformly in $n$. We denote this new condition by \eqref{eq:H2}. Since $\psi_n(0)=\chi_n(0)=\psi(0)=\chi(0)=1$, it implies by integration that 
\begin{align*}
\forall n \in \N,\,\,\, &\psi_n(\xi)-\chi_n(\xi) = \beta_n\,(\I\xi)^{r+1}+\gamma_n\,(\I\xi)^{r+2}\,(1+o_{\xi}(1))\\
\text{and}\,\,\, &\psi(\xi)-\chi(\xi) = \beta\,(\I\xi)^{r+1}+\gamma\,(\I\xi)^{r+2}\,(1+o_{\xi}(1))
\end{align*}
so in particular \eqref{eq:H2} is stronger than \eqref{eq:H1}. On the other hand, the hypothesis \eqref{eq:H2} is satisfied for instance if the convergences $\psi_n \to \psi$ and $\chi_n \to \chi$ occur in $\mathscr{C}^{r+2}(\T)$, and if the scheme of approximation $(\nu_n)_{n \in \N}$ is the scheme of order $r$ as defined in Example \ref{ex:approximationorderr}.

\begin{lemma}
Assume that $(\nu_n)_{n \in \N}$ is a general approximation scheme of a mod-$\phi$ convergent sequence $(X_n)_{n\in \N}$, with the new hypothesis \eqref{eq:H2} satisfied. Then, with $\rho_n$ defined as in Lemma \ref{lem:modifmeasure},
$$\dtv(\rho_n,\nu_n) = o\!\left(\frac{1}{(\l_n)^{\frac{r+1}{2}}}\right).$$
Moreover, $(\rho_n)_{n \in \N}$ is a new general approximation scheme of $(X_n)_{n\in \N}$, with the hypothesis \eqref{eq:H2} again satisfied, with the same parameter $r\geq 0$ and the same sequence $(\beta_n)_{n \in \N}$.
\end{lemma}

\begin{proof}
Denote $\widetilde{\psi_n}(\xi) = \psi_n(\xi)-\beta_n\,(\E^{\I\xi}-1)^{r+1}$, so that $\widehat{\rho_n}(\xi) = \E^{\l_n\phi(\xi)}\,\widetilde{\psi_n}(\xi)$. On the one hand, one has
\begin{align*}
(\widetilde{\psi_n} - \chi_n)'(\xi) &= (\psi_n'(\xi) - \chi_n'(\xi)) -\I(r+1)\,\beta_n\,\E^{\I\xi}(\E^{\I\xi}-1)^{r} \\
& = \I(r+2)\,\left(\gamma_n - \frac{\beta_n(r+1)}{2}\right)\,(\I\xi)^{r+1}\, (1+o_\xi(1))
\end{align*}
by using the Taylor expansion $(\E^{\I\xi}-1)^{r} = (\I\xi)^{r} + \frac{r(\I \xi)^{r+1}}{2} + o(\xi^{r+1})$. By the remark following the proof of Theorem \ref{thm:concentrationinequality}, this is a sufficient condition in order to use the norm inequality of Section \ref{sec:wiener}, but with parameter $r+1$ instead of $r$. Moreover, the sequence $$\left(\gamma_n - \frac{\beta_n(r+1)}{2}\right)_{n \in \N}$$ is convergent under the hypothesis \eqref{eq:H2}, hence bounded. As a consequence, the sequence $(\widetilde{\psi_n}-\chi_n)_{n \in \N}$ is bounded in $\mathscr{C}^1(\T)$, and therefore in $\Ac$ by Proposition \ref{prop:H_in_A}. For the same reason, one has a uniform bound on the parameters 
$$\beta_{r+2,n}(\eps) = \sup_{\xi \in [-\eps,\eps]} \frac{\left|(\widetilde{\psi_n}-\chi_n)'(\xi)\right|}{|\xi|^{r+1}} .$$ 
According to the discussion after Theorem \ref{thm:concentrationinequality}, we can conclude that
$$\dtv(\rho_n,\nu_n)= O\!\left(\frac{1}{(\lambda_n)^{\frac{r+1}{2}+\frac{1}{4}}}\right) = o\!\left(\frac{1}{(\lambda_n)^{\frac{r+1}{2}}}\right).$$
This ends the proof of the first part of the lemma, and for the second part, we can write
\begin{align*}
\psi_n'(\xi) - \widetilde{\psi_n}'(\xi) &= \I(r+1)\,\beta_n\,\E^{\I\xi}(\E^{\I\xi}-1)^r\\
&=\I(r+1)\,\beta_n\,(\I\xi)^{r}+\I(r+2)\,\frac{\beta_n(r+1)}{2}\,(\I\xi)^{r+1}\,(1+o_{\xi}(1)),
\end{align*}
hence the hypothesis \eqref{eq:H2} is satisfied.
\end{proof}
\bigskip

The previous Lemma shows that, if under the hypothesis \eqref{eq:H2} one wants to obtain an estimate of $\dtv(\mu_n,\nu_n)$ which is a $O((\l_n)^{-\frac{r+1}{2}})$, then one can assume without loss of generality that 
$$\psi_n(\xi)-\chi_n(\xi) = \beta_n\,(\E^{\I\xi}-1)^{r+1},$$
with $ \lim_{n \to \infty} \beta_n = \beta$ (in other words, $\nu_n=\rho_n$ is given exactly by the formula of Lemma \ref{lem:modifmeasure}). This is now of course much easier, and one obtains:

\begin{theorem}\label{thm:maindtv}
Consider a general scheme of approximation $(\nu_n)_{n \in \N}$ of a mod-$\phi$ convergent sequence of random variables $(X_n)_{n \in \N}$. We also assume that the reference law $\phi$ has a third moment, so that in the neighborhood of $0$,
$\phi(\xi) = -\frac{\sigma^2\xi^2}{2} + O(|\xi|^3)$. If the hypothesis \eqref{eq:H2} is satisfied, then
$$\dtv(\mu_n,\nu_n) = \frac{|\beta|}{\sqrt{2\pi}\,(\sigma^2\l_n)^{\frac{r+1}{2}}}\left(\int_{\R} |G_{r+1}(\alpha)|\,d\alpha\right) + o\!\left(\frac{1}{(\l_n)^{\frac{r+1}{2}}}\right).$$
\end{theorem}

\begin{proof}
According to the previous discussion, we have to compute the asymptotics of 
\begin{align*}
\dtv(\mu_n,\nu_n) &= \normA{\widehat{\mu_n}-\widehat{\nu_n}} = \beta_n\,\normA{\E^{\l_n\phi(\xi)}\,(\E^{\I\xi}-1)^{r+1}}\\
&=\beta_n\,\sum_{k\in \Z} \left|\int_{\T} \E^{\l_n\phi(\xi)}\,(\E^{\I\xi}-1)^{r+1}\,\E^{-\I k \xi}\,\frac{d\xi}{2\pi}\right|.
\end{align*}
We can here follow the arguments of \cite[Proposition 1]{Hwa99}, but in the general case of a discrete reference measure with L\'evy exponent $\phi$ (instead of a Poisson distribution). \vspace{2mm}

\begin{enumerate}[Step 1:]

\item We remove all the terms $k$ outside the range of the central limit theorem $\frac{Y_{n}-\l_n m}{\sigma\sqrt{\l_n}} \rightharpoonup \mathcal{N}_{\R}(0,1)$, where $Y_n$ follows the infinitely divisible law with exponent $\l_n\phi$. We claim that the terms $k$ outside the interval 
 $$I=I\!\left((\l_n)^{\frac{1}{7}}\right)=\left[\!\!\left[ \lfloor \l_n m-\sigma (\l_n)^{\frac{1}{7}+\frac{1}{2}} \rfloor ,\lfloor \l_n m+\sigma (\l_n)^{\frac{1}{7}+\frac{1}{2}} \rfloor \right]\!\!\right]$$
  give an exponentially small contribution to the sum. Indeed, note that
\begin{align*}\sum_{k \in \Z} \left|\mu_n(\{k\}) - \nu_n(\{k\})\right|&=\beta_n\,\sum_{k \in \Z}\left|\sum_{l=0}^{r+1} (-1)^{r+1-l}\,  \binom{r+1}{l}\, \nu_n^{(0)}(\{k-l\})\right|\\[2mm]
&=\sum_{k \in \Z}\left|\left((\Delta_{-1})^{\circ (r+1)} \nu_n^{(0)}\right)\!(\{k\})\right|\\
&=\sup_{f :\Z\to[-1,1]} \left(\sum_{k \in \Z} \left((\Delta_{-1})^{\circ (r+1)} \nu_n^{(0)}\right)\!(\{k\})\,\,f(k)\right)\\
&=\sup_{f :\Z\to[-1,1]} \left(\sum_{k \in \Z} \nu_{n}^{(0)}(\{k\})\,\,\left\{(\Delta_{1})^{\circ (r+1)}f\right\}\!(k)\right)
\end{align*}
with the notations $\Delta_{j}$ of discrete derivatives introduced at the end of \S\ref{subsec:modphi}. Indeed, $\Delta_1$ and $\Delta_{-1}$ are adjoint operators on $\ell^2(\Z)$. By the central limit theorem, for every $f: \Z \to [-1,1]$, the sum over terms $k$ outside the interval $I$ is now bounded by 
$$2^{r+1}\,\proba\left[\left|\frac{Y_{n}-\l_n m}{\sigma\sqrt{\l_n}}\right|\geq (\l_n)^{\frac{1}{7}}\right]=o\!\left(\frac{1}{(\l_n)^p}\right)\quad\text{for every exponent }p,$$
 so it will not contribute in the asymptotics. Therefore,
 $$\dtv(\mu_n,\nu_n)= |\beta_n|\sum_{k \in I((\l_n)^{\frac{1}{7}})} \left|\int_{\T} \E^{\l_n\phi(\xi)}\,(\E^{\I \xi}-1)^{r+1}\,\E^{-\I k\xi}\,\frac{d\xi}{2\pi}\right| + o\!\left(\frac{1}{(\l_n)^{\frac{r+1}{2}}}\right).$$\vspace{2mm}

 \item We now estimate the terms with $k \in I((\l_n)^{1/7})$ with a more careful application of  the Laplace method  than in Proposition \ref{prop:localestimate}. Namely, outside the interval $(-(\l_n)^{\frac{1}{7}-\frac{1}{2}},(\l_n)^{\frac{1}{7}-\frac{1}{2}})$, one can bound $(\E^{\I\xi}-1)^{r+1}$ by $|\xi|^{r+1}$ and $\E^{\l_n \phi(\xi)}$ by $\E^{-\l_n M\xi^2}$, whence
\begin{align*}
&\left|\int_{(-\pi,\pi)\setminus\left(-(\l_n)^{\frac{1}{7}-\frac{1}{2}},(\l_n)^{\frac{1}{7}-\frac{1}{2}}\right)} \E^{\lambda_n\phi(\xi)}\,(\E^{\I \xi}-1)^{r+1}\,\E^{-\I k\xi}\,d\xi\right| \\
& \leq \int_{(-\pi,\pi)\setminus\left(-(\l_n)^{\frac{1}{7}-\frac{1}{2}},(\l_n)^{\frac{1}{7}-\frac{1}{2}}\right)}  \E^{-\l_n M \xi^{2}}\,|\xi|^{r+1}\,d\xi \\
&\leq \frac{2}{(\l_n)^{\frac{r+1}{2}}}\,\int_{(\l_n)^{\frac{1}{7}}}^{\infty} \E^{-Mu^{2}}\,u^{r+1}\,du
\end{align*}
which is exponentially small and will not contribute to the asymptotics, even after multiplication by the number of terms $O((\l_n)^{\frac{1}{2}+\frac{1}{7}})$ of the interval $I$. On the other hand, if $k=\l_n m+x\,\sigma\sqrt{\l_n}$, then
\begin{align*}
&\frac{1}{2\pi} \int_{-(\l_n)^{\frac{1}{7}-\frac{1}{2}}}^{(\l_n)^{\frac{1}{7}-\frac{1}{2}}} \E^{\l_n\phi(\xi)}\,(\E^{\I \xi}-1)^{r+1}\,\E^{-\I k\xi}\,d\xi \\
&=\frac{1}{2\pi} \int_{-(\l_n)^{\frac{1}{7}-\frac{1}{2}}}^{(\l_n)^{\frac{1}{7}-\frac{1}{2}}} \E^{-\I (k - \l_n m) \xi - \l_n \frac{\sigma^{2}}{2}\,\xi^{2}}\,(\I \xi)^{r+1}\,\left(1+O((\l_n)^{\frac{3}{7}-\frac{1}{2}})\right)\,d\xi\\
&=\frac{\left(1+o(1)\right)}{2\pi\,(\sigma^{2}\,\l_n)^{\frac{r}{2}+1}} \int_{-\sigma (\l_n)^{\frac{1}{7}}}^{\sigma (\l_n)^{\frac{1}{7}}}\E^{-\I x u - \frac{u^{2}}{2}}\,(\I u)^{r+1}\,du=\frac{\left(1+o(1)\right)}{\sqrt{2\pi}\,(\sigma^{2}\,\l_n)^{\frac{r}{2}+1}}\, (-1)^{r+1} \,H_{r+1}(x)\,\E^{-\frac{x^{2}}{2}},
\end{align*}
the $o(1)$ being uniform in $k \in I((\l_n)^{1/7})$. Note that we used the hypothesis of third moment of the reference law $\phi$ in order to get the multiplicative error term $1+O((\l_n)^{\frac{3}{7}-\frac{1}{2}})$. \vspace{2mm}
\end{enumerate}
The proof can now be completed by using the same argument of Riemann sums as in the beginning of this paragraph, and the convergence $\beta_n \to \beta$.
\end{proof}
\medskip

\begin{remark}
The first values of $V_r = \int_{\R} |G_{r+1}(\alpha)|\,d\alpha$ are:
\begin{align*}V_{0}&=2;\\ 
V_{1}&=4\,\E^{-\frac{1}{2}};\\
V_{2}&=2\,(1+4\,\E^{-\frac{3}{2}}).
\end{align*}
\end{remark}
\medskip

\begin{example}
For the basic scheme of approximation ($\chi_n=\chi=1$), assume that the residues $\psi_n$ converge in $\mathscr{C}^2(\T)$, which implies hypothesis \eqref{eq:H2} with $r=0$ and $|\beta|=|\psi'(0)|$. We get the following estimate for the total variation distance between $X_n$ and an infinitely divisible random variable $Y_n$ with exponent $\l_n\phi$:
$$\dtv(X_n,Y_n) = \frac{2\,|\psi'(0)|}{\sqrt{2\pi \sigma^2 \l_n}}+ o\!\left(\frac{1}{\sqrt{\l_n}}\right).$$
Note that this is twice the asymptotic formula for the \emph{Kolmogorov} distance.
\end{example}
\bigskip

\subsection{Derived scheme of approximation with constant residue}\label{subsec:derived}
To conclude this section, let us simplify a bit the general setting of approximation of discrete measures presented in Definition \ref{def:approxscheme}. Until now, we have worked with measures $\nu_n$ such that
$$\widehat{\nu_n}(\xi) = \E^{\l_n \phi(\xi)}\,\chi_n(\xi),$$
where the residues $\chi_n(\xi)$ converge to a residue $\chi(\xi)$, and on the other hand approximate the residues $\psi_n(\xi)$ associated to the random variables $X_n$. For instance, $\chi_n(\xi)$ can be the Laurent polynomial of degree $r$ in $\E^{\I\xi}$ derived from $\psi_n(\xi)$ (scheme of approximation of order $r$). In this setting, the residues $\chi_n(\xi)$ are usually very good approximations of the residues $\psi_n(\xi)$, but on the other hand they vary with $n$, which is quite a complication for explicit computations. Hence, for applications, it is simpler to work with a \emph{constant} residue $\chi(\xi)$ that does not depend on $n$. This leads to the following definition:

\begin{definition}
Let $(\nu_n)_{n \in \N}$ be a general scheme of approximation of a sequence of random variables $(X_n)_{n \in \N}$ that converges mod-$\phi$ with parameters $\l_n$. The \emph{derived scheme of approximation with constant residue} is the new scheme $(\sigma_n)_{n \in \N}$, defined by the Fourier transforms
$$\widehat{\sigma_n}(\xi) = \E^{\l_n \phi(\xi)}\,\chi(\xi),$$
where $\chi(\xi) = \lim_{n \to \infty} \chi_n(\xi)$ is the limit of the residues of the law $\nu_n$ (with the notations of Definition \ref{def:approxscheme}).
\end{definition}

\begin{example}
Suppose that $\nu_n=\nu_n^{(r)}$ is the scheme of approximation of order $r$ of $(X_n)_{n \in \N}$ (see Example \ref{ex:approximationorderr}), and that the convergence $\psi_n \to \psi$ occurs in $\mathscr{C}^{r}(\T)$. Then, the derived scheme with constant residue $(\sigma_n)_{n \in \N}$ is defined by the Fourier transforms
$$\widehat{\sigma_n}(\xi) = \E^{\l_n\phi(\xi)}\,P^{(r)}(\E^{\I\xi}),$$
where $P^{(r)}(\E^{\I\xi})$ is the Laurent polynomial of degree $r$ associated to the limit $\psi$ of the mod-$\phi$ convergence. Thus, to compute $\sigma_n(f)$ with $f : \Z \to \C$, one can take a random variable $Y_n$ with law of exponent $\l_n \phi$, and then calculate
$$\esper[(P^{(r)}(S)f)(Y_n)]$$
where $T=P^{(r)}(S)$ is a \emph{fixed} linear operator, which is a finite linear combination of discrete difference operators. This is clearly convenient for concrete applications.
\end{example}\bigskip

Informally, if the derived scheme with constant residue $(\sigma_n)_{n\in\N}$ is sufficiently close to the initial scheme $(\nu_n)_{n \in \N}$, then the conclusions of Theorems \ref{thm:mainloc}, \ref{thm:mainkol} and \ref{thm:maindtv} hold also for $(\sigma_n)_{n \in \N}$ (under the appropriate condition \eqref{eq:H1} or \eqref{eq:H2} for $(\nu_n)_{n \in \N}$). Practically, this happens as soon as the size of $\chi_n-\chi$ is negligible in comparison to negative powers of $\l_n$. More precisely, we have:

\begin{theorem}\label{thm:withfixedresidue}
Let $(\nu_n)_{n \in \N}$ be a general scheme of approximation of a sequence $(X_n)_{n\in \N}$ that is mod-$\phi$ convergent. We denote $(\sigma_n)_{n \in \N}$ the derived scheme with constant residue. \vspace{2mm}
\begin{enumerate}
	\item Suppose that $(\nu_n)_{n \in \N}$ satisfies the hypothesis \eqref{eq:H1}. If $\|\chi_n-\chi\|_\infty = o\!\left(\frac{1}{(\l_n)^{\frac{r}{2}+1}}\right)$, then
	$$\dloc(\mu_{n},\sigma_{n}) = \frac{|\beta|\,|G_{r+1}(z_{r+2})|}{\sqrt{2 \pi}\, (\sigma^2 \lambda_n)^{\frac{r}{2}+1}}+o\!\left( \frac{1}{(\lambda_n)^{\frac{r}{2}+1}}\right) .$$
	\vspace{2mm}
	\item Suppose that $(\nu_n)_{n \in \N}$ satisfies the hypothesis \eqref{eq:H1}. If $\|\chi_n'-\chi'\|_\infty = o\!\left(\frac{1}{(\l_n)^{\frac{r+1}{2}}}\right)$, then
	$$\dkol(\mu_{n},\sigma_{n}) = \frac{|\beta|\,|G_{r}(z_{r+1})|}{\sqrt{2 \pi}\, (\sigma^2 \lambda_n)^{\frac{r+1}{2}}}+o\!\left( \frac{1}{(\lambda_n)^{\frac{r+1}{2}}}\right) .$$
	\vspace{2mm}
	\item Suppose that $(\nu_n)_{n \in \N}$ satisfies the hypothesis \eqref{eq:H2}. If $\|\chi_n'-\chi'\|_\infty = o\!\left(\frac{1}{(\l_n)^{\frac{r}{2}+\frac{5}{4}}}\right)$, then	
	$$\dtv(\mu_n,\sigma_n) = \frac{|\beta|}{\sqrt{2\pi}\,(\sigma^2\l_n)^{\frac{r+1}{2}}}\left(\int_{\R} |G_{r+1}(\alpha)|\,d\alpha\right) + o\!\left(\frac{1}{(\l_n)^{\frac{r+1}{2}}}\right).$$
\end{enumerate}
\end{theorem}
\medskip

\begin{proof}
In each case, we have to bound the distance between $\nu_n$ and the derived scheme $\sigma_n$. For the local distance, one has
\begin{align*}
\dloc(\nu_n,\sigma_n) &= \sup_{k \in \Z} \left|\int_{\T} (\chi_n(\xi)-\chi(\xi))\,\E^{\l_n \phi(\xi)}\,\E^{-\I k \xi}\,\frac{d\xi}{2\pi}\right| \leq \|\chi_n-\chi\|_{\infty} = o\!\left(\frac{1}{(\l_n)^{\frac{r}{2}+1}}\right), 
\end{align*}
hence the estimate on $\dloc(\mu_n,\sigma_n)$. For the Kolmogorov distance, notice that
$$
\left|\frac{\chi_n(\xi)-\chi(\xi)}{1-\E^{-\I \xi}}\right| \leq \|\chi_n'-\chi'\|_\infty\,\left|\frac{\xi}{1-\E^{-\I\xi}}\right|  \leq \pi\,\|\chi_n'-\chi'\|_{\infty}
$$
for any $\xi \in (-\pi,\pi)$. Therefore,
\begin{align*}
\dkol(\nu_n,\sigma_n)&= \sup_{k \in \Z} \left|\int_{\T} \frac{\chi_n(\xi)-\chi(\xi)}{1-\E^{-\I\xi}}\,\E^{\l_n \phi(\xi)}\,\E^{-\I k \xi}\,\frac{d\xi}{2\pi}\right| \leq \pi\|\chi_n'-\chi'\|_{\infty} = o\!\left(\frac{1}{(\l_n)^{\frac{r+1}{2}}}\right), 
\end{align*}
hence the estimate on $\dkol(\mu_n,\sigma_n)$. Last, for the total variation distance, we use Proposition \ref{prop:H_in_A}:
\begin{align*}
\dtv(\nu_n,\sigma_n) &= \normA{(\chi_n-\chi)(\xi)\,\E^{\l_n \phi(\xi)}} \\
&\leq \|\chi_n-\chi\|_\infty + C_H\,\|((\chi_n-\chi)'(\xi)+\l_n\,\phi'(\xi)\,(\chi_n-\chi)(\xi))\,\E^{\l_n\phi(\xi)}\|_{\leb^2} \\
&\leq \pi \|\chi_n'-\chi'\|_\infty + C_H\,(\|\chi_n'-\chi'\|_{\infty}+\l_n\,\|\phi'\|_{\infty}\,\|\chi_n-\chi\|_\infty)\,\|\E^{\l_n\phi(\xi)}\|_{\leb^2} \\
&\leq \|\chi_n'-\chi'\|_\infty\,\left(\pi + C_H (1+\pi \l_n\,\|\phi'\|_\infty) \,\|\E^{\l_n\phi(\xi)}\|_{\leb^2} \right).
\end{align*}
Since $|\E^{\phi(\xi)}| \leq \E^{-M\xi^2}$, we have the estimate $\|\E^{\l_n\phi(\xi)}\|_{\leb^2} = O(\frac{1}{(\l_n)^{1/4}})$, so in the end
$$\dtv(\nu_n,\sigma_n) = O\!\left(\|\chi_n'-\chi'\|_\infty\,(\l_n)^{\frac{3}{4}}\right) = o\!\left(\frac{1}{(\l_n)^{\frac{r+1}{2}}}\right),$$
whence the estimate on $\dtv(\mu_n,\sigma_n)$.
\end{proof}
\bigskip
\bigskip

\section{One-dimensional applications}\label{sec:oneexample}

In this section, we apply our main Theorems \ref{thm:mainloc}, \ref{thm:mainkol} and \ref{thm:maindtv} to various one-dimensional examples, most of them coming from number theory or combinatorics. In each case, we compute the exact asymptotics of the distances for the basic approximation scheme, and for some better approximation schemes, which are derived from the approximation schemes of order $r\geq 1$.\bigskip

\subsection{Poisson approximation of a sum of independent Bernoulli variables}\label{subsec:poissonbernoulli}
As in Example \ref{ex:toymodel}, we consider a sum $X_n=\sum_{j=1}^n \mathcal{B}(p_j)$ of independent Bernoulli random variables, with $\sum_{j=1}^\infty p_j = +\infty$ and $\sum_{j=1}^\infty (p_j)^2 < + \infty$. The corresponding Fourier transforms are:
$$\widehat{\mu_n}(\xi) = \E^{\l_n(\E^{\I \xi}-1)}\,\psi_n(\xi)\quad\text{with }\psi_n(\xi) = \prod_{j=1}^n (1+p_j(\E^{\I \xi}-1))\,\E^{-p_j(\E^{\I\xi}-1)}.$$
It is convenient to have an exact formula for the expansion in Laurent series of the deconvolution residues $\psi_n$. Denote 
$$
\frakp_{1,n} = 0\quad\text{and}\quad \frakp_{k\geq 2,n} = \sum_{j=1}^n \,(p_j)^k.$$
One has 
\begin{align*}
\psi_n(\xi) &= \left(\prod_{j=1}^n (1+p_j(\E^{\I\xi}-1)) \right)\,\E^{-\sum_{j=1}^n p_j(\E^{\I\xi}-1)} \\
&=\exp\left(\sum_{j=1}^n \log(1+p_j(\E^{\I\xi}-1)) - p_j(\E^{\I\xi}-1)\right)\\
&=\exp\left(\sum_{j=1}^n \sum_{k=2}^{\infty} \frac{(-1)^{k-1}}{k}\,(p_j(\E^{\I\xi}-1))^k \right) \\
&=\exp\left(\sum_{k=1}^\infty \frac{(-1)^{k-1}}{k}\,\frakp_{k,n}\,(\E^{\I\xi}-1)^k\right).
\end{align*}
Denote $\mathfrak{P}$ the set of integer partitions, that is to say finite non-increasing sequences of positive integers 
$$L=(L_1 \geq L_2 \geq \cdots \geq L_r>0).$$
If $L$ is an integer partition, we denote $|L|=\sum_{i=1}^r L_i$, $\ell(L) = r$, and $m_{k}(L)$ the number of parts of $L$ of size $k$. We also denote $z_L = \prod_{k \geq 1} k^{m_k(L)}\,m_k(L)!$, which is the size of the centralizer of a permutation of cycle-type $L$ in the symmetric group $\sym_{|L|}$. The previous power series in $\E^{\I\xi}-1$ expands then as
$$\psi_n(\xi) = \sum_{L \in \mathfrak{P}} \frac{(-1)^{|L|-\ell(L)}}{z_L}\,\frakp_{L,n}\,(\E^{\I\xi}-1)^{|L|},$$
where $\frakp_{L,n} = \prod_{i=1}^{\ell(L)} \frakp_{L_i,n}$. Since $\frakp_{n,1}=0$, the sum actually runs over integer partitions without part equal to $1$. In particular, one has
$$\psi_n(\xi) = 1 - \frac{\frakp_{2,n}}{2}\,(\E^{\I\xi}-1)^2 + \frac{\frakp_{3,n}}{3} (\E^{\I\xi}-1)^3+ o(|\xi|^3).$$
Consider then the basic approximation scheme $(\nu_n^{(0)})_{n\in\N}$ of the sequence of random variables $(X_n)_{n\in\N}$, which is mod-Poisson convergent with parameters $\l_n = \sum_{j=1}^n p_j$ and limit $\psi(\xi) = \prod_{j=1}^\infty (1+p_j(\E^{\I \xi}-1))\,\E^{-p_j(\E^{\I\xi}-1)}$. It satisfies the hypothesis \eqref{eq:H1} with $r=1$ and
$$\beta_n=-\frac{1}{2}\,\frakp_{2,n}\qquad;\qquad\beta = -\frac{1}{2}\,\frakp_2 = -\frac{1}{2}\sum_{j=1}^\infty (p_j)^2.$$
By Theorems \ref{thm:mainloc} and \ref{thm:mainkol}, if $(Y_n)_{n \in \N}$ is a sequence of Poisson random variables with parameters $\l_n$, then
\begin{align*}
\dloc(X_n,Y_n) & = \frac{\frakp_2}{2\sqrt{2\pi}(\l_n)^{\frac{3}{2}}} + o\!\left(\frac{1}{(\l_n)^{\frac{3}{2}}}\right);\\
\dkol(X_n,Y_n) &= \frac{\frakp_2}{2\sqrt{2\pi\hspace{0.4mm}\E}\,\l_n}+ o\!\left(\frac{1}{\l_n}\right).
\end{align*}
The sequence of infinitely divisible laws $(\nu_n^{(0)})_{n \in \N}$ also satisfies the hypothesis \eqref{eq:H2}, since the expansion of $\psi_n(\xi)$ given before is exact and can be derived term by term. So, one can apply Theorem \ref{thm:maindtv}, and
$$\dtv(X_n,Y_n) = \frac{2\,\frakp_2}{\sqrt{2\pi\hspace{0.4mm}\E}\,\l_n} + o\!\left(\frac{1}{\l_n}\right).$$
This last estimate should be compared with the unconditional bound
$$ \dtv(X_n,Y_n) \leq 2\,\frac{\frakp_{2,n}}{\l_n} = 2\, \frac{\sum_{j=1}^n (p_j)^2}{\sum_{j=1}^n p_j}$$
stemming from the Chen--Stein method. On the other hand, we recover the asymptotic formula for $\dtv(X_n,Y_n)$ established by Deheuvels and Pfeifer in \cite{DP86}, using only tools from harmonic analysis (instead of the semi-group method developed in \emph{loc.~cit.}).\bigskip

We now look at a higher order approximation scheme, namely, the approximation scheme $(\nu_n^{(2)})_{n\in\N}$ of order $r=2$. It is defined by the Fourier transforms
$$\widehat{\nu_n^{(2)}}(\xi) = \E^{\l_n(\E^{\I\xi}-1)}\,\left(1-\frac{\frakp_{2,n}}{2}\,(\E^{\I\xi}-1)^2\right). $$
Again, the hypotheses \eqref{eq:H1} and \eqref{eq:H2} are satisfied, this time with $r=2$ and 
$$\beta_n = \frac{1}{3}\,\frakp_{3,n}\qquad;\qquad \beta = \frac{1}{3}\,\frakp_3 = \frac{1}{3}\sum_{j=1}^\infty (p_j)^3.$$
So, applying the main theorems of Section \ref{sec:asymptoticdistance}, we obtain
\begin{align*}
\dloc(X_n,\nu_n^{(2)}) &= \frac{(\sqrt{6}-2)\,\frakp_3}{\sqrt{2\pi\,\E^{3-\sqrt{6}}}\,(\l_n)^2} + o\!\left(\frac{1}{(\l_n)^2}\right); \\
\dkol(X_n,\nu_n^{(2)}) &= \frac{\frakp_3}{3\,\sqrt{2\pi}\,(\l_n)^{\frac{3}{2}}} + o\!\left(\frac{1}{(\l_n)^{\frac{3}{2}}}\right);\\
\dtv(X_n,\nu_n^{(2)}) &= \frac{2(1+4\E^{-\frac{3}{2}})\,\frakp_3}{3\,\sqrt{2\pi}\,(\l_n)^{\frac{3}{2}}} + o\!\left(\frac{1}{(\l_n)^{\frac{3}{2}}}\right).
\end{align*}
Thus, $(\nu_n^{(2)})_{n \in \N}$ is asymptotically a better approximation scheme than the simple Poisson approximation $(\nu_n^{(0)})_{n \in \N}$. Moreover, if 
$$\sum_{j=n+1}^{+\infty} (p_j)^2 = o\!\left((\l_n)^{-\frac{9}{4}}\right) =  o\!\left(\left(\sum_{j=1}^n p_j\right)^{\!-\frac{9}{4}}\right),$$
then all the results of Theorem \ref{thm:withfixedresidue} apply, and in the previous estimates one can replace $\nu_n^{(2)}$ by $\sigma_n^{(2)}$, defined by the Fourier transform
$$\widehat{\sigma_n^{(2)}}(\xi) = \E^{\l_n(\E^{\I\xi}-1)}\,\left(1-\frac{\sum_{j=1}^\infty (p_j)^2}{2}\,(\E^{\I\xi}-1)^2\right). $$
with fixed residue.\medskip

\begin{remark}
Performing the same computations as in Lemma \ref{lem:modifmeasure}, one can give an exact formula for $\nu_n^{(2)}(\{k\})$:
$$\nu_n^{(2)}(\{k\})=\E^{-\l_n}\,\frac{(\l_n)^k}{k!}\left(1+\beta\left(1-\frac{2k}{\l_n} + \frac{k(k-1)}{(\l_n)^2}\right)\right),$$
here with $\beta=-\frac{\frakp_{2,n}}{2}$. At first sight this might look like an important deformation of the Poisson measure with parameter $\l_n$, but note that in the range $k =\l_n(1+o(1))$ where the Poisson law $\nu_n^{(0)} = \mathcal{P}_{(\l_n)}$ has its mass concentrated, one has
$$1-\frac{2k}{\l_n} + \frac{k(k-1)}{(\l_n)^2} = 1-2(1+o(1)) + (1+o(1)) = o(1).$$
\end{remark}
\medskip

More generally, assume that the remainder of the series $\sum_{j=n+1}^{+\infty} (p_j)^2$ is asymptotically smaller than any negative power of $\l_n$. If 
$$\widehat{\sigma_n^{(r)}}(\xi) =  \E^{\l_n(\E^{\I\xi}-1)}\,\left(\sum_{k=0}^r \frake_k\,(\E^{\I\xi}-1)^k\right),$$
with 
$$
\frake_k = \sum_{\substack{L\text{ partition of size }k \\ \text{without part of size }1}} \frac{(-1)^{|L|-\ell(L)}}{z_L}\, \frakp_L\qquad;\qquad \frakp_{k\geq 2} = \sum_{j=1}^\infty (p_j)^k,
$$
then the sequence of signed measures $(\sigma_n^{(r)})_{n \in \N}$ forms a scheme of approximation of $(X_n)_{n \in \N}$, with
\begin{align*}
\dloc(X_n,\sigma_n^{(r)}) &= \frac{|\frake_{r+1}\,G_{r+1}(z_{r+2})|}{\sqrt{2\pi}\,(\l_n)^{\frac{r}{2}+1}} + o\!\left(\frac{1}{(\l_n)^{\frac{r}{2}+1}}\right);\\
\dkol(X_n,\sigma_n^{(r)}) &= \frac{|\frake_{r+1}\,G_{r}(z_{r+1})|}{\sqrt{2\pi}\,(\l_n)^{\frac{r+1}{2}}} + o\!\left(\frac{1}{(\l_n)^{\frac{r+1}{2}}}\right);\\
\dtv(X_n,\sigma_n^{(r)}) &= \frac{\int_{\R}|\frake_{r+1}\,G_{r+1}(\alpha)|\,d\alpha}{\sqrt{2\pi}\,(\l_n)^{\frac{r+1}{2}}} + o\!\left(\frac{1}{(\l_n)^{\frac{r+1}{2}}}\right).
\end{align*}
Thus, one can in this setting approximate the law of $X_n$ up to any order (negative power of $\l_n$).
\bigskip


\subsection{Poisson approximation schemes and formal alphabets}\label{subsec:alphabet}
The description of the approximation schemes and of the asymptotic estimates for the distances by means of the parameters $\frakp_k$ and $\frake_k$ is actually possible for many other examples of discrete random variables. Thus, the combinatorics of approximation schemes can be encoded in the \emph{algebra of symmetric functions} $\Sym$. From an algebraic point of view, the language of $\lambda$-rings and \emph{formal alphabets} make this description even more stunning. Denote $\Sym$ the algebra of symmetric functions, that is to say symmetric polynomials in an infinity of variables $\{x_1,x_2,x_3,\ldots\}$; we refer to \cite[Chapter 1]{Mac95} or \cite[Chapter 2]{Mel17} for the whole discussion of this paragraph. The algebra $\Sym$ admits for algebraic basis over $\C$ the power sums
$$\frakp_{k\geq 1} = \sum_{i} (x_i)^k$$
and the elementary functions
$$\frake_{k \geq 1} = \sum_{i_1<i_2<\cdots<i_k} x_{i_1}x_{i_2}\cdots x_{i_k};$$
thus, $\Sym = \C[\frakp_1,\frakp_2,\ldots]=\C[\frake_1,\frake_2,\ldots]$. The two bases are related by the identity of formal series
$$1  = \left(\prod_{i} (1-tx_i)\right) \exp\!\left(\sum_i \log \frac{1}{1-tx_i}\right) = \left(1+\sum_{k=1}^\infty \frake_k\,t^k\right) \exp\!\left( \sum_{k=1}^\infty \frac{\frakp_k}{k}\,t^k \right),$$
which leads to the formula
$$\frake_k = \sum_{L \in \mathfrak{P}_k} \frac{(-1)^{|L|-\ell(L)}}{z_L} \,\frakp_L$$
where $\mathfrak{P}_k$ is the set of integer partitions of size $k$, and $\frakp_L = \prod_{i=1}^{\ell(L)}\frakp_{L_i}$. This is a particular case of the Frobenius--Schur formula for Schur functions \cite[Theorem 2.32]{Mel17}.\medskip

A formal alphabet, or specialisation is a morphism of algebras $\psi: \Sym \to \C$. In particular, if $A=\{a_1,a_2,\ldots\}$ is a set of complex numbers with $\sum_{i} |a_i|<+\infty$, then the map 
$$\frakp_k \mapsto \frakp_k(A) = \sum_i (a_i)^k $$ 
can be extended to a morphism of algebras $\Sym \to \C$, hence a specialisation of $\Sym$ (in this case one can speak of a ``true'' alphabet for $A$). Beware that some specialisations are not of this form: indeed, a formal alphabet can send the power sums $\frakp_k$ to an arbitrary set of values in $\C$. Still, we can agree to use the notation $f \mapsto f(A)$ for any specialisation $f \mapsto \psi(f)$ of the algebra $\Sym$. In particular, suppose that $A=\{a_1,a_2,\ldots\}$ is a set of complex numbers, but this time with $\sum_i |a_i| = +\infty$ and $\sum_{i} |a_i|^2 < +\infty$. We then denote $f \mapsto f(A)$ the specialisation
$$\frakp_1(A) = 0\qquad;\qquad \frakp_{k \geq 2}(A) = \sum_i (a_i)^k.$$
This definition extends indeed to a unique morphism of algebras $\Sym \to \C$, although there is \emph{a priori} no underlying ``true'' alphabet.  Given two formal alphabets $A$ and $B$, one can define new formal alphabets $\eps(A)$ and $A+B$, as follows:
$$\frakp_k(\eps(A)) = (-1)^{k-1}\,\frakp_k(A)\qquad;\qquad \frakp_k(A+B) = \frakp_k(A) + \frakp_k(B).$$
As before, these definitions extend uniquely to new morphisms of algebras $\Sym \to \C$. If $A$ and $B$ were true alphabets, then $A+B$ is the true alphabet which is the disjoint union of $A$ and $B$. On the other hand, $\eps$ is related to a certain involution of the Hopf algebra of symmetric functions.\bigskip

Let us now relate the theory of symmetric functions to the approximation schemes of discrete distributions. Consider a sequence of random integer-valued variables $(X_n)_{n \in \N}$, whose Fourier transforms write as:
$$\esper[\E^{\I \xi X_n}] = \E^{\l_n (\E^{\I\xi}-1)} \exp\!\left( \sum_{k=2}^\infty \frac{c_{k,n}}{k!} (\E^{\I \xi}-1)^k \right).$$
The coefficients $\l_n = c_{1,n}$ and $c_{k\geq 2,n}$ are called the \emph{factorial cumulants} of $X_n$, and the Poisson random variables are characterized by the vanishing of their factorial cumulants of order $k \geq 2$ (see \cite[Section 5.2]{DVJ02}). We associate to these coefficients the following specialisation of $\Sym$:
$$\frakp_1(A_n)=0\qquad;\qquad \frakp_{k \geq 2}(A_n) = \frac{(-1)^{k-1}\,c_{k,n}}{(k-1)!}.$$
Then, the deconvolution residue $\psi_n(\xi) = \esper[\E^{\I \xi X_n}] \,\E^{-\l_n (\E^{\I \xi}-1)}$ admits the expansion:
$$\psi_n(\xi) = \sum_{k=0}^\infty \frake_k(A_n)\,(\E^{\I\xi}-1)^k,$$
with $\frake_0(A_n)=1$ and $\frake_1(A_n)=0$. The approximation scheme of order $r\geq 1$ is given by the residue
$$\chi_n^{(r)}(\xi) = \sum_{k=0}^r \frake_k(A_n)\,(\E^{\I\xi}-1)^k.$$
The reason why this formalism is convenient is that, for all the examples that we shall look at in this section, the limiting residue
$$\psi(\xi) = \sum_{k=0}^\infty \frake_k(A)\,(\E^{\I\xi}-1)^k,$$
corresponds to a formal alphabet $A$ which is explicit and which expresses in terms of ``natural'' parameters of the random model. These formal alphabets are given by Table \ref{tab:alphabet}.\medskip

\begin{table}[ht]
\begin{tabular}{|l|c|}
\hline \qquad$X_n$ & formal alphabet $A$ \\
\hline sum of independent Bernoulli  & $\{p_1,p_2,p_3\ldots\}$\\
variables with parameters $p_j$ (\S\ref{subsec:poissonbernoulli})& \\
\hline number of cycles of a random & ${\N^*}^{-1}$\\
 uniform permutation (\S\ref{subsec:disjointcycles}) &\\
\hline number of cycles of a random Ewens & $\Theta$ \\
 permutation with parameter $\theta$ (\S\ref{subsec:disjointcycles}) & \\
\hline number of connected components& $ (2\N+1)^{-1} $\\
of a uniform random map (\S\ref{subsec:combinatorialexamples})& \\
\hline number of distinct irreducible factors& ${\N^*}^{-1} + (q^{\deg \mathfrak{I}})^{-1}$  \\
of a random monic polynomial (\S\ref{subsec:combinatorialexamples})& \\
\hline number of irreducible factors of a & ${\N^*}^{-1} + \eps((q^{\deg \mathfrak{I}}-1)^{-1})$  \\
random monic polynomial (\S\ref{subsec:combinatorialexamples})& \\
\hline number of distinct prime divisors& ${\N^*}^{-1} + \P^{-1}$ \\
of a random integer (\S\ref{subsec:erdoskac})& \\
\hline
\end{tabular}\bigskip

\caption{Formal alphabets associated to the mod-Poisson convergent sequences.\label{tab:alphabet}}
\end{table}

\noindent In this table we denote $A^{-1}=\{\frac{1}{a},\,\,a\in A\}$ if $A$ is a true alphabet, and 
\begin{align*}
 \Theta &= \left\{1,\frac{\theta}{\theta+1},\frac{\theta}{\theta+2},\ldots\right\} ;\\
  q^{\deg \mathfrak{I}} &= \left\{q^{\deg P},\,\,P \in \mathfrak{I}\right\} ;\\
  (q^{\deg \mathfrak{I}}-1) &= \left\{q^{\deg P}-1,\,\,P \in \mathfrak{I}\right\} 
 \end{align*}
 where $\mathfrak{I}$ is the set of irreducible polynomials over the finite field $\mathbb{F}_q$. Therefore, the asymptotic behavior of a sequence $(X_n)_{n\in \N}$ in Table \ref{tab:alphabet} is entirely encoded in:\vspace{2mm}
 \begin{itemize}
  	\item a sequence of parameters $(\l_n)_{n \in \N}$,\vspace{2mm}
  	\item and a formal alphabet $A$ which in some sense captures the geometric and arithmetic properties of the model.\vspace{2mm}
  \end{itemize}   
  In particular, the reader should notice the appearance of two contributions to the formal alphabet as soon as one studies the problem of factorization of a random element in a factorial ring.

\subsection{Number of disjoint cycles of a random permutation}\label{subsec:disjointcycles}
In this section, we study the number $\ell_n$ of disjoint cycles of a random permutation $\sigma_n \in \sym_n$. If the permutation $\sigma_n$ is chosen according to the uniform probability measure (Example \ref{ex:sigman}), then the discussion of Section \ref{subsec:poissonbernoulli} applies, with the parameters $p_j$ taken equal to $\frac{1}{j}$, and
\begin{align*}
\l_n &= \sum_{j=1}^n \frac{1}{j} = H_n \simeq \log n;\\
\sum_{j=n+1}^{+\infty} (p_j)^2 & \simeq  \frac{1}{n} = o\!\left(\frac{1}{(\l_n)^r}\right) \quad\text{for every }r\geq0. 
\end{align*}
Therefore, if one looks for instance at the derived scheme with constant residue from the approximation scheme of order $r=2$:
$$\widehat{\sigma_n^{(2)}}(\xi) = \E^{H_n(\E^{\I\xi}-1)}\,\left(1-\frac{\pi^2}{12}\,(\E^{\I\xi}-1)^2\right)$$
then
$$\dloc(\ell_n,\sigma_n^{(2)}) = \frac{(\sqrt{6}-2)\,\zeta(3)}{\sqrt{2\pi\,\E^{3-\sqrt{6}}}\,(\log n)^2}+o\!\left(\frac{1}{(\log n)^2}\right),$$
and one has similar estimates for the Kolmogorov distance and for the total variation distance.
\bigskip

One can extend this result to more general models of random permutations, namely, random permutations $\sigma_n \in \sym_n$ under the so-called \emph{generalized weighted measures}
$$\proba_{\Theta,n}[\sigma_n] = \frac{1}{n!\,h_n(\Theta)}\,\prod_{k=1}^n (\theta_k)^{m_k(\sigma_n)},$$
where $\Theta = (\theta_k)_{k \geq 1}$ is a sequence of non-negative parameters, $m_k(\sigma)$ is the number of cycles of length $k$ in $\sigma \in \sym_n$, and $h_n(\Theta)$ is the normalization constant so that each $\proba_{\Theta,n}$ is a probability measure on $\sym_n$. If $\Theta=(\theta,\theta,\ldots)$ is a constant sequence, then one recovers the Ewens measures of parameter $\theta$:
$$\proba_n[\sigma_n] = \frac{\theta^{\ell(\sigma_n)}}{\theta(\theta+1)\cdots(\theta+n-1)}.$$ 
The generalized weighted measures have been studied previously in \cite{BU09,BU11,BUV11,EU12,NZ13}, and they are related to the quantum Bose gas of statistical mechanics. Denote as before $\ell_n= \ell(\sigma_n) = \sum_{k=1}^n m_k(\sigma_n)$ the number of disjoint cycles of a random permutation under the measure $\proba_{\Theta,n}$. We also introduce the generating series of the parameter $\Theta$:
$$g_{\Theta}(z) = \sum_{k=1}^\infty \frac{\theta_k}{k} \,z^k .$$
Note that $\exp(g_{\Theta}(z)) = \sum_{n=0}^\infty h_n(\Theta)\,z^n$ is the generating series of the partition functions of the models.\medskip

If $\Theta$ is the constant sequence equal to $\theta$, then $g_{\Theta}(z) = \theta \log(\frac{1}{1-z})$. On the other hand, in this setting, one can represent $\ell_n$ as a sum of independent Bernoulli variables (Feller coupling):
$$\ell_n = \sum_{j=1}^n \,\mathcal{B}\!\left(\frac{\theta}{\theta+j-1}\right).$$ 
We can then use the discussion of \S\ref{subsec:poissonbernoulli}, and the mod-Poisson convergence of $(\ell_n)_{n \in \N}$ with parameters $\l_n = \sum_{j=1}^n \frac{\theta}{\theta+j-1} \simeq \theta \,\log n$ and limiting residue 
$$
\psi(\xi) =\prod_{j=1}^\infty \left(1+\frac{\theta}{\theta+j-1}\,(\E^{\I\xi}-1)\right)\,\E^{-\frac{\theta}{\theta+j-1}\,(\E^{\I\xi}-1)} = \E^{\gamma\theta(\E^{\I\xi}-1)}\frac{\Gamma(\theta)}{\Gamma(\theta\E^{\I\xi})}.
$$
The last formula relies on the infinite product representation $\frac{1}{\Gamma(z+1)} = \E^{\gamma z}\,\prod_{k=1}^\infty (1+\frac{z}{k})\,\E^{-\frac{z}{k}}$. \medskip

More generally, the article \cite{NZ13} shows that if one has a good understanding of the analytic properties of the generating series $g_{\Theta}(z)$, then one can establish the mod-Poisson convergence of $(\ell_n)_{n \in \N}$, with a limiting residue $\psi(\xi)$ similar to the previous expressions. Thus:

\begin{theorem}[\cite{NZ13}, Lemma 4.1]\label{thm:nz}
Assume that $g_{\Theta}(z)$ is holomorphic on a domain 
$$\Delta_0(r,R,\phi) = \{z \in \C,\,\, |z| < R,\,\,z\neq r,\,\,|\arg(z-r)| > \phi\}$$
with $0<r<R$ and $\phi \in (0,\frac{\pi}{2})$, see Figure \ref{fig:domain}.

\begin{center}
\begin{figure}[ht]
\begin{tikzpicture}[scale=1]
\filldraw[thick,NavyBlue,fill=NavyBlue!20!white] (1.3,0) -- ([shift=(20:2cm)]0,0) arc (20:340:2cm) -- (1.3,0);
\draw[dashed] (0,0) -- (2.5,0);
\draw[<->,dashed] (0,0) -- (40:2cm);
\draw (0.8,1) node {$R$};
\fill (0,0) circle (1.5pt);
\draw (0,-0.25) node {$0$};
\draw (1.3,-0.25) node {$r$};
\fill (1.3,0) circle (1.5pt);
\draw[<->] (1.75,0) arc (0:50:4.5mm);
\draw (2.0,0.3) node {$\phi$};
\end{tikzpicture}
\caption{Domain of analyticity of the generating series $g_{\Theta}(z)$.\label{fig:domain}}
\end{figure}
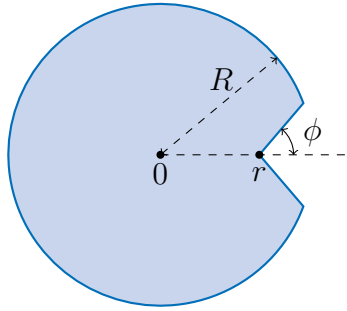
\end{center}
Suppose moreover that the singularity of $g_{\Theta}$ at $z=r$ is logarithmic:
$$g_{\Theta}(z) = \theta \log\left( \frac{1}{1-\frac{z}{r}} \right) + K + O(|z-r|).$$
Then, one has the asymptotics
$$ \esper_{\Theta,n}\!\left[ \E^{z \ell_n} \right] = \E^{ (\theta \log n+K)( \E^{z} - 1 ) } \left( \frac{\Gamma(\theta)}{\Gamma(\theta \E^{z})} + O\!\left( \frac{1}{n}\right) \right),$$
with a uniform remainder on compact subsets of $\C$. Hence, $(\ell_n)_{n \in \N}$ converges mod-Poisson with parameters $\l_n = \theta\log n +K$ and limit 
$$\psi(\xi) = \frac{\Gamma(\theta)}{\Gamma(\theta\E^{\I\xi})}.$$
Moreover, the convergence of residues $\psi_n \to \psi$ happens in every space $\mathscr{C}^r$, with a norm $\|\psi_n-\psi\|_{\mathscr{C}^r}$ which is each time $O(\frac{1}{n})$.
\end{theorem}
\noindent Note that the last part of the theorem (convergence in every space $\mathscr{C}^r$) is an immediate consequence of the estimate of $\esper_{\Theta,n}[\E^{z\ell_n}]$, and of the analyticity of all terms in
$$\esper_{\Theta,n}[\E^{\I\xi \ell_n}]\,\E^{-(\theta\log n +K)(\E^{\I\xi}-1)} = \psi_n(\xi) .$$
On the other hand, in the case of the Ewens measure, one has $K=0$ and $r=1$, and in the case of the uniform measure, one has moreover $\theta=1$.
\bigskip

Because of the infinite product representation of $\frac{\Gamma(\theta)}{\Gamma(\theta\E^{\I\xi})}$, one can restate Theorem \ref{thm:nz} as follows. If one looks at a generalized weighted measure with a generating series $g_{\Theta}(z)$ that satisfies the hypotheses of Theorem \ref{thm:nz}, then
$$\esper[\E^{\I\xi \ell_n}] = \E^{(\theta(\log n+\gamma) + K)(\E^{\I\xi}-1)}\,\left(\prod_{j=1}^\infty \left(1+\frac{\theta(\E^{\I\xi}-1)}{\theta+j-1}\right)\E^{-\frac{\theta(\E^{\I\xi}-1)}{\theta+j-1}} \, + O\!\left(\frac{1}{n}\right)\right).$$
Therefore, in this setting, one can deal with the asymptotics of $(\ell_n)_{n \in \N}$ in the same way as in \S\ref{subsec:poissonbernoulli}, but with parameters $\l_n = \theta H_n + K$, and
\begin{equation}
\frakp_1(\Theta)=0\qquad;\qquad\frakp_{k \geq 2}(\Theta) = \sum_{j=1}^\infty \left(\frac{\theta}{\theta+j-1}\right)^{\!k}. \tag{$\Theta$}\label{eq:thetaspec}
\end{equation}
Thus, one has the following:

\begin{theorem}\label{thm:approxschemecycles}
We consider a generalized weighted measure $\proba_{\Theta,n}$, with a generating series $g_{\Theta}(z)$ which satisfies the hypotheses of Theorem \ref{thm:nz}. For any $r \geq 1$, we introduce the scheme of approximation $(\sigma_n^{(r)})_{n \in \N}$, defined by the Fourier transforms
$$\widehat{\sigma_n^{(r)}}(\xi) = \E^{(\theta H_n + K)(\E^{\I \xi}-1)}\,\left(\sum_{k=0}^r \frake_k(\Theta)\,(\E^{\I\xi}-1)^k\right),$$
where $\Theta$ is the specialisation of $\Sym$ specified by Equation \eqref{eq:thetaspec}. Note that the residue is equal to $1$ for $r=1$ (basic scheme of approximation). One has the asymptotics:
\begin{align*}
\dloc(\ell_n,\sigma_n^{(r)}) &= \frac{|\frake_{r+1}(\Theta)\,G_{r+1}(z_{r+2})|}{\sqrt{2\pi}\,(\theta \log n)^{\frac{r}{2}+1}} + o\!\left(\frac{1}{(\log n)^{\frac{r}{2}+1}}\right);\\
\dkol(\ell_n,\sigma_n^{(r)}) &= \frac{|\frake_{r+1}(\Theta)\,G_{r}(z_{r+1})|}{\sqrt{2\pi}\,(\theta \log n)^{\frac{r+1}{2}}} + o\!\left(\frac{1}{(\log n)^{\frac{r+1}{2}}}\right);\\
\dtv(\ell_n,\sigma_n^{(r)}) &= \frac{\int_{\R}|\frake_{r+1}(\Theta)\,G_{r+1}(\alpha)|\,d\alpha}{\sqrt{2\pi}\,(\theta \log n)^{\frac{r+1}{2}}} + o\!\left(\frac{1}{(\log n)^{\frac{r+1}{2}}}\right).
\end{align*}
\end{theorem}
\bigskip

\subsection{Mod-Poisson convergence and algebraico-logarithmic singularities} \label{subsec:combinatorialexamples} The discussion of \S\ref{subsec:disjointcycles} can be extended to any statistic of a random combinatorial object, whose double generating series admits a certain form of singularity. Suppose given a combinatorial class, that is to say a sequence $\mathfrak{C} = (\mathfrak{C}_n)_{n \in \N}$ of finite sets. We write $\card\,\mathfrak{C}_n = |\mathfrak{C}_n|$. In this paragraph, we shall in particular study the following two examples, \emph{cf.} \cite[\S5, Example 2]{Hwa96} and \cite[p. 449 and p. 671]{FSe09}:\vspace{2mm}
\begin{enumerate}
	\item $\mathfrak{F}_n = \mathcal{F}(\lle 1,n\rre ,\lle 1,n\rre)$ is the set of maps from the finite set $\lle 1,n\rre$ to itself. It has cardinality $n^n$.\vspace{2mm}
	\item $\mathfrak{P}_n = (\mathbb{F}_q[t])_n$ is the set of monic polynomials of degree $n$ with coefficients in the finite field $\mathbb{F}_q$, where $q=p^e$ is some prime power. It has cardinality $q^n$.\vspace{2mm}
\end{enumerate}
This is mainly to fix the ideas, and the techniques hereafter can be applied to many other examples from the aforementioned sources. We consider a statistic $X : \bigsqcup_{n \in \N} \mathfrak{C}_n \to \N$, and we denote by $X_n$ the random variable obtained by evaluating $X$ on an element of $\mathfrak{C}_n$ chosen uniformly at random. We introduce the bivariate generating function
\begin{align*}
F(z,w) &= \sum_{n=0}^\infty \frac{z^n}{n!} \sum_{c \in \mathfrak{C}_n} w^{X(c)} = \sum_{n=0}^\infty \frac{z^n\,(\card\,\mathfrak{C}_n)}{n!} \,\esper[w^{X_n}]\\
\text{or }F(z,w)&=\sum_{n=0}^\infty z^n \sum_{c \in \mathfrak{C}_n} w^{X(c)} = \sum_{n=0}^\infty z^n\,(\card\,\mathfrak{C}_n) \,\esper[w^{X_n}].
\end{align*}
Suppose that for any $w$, $F(z,w)$ is holomorphic on a domain $\Delta_0(r(w),R(w),\phi)$, and admits an \emph{algebraico-logarithmic singularity} at $z=r(w)$, that writes as
$$F(z,w)=K(w)\,\left(\frac{1}{1-\frac{z}{r(w)}}\right)^{\!\alpha(w)}\,\left(\log \frac{1}{1-\frac{z}{r(w)}} \right)^{\beta(w)} (1+o(1)).$$
Then, the coefficient $c_n(w)$ of $z^n$ in $F(z,w)$ has for asymptotics:
$$c_n(w) = \frac{K(w)}{\Gamma(\alpha(w))}\,(r(w))^{-n}\,n^{\alpha(w)-1}\,(\log n)^{\beta(w)}\,\left(1+O\!\left(\frac{1}{\log n}\right)\right),$$
see \cite{FO90}. Moreover, the $O(\frac{1}{\log n})$ is actually a $O(\frac{1}{n})$ if $\beta(w)=0$, or if $\alpha(w)=0$ and $\beta(w)=1$; \emph{cf.} \cite{Hwa98}. If $w$ stays in a sufficiently small compact subset of $\C$, then one can take a uniform constant in the $O(\cdot)$, which leads to an asymptotic formula for the expectations $\esper[w^{X_n}]$:
$$\esper[w^{X_n}] = \frac{c_n(w)}{c_n(1)} = \frac{K(w)}{K(1)}\,\frac{\Gamma(\alpha(1))}{\Gamma(\alpha(w))}\left(\frac{r(1)}{r(w)}\right)^{\!n}\,n^{\alpha(w)-\alpha(1)}\,(\log n)^{\beta(w)-\beta(1)}\,\left(1+O\!\left(\frac{1}{\log n}\right)\right),$$
again with a smaller remainder $O(\frac{1}{n})$ if $\beta(w)=0$ for all $w$. In this paragraph, we shall concentrate on this particular case; the case $\beta(w)\neq 0$ is more commonly observed in probabilistic number theory, see our next paragraph \ref{subsec:erdoskac}.
\medskip

\begin{example}
If $f \in \mathfrak{F}_n=\mathcal{F}(\lle 1,n\rre ,\lle 1,n\rre)$, the \emph{functional graph} corresponding to $f$ is the directed graph $G(f)$ with vertex set $\lle 1,n\rre$, and with edges $(k,f(k))$, $k \in \lle 1,n\rre$. Since $f$ is a map, $G(f)$ is a disjoint union of cycles on which trees are grafted, see \cite{FO90b} and Figure \ref{fig:functionalgraph}.
We denote $X(f)$ the number of connected components of the functional graph $G(f)$. This generalizes the notion of cycle of a permutation. If $\mathfrak{T}=(\mathfrak{T}_n)_{n\in \N}$ is the combinatorial class of unordered rooted labeled trees, recall that its generating function $T(z) = \sum_{n = 1}^\infty \frac{z^n}{n!}\,|\mathfrak{T}_n|$ is the solution of $T(z)= z\E^{T(z)}$, because a rooted tree is constructed recursively by taking a node and connecting to it a set of trees (also, Cayley's theorem gives $\card\,\mathfrak{T}_n = n^{n-1}$). 
\begin{center}
\begin{figure}[ht]
\begin{tikzpicture}[scale=1]
\draw (0,0) node (1) {$1$};
\draw (0,2) node (3) {$3$};
\draw (2,2) node (4) {$4$};
\draw (2,0) node (7) {$7$};
\draw (-1,-1) node (5) {$5$};
\draw (3,3) node (8) {$8$};
\draw (4,3) node (6) {$6$};
\draw (3,4) node (2) {$2$};
\draw [->] (1.north) -- (3.south);
\draw [->] (3.east) -- (4.west);
\draw [->] (4.south) -- (7.north);
\draw [->] (7.west) -- (1.east);
\draw [->] (5.north east) -- (1.south west);
\draw [->] (2.south) -- (8.north);
\draw [->] (6.west) -- (8.east);
\draw [->] (8.south west) -- (4.north east);
\end{tikzpicture}
\caption{The functional graph associated to the function $f : \lle 1,8 \rre \to \lle 1,8\rre$ with word $38471814$, here with one connected component.\label{fig:functionalgraph}}
\end{figure}
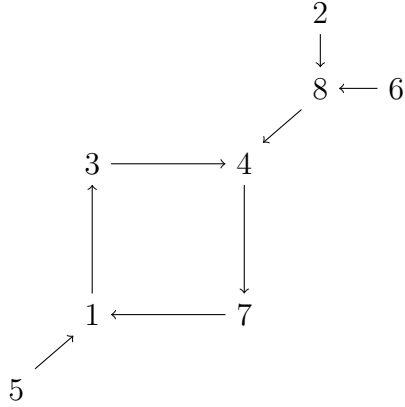
\end{center}

The class $\mathfrak{C}=(\mathfrak{C}_n)_{n\in \N}$ of cycles on such trees has generating series
\begin{align*}
C(z) &= \sum_{n=0}^\infty \frac{z^n}{n!}\,|\mathfrak{C}_n| \\
 &= \sum_{n=0}^\infty \sum_{k_1+\cdots+k_r = n} \frac{|\mathfrak{T}_{k_1}|\,|\mathfrak{T}_{k_2}|\cdots |\mathfrak{T}_{k_r}|}{r\,n!}\,\binom{n}{k_1,\ldots,k_r}\,z^n\\
 &= \sum_{r=1}^\infty \frac{1}{r} \sum_{k_1,\ldots,k_r \geq 1} \frac{|\mathfrak{T}_{k_1}| \cdots |\mathfrak{T}_{k_r}|}{(k_1)!\cdots (k_r)!}\,z^{k_1+\cdots+k_r} \\
 & = \sum_{r=1}^\infty \frac{(T(z))^r}{r} = \log \frac{1}{1-T(z)}.
\end{align*}
On the second line of this computation, the multinomial coefficient comes from the choice of the repartition of the integers of $\lle 1,n\rre$ into the $r$ different trees, and the factor $\frac{1}{r}$ comes from the fact that $r$ distinct cyclic permutations of the choices yield the same cycle of trees. Finally, a functional graph is a set of cycles of trees, so,
\begin{align*}
F(z) = \sum_{n=0}^\infty \frac{z^n}{n!}\,|\mathfrak{F}_n| = \exp(C(z)) = \frac{1}{1-T(z)}.
\end{align*}
Moreover, counting connected components of functional graphs amounts to count a factor $w$ for each cycle, hence,
$$F(z,w) = \sum_{n=0}^\infty \frac{z^n}{n!}\sum_{f \in \mathfrak{F}_n} w^{X(f)} = \exp(wC(z)) = \left(\frac{1}{1-T(z)}\right)^w.$$
The generating series of rooted trees is classically known to have radius of convergence $\frac{1}{\E}$, and a square-root type singularity at $z=\frac{1}{\E}$, see \cite[Formula (11)]{FO90b}:
$$T(z) = 1-\sqrt{2(1-\E z)} + O(1-\E z).$$
Therefore, 
\begin{align*}
F(z,w) &=  2^{-\frac{w}{2}}\,\left(\frac{1}{1-\E z}\right)^{\frac{w}{2}}\,(1+o(1)),
\end{align*}
and the previous discussion applies with $K(w)=2^{-\frac{w}{2}}$, $\alpha(w)=\frac{w}{2}$ and $\beta(w)=0$. Consequently,
$$\esper[w^{X_n}]=  \E^{\frac{1}{2}\log \frac{n}{2}\,(w-1)}\,\frac{\Gamma(\frac{1}{2})}{\Gamma(\frac{w}{2})} \left(1+O\!\left(\frac{1}{n}\right)\right),$$
with a remainder that is uniform if $w$ remains on the unit circle. We shall next interpret this result as a mod-Poisson convergence, and deduce from it the construction of approximation schemes.
\end{example} 
\medskip

\begin{example}
If $P \in \mathfrak{P}_n = (\mathbb{F}_q[t])_n$, then it writes uniquely as a product of monic irreducible polynomials, each of these monic irreducible polynomials appearing with a certain multiplicity. In terms of generating series, introducing the combinatorial class $\mathfrak{I} = (\mathfrak{I}_{n})_{n\in \N}$ of monic irreducibles, this can be rewritten as:
\begin{align*}
&P(z)=\sum_{P \in \mathfrak{P}} z^{\deg P}  = \prod_{P \in \mathfrak{I}} (1+z^{\deg P} + z^{2\deg P} + \cdots ) =\prod_{P \in \mathfrak{I}} \frac{1}{1-z^{\deg P}}\\
&=\exp\left(\sum_{P \in \mathfrak{I}} \log\left(\frac{1}{1-z^{\deg P}}\right)\right)=\exp\left(\sum_{k=1}^\infty \sum_{P \in \mathfrak{I}} \frac{z^{k\deg P}}{k}\right) = \exp\left(\sum_{k=1}^\infty \frac{I(z^k)}{k}\right),
\end{align*}
where $P(z)=\sum_{n \geq 0} |\mathfrak{P}_n|\,z^n$ and $I(z)=\sum_{n \geq 1} |\mathfrak{I}_n|\,z^n$. Similarly, if $Y(P)$ and $Z(P)$ are respectively the number of distinct irreducible factors and the number of irreducible factors counted with multiplicity for $P$, then
\begin{align*}
P_Y(z,w) &= \sum_{P \in \mathfrak{P}} w^{Y(P)}\,z^{\deg P} = \prod_{P \in \mathfrak{I}} \frac{1+(w-1)z^{\deg P}}{1-z^{\deg P}} \\
&=\exp\left(\sum_{k=1}^\infty \sum_{P \in \mathfrak{I}}\frac{z^{k\deg P}}{k} -  \frac{((1-w)z^{\deg P})^k}{k} \right) = \exp\left(\sum_{k=1}^\infty \frac{I(z^k)\,(1-(1-w)^k)}{k} \right)\end{align*}
and
\begin{align*}
P_Z(z,w) &= \sum_{P \in \mathfrak{P}} w^{Z(P)}\,z^{\deg P} = \prod_{P \in \mathfrak{I}} \frac{1}{1-wz^{\deg P}}\\
&=\exp\left(\sum_{k=1}^\infty \sum_{P \in \mathfrak{I}}\frac{w^kz^{k\deg P}}{k}  \right) = \exp\left(\sum_{k=1}^\infty \frac{w^k\, I(z^k)}{k}  \right).
\end{align*}
Note that $P(z)$ is simply equal to $\sum_{n=0}^\infty q^nz^n = \frac{1}{1-qz}$. On the other hand, the number $|\mathfrak{I}_n|$ of irreducible monic polynomials of degree $n$ is given by Gauss' formula
$$ |\mathfrak{I}_n| = \frac{1}{n}\sum_{d\mathrel{|} n} \,\mu\!\left(\frac{n}{d}\right)q^d,$$
as can be seen by gathering the elements of $\mathbb{F}_{q^n}$ according to their irreducible polynomials over $\mathbb{F}_q$ (here, $\mu$ is the arithmetic M\"obius function). Consequently,
$$
I(z) = \sum_{n \geq 1}|\mathfrak{I}_n|\,z^n =  \sum_{k \geq 1} \sum_{d \geq 1} \mu(k)\,\frac{q^d\,z^{dk}}{dk} = \sum_{k \geq 1} \frac{\mu(k)}{k} \, \log\left(\frac{1}{1-qz^k}\right).
$$
We split $I(z)$ into two parts: the first term $k=1$ and the remainder 
$$R(z) = \sum_{k \geq 2} \frac{\mu(k)}{k} \, \log\left(\frac{1}{1-qz^k}\right).$$ 
The remainder $R(z)$ is an analytic function in $z$ on the open disk of radius $q^{-1/2}$, whereas the first term $\log \frac{1}{1-qz}$ is analytic on the smaller open disk of radius $q^{-1}$. It follows that
\begin{align*}
P_Y(z,w) &= \left(\frac{1}{1-qz}\right)^w\, \exp\left(wR(q^{-1})+\sum_{k\geq 2} \frac{I(q^{-k})\,(1-(1-w)^k)}{k} \right)\,(1+o(1))\\
P_Z(z,w) &= \left(\frac{1}{1-qz}\right)^w\,\exp\left(wR(q^{-1})+\sum_{k\geq 2} \frac{I(q^{-k})\,w^k}{k} \right)\,(1+o(1))
\end{align*}
in the neighborhood of the singularity $z=q^{-1}$. These algebraic singularities with exponents $\alpha(w)=w$ and $\beta(w)=0$ lead to the asymptotic formulas
\begin{align*}
\esper[w^{Y_n}] & = \E^{(\log n + R(q^{-1}))(w-1)}\,\frac{1}{\Gamma(w)}\exp\left(\sum_{k \geq 2} \frac{(-1)^{k-1}\,I(q^{-k})}{k}\,(w-1)^k\right)\left(1+O\!\left(\frac{1}{n}\right)\right);\\
\esper[w^{Z_n}] & = \E^{(\log n + R(q^{-1}))(w-1)}\,\frac{1}{\Gamma(w)}\exp\left(\sum_{k \geq 2} \frac{I(q^{-k})}{k}\,(w^k-1)\right)\left(1+O\!\left(\frac{1}{n}\right)\right).
\end{align*}
\end{example}

Thus, if a statistic $X$ of a random combinatorial object has a double generating series that admits an algebraico-logarithmic singularity, then in many cases one can deduce from it the mod-Poisson convergence of $(X_n)_{n \in \N}$. We restate hereafter the previous computations in this framework. We note $\varphi(D)$ the Euler function (number of integers in $\lle 1,D\rre$ that are coprime with $D$), which satisfies the identity $\varphi(D) = \sum_{m \mathrel{|} D} \mu(m)\,\frac{D}{m}$; and 
$$S(z) = \sum_{k \geq 2} \frac{\varphi(k)}{k} \, \log\left(\frac{1}{1-qz^k}\right),$$
which is analytic on the open disk of radius $q^{-1/2}$. 

\begin{theorem}\label{thm:combinatorialexamples}
 Let $X_n$ be the number of components of a random map in $\mathcal{F}(\lle 1,n\rre , \lle 1,n\rre)$, and $Y_n$ and $Z_n$ be the numbers of irreducible factors of a random monic polynomial in $(\mathbb{F}_q[t])_{n \in \N}$, counted respectively without and with multiplicity. The sequences of random variables $(X_n)_{n\in \N}$, $(Y_n)_{n\in \N}$ and $(Z_n)_{n\in \N}$ converge mod-Poisson with parameters
 \begin{align*}
 \lambda_n^X &= \frac{1}{2}\left(\log 2n+\gamma\right);\\
 \lambda_n^Y &= \log n + R(q^{-1}) + \gamma;\\
 \lambda_n^Z &= \log n + S(q^{-1}) + \gamma.
 \end{align*}
 and limiting functions $\psi^{X/Y/Z}(\xi) = \sum_{k=0}^\infty \frake_k(X/Y/Z)\,(\E^{\I\xi}-1)^k$, where the parameters $\frake_k(X)$, $\frake_k(Y)$ and $\frake_k(Z)$ correspond to the specialisations of $\Sym$:
 \begin{align*}
 \frakp_{k\geq 2}(X) &= \sum_{n\geq 1} \frac{1}{(2n-1)^k}\qquad\qquad\qquad\qquad\,\,;\qquad \frakp_1(X)=0\quad; \\
 \frakp_{k\geq 2}(Y) &= \zeta(k) + I(q^{-k})\qquad\qquad\qquad\qquad\,\,\,;\qquad \frakp_1(Y)=0\quad;\\
 \frakp_{k\geq 2}(Z)&= \zeta(k) + (-1)^{k-1}\,\sum_{n\geq 1}\frac{|\mathfrak{I}_n|}{(q^n-1)^k}\qquad;\qquad \frakp_1(Z)=0.
 \end{align*}
 Moreover, in each case, one has $\psi_n(\xi)-\psi(\xi) = O(\frac{1}{n})$ uniformly on the cycle. 
 \end{theorem} 

\begin{proof}
Each time we set $w=\E^{\I\xi}$ in the previous calculations, and we also rewrite the residue so that it does not involve a power of $w-1$ equal to $1$.\vspace{2mm}

\begin{enumerate}
	\item Number of connected components of a random map. One can write the infinite products
	\begin{align*}
	\frac{\Gamma(\frac{1}{2})}{\Gamma(\frac{w}{2})} & = \E^{\frac{\gamma}{2}\,(w-1)}\,\prod_{n=1}^\infty \frac{1+\frac{\frac{w}{2}-1}{n}}{1-\frac{1}{2n}}\,\E^{-\frac{w-1}{2n}} =\E^{\frac{\gamma}{2}\,(w-1)}\,\prod_{n=1}^\infty \left(1+\frac{w-1}{2n-1}\right)\,\E^{-\frac{w-1}{2n}} \\ 
	&=\E^{(\frac{\gamma}{2}+\log 2)\,(w-1)}\,\prod_{n=1}^\infty \left(1+\frac{w-1}{2n-1}\right)\,\E^{-\frac{w-1}{2n-1}}
	\end{align*}
	by using the identity $\sum_{n=1}^\infty \frac{1}{2n-1}-\frac{1}{2n} = \log 2$. In the last infinite product, one recognizes the same form as in \S\ref{subsec:poissonbernoulli}, with parameters $p_j = \frac{1}{2j-1}$.
	\vspace{2mm}
	\item Number of distinct irreducible factors of a random monic polynomial. One rewrites the asymptotic formula for $\esper[w^{Y_n}]$:
	\begin{align*}
	&\E^{(\log n + R(q^{-1}) + \gamma)(w-1)}\left(\prod_{n=1}^{\infty}\left(1+\frac{w-1}{n}\right)\,\E^{-\frac{w-1}{n}}\right) \left(\prod_{k\geq 2} \frac{(-1)^{k-1}\,I(q^{-k})}{k}\,(w-1)^k \right) \\
	&=\E^{(\log n + R(q^{-1}) + \gamma)(w-1)}\,\left(\prod_{k\geq 2} \frac{(-1)^{k-1}\,(w-1)^k}{k}\left(\sum_{n=1}^{\infty}\frac{1}{n^k}+I(q^{-k})\right)\right).
	\end{align*}
	Again we recognize the formula $\exp\left(\sum_{k = 1}^\infty \frac{(-1)^{k-1}(w-1)^k}{k}\,\frakp_k\right) = \sum_{k=0}^\infty \frake_k\,(w-1)^k$.
	\vspace{2mm}
	\item Number of irreducible factors of a random monic polynomial, counted with multiplicity. The factor $\Gamma(w)^{-1}$ is dealt with exactly as in the previous case, so we only have to deal with $\exp(\sum_{k\geq 2}\frac{I(q^{-k})}{k}\,(w^k-1))$. However,
	\begin{align*}
	\sum_{k \geq 2}\frac{I(q^{-k})}{k}\,(w^k-1) &= \sum_{k \geq 2}\sum_{l=1}^k \frac{I(q^{-k})}{k}\,\binom{k}{l}\,(w-1)^l \\
	&=\left(\sum_{k\geq 2}I(q^{-k})\right)(w-1) + \sum_{l=2}^\infty \left(\sum_{k \geq l} \binom{k}{l}\,\frac{I(q^{-k})}{k}\right)(w-1)^l.
	\end{align*}
	The coefficient of $(w-1)$ can be rewritten as follows:
	\begin{align*}
	\sum_{k \geq 2} I(q^{-k}) &= \sum_{\substack{k\geq 2\\m\geq 1}} \frac{\mu(m)}{m} \log\left(\frac{1}{1-q^{1-km}}\right) = \sum_{D=km \geq 2} \log\left(\frac{1}{1-q^{1-D}}\right)\,\sum_{\substack{m \mathrel{|} D \\ m \neq D}} \frac{\mu(m)}{m}\\
	&= \sum_{D \geq 2} \frac{\varphi(D)-\mu(D)}{D}\,\log\left(\frac{1}{1-q^{1-D}}\right) = S(q^{-1}) - R(q^{-1}).
	\end{align*}
	On the other hand, if $l \geq 2$, then the coefficient of $(w-1)^l$ rewrites as
	\begin{align*}
	\sum_{k \geq l} \binom{k}{l}\,\frac{I(q^{-k})}{k} &=\frac{1}{l}\sum_{k \geq l} \sum_{m \geq 1} \frac{\mu(m)}{m}\,\binom{k-1}{l-1} \,\log\left(\frac{1}{1-q^{1-km}}\right) \\
	&= \frac{1}{l} \sum_{m,n\geq 1} \frac{\mu(m)}{nm}\,q^n\,\sum_{k\geq l} \binom{k-1}{l-1} q^{-k(mn)} \\
	&=\frac{1}{l} \sum_{D \geq 1} \frac{|\mathfrak{I}_D|}{(q^D-1)^l}
	\end{align*}
	by setting $D=nm$ on the last line. We conclude that the Fourier transform $\esper[w^{Z_n}]$ has asymptotics
	$$\E^{(\log n + S(q^{-1})+\gamma)(w-1)}\,\exp\left(\sum_{l=2}^\infty \left((-1)^{l-1}\,\zeta(l) + \sum_{D\geq 1} \frac{|\mathfrak{I}_D|}{(q^D-1)^l}\right)\frac{(w-1)^l}{l}\right),$$
	and the term of degree $1$ in $(w-1)$ is separated from the other terms.\qedhere
\end{enumerate}
\end{proof}\bigskip

From Theorem \ref{thm:combinatorialexamples}, one can easily construct approximation schemes of the sequences $(X_n)_{n \in \N}$, $(Y_n)_{n \in \N}$ and $(Z_n)_{n \in \N}$, which yield distances that are arbitrary large negative powers of $\log n$. For instance, suppose that one wants to approximate the law of $Y_n$ by a signed measure $\nu_n$, such that $\dtv(\mu_n,\nu_n)=O(\frac{1}{(\log n)^2})$. By Theorem \ref{thm:maindtv}, we can take the derived scheme of the 
approximation scheme of order $r=3$, that is to say the sequence of measures $(\sigma_n)_{n \in \N}$ with Fourier transforms
$$\widehat{\sigma_n}(\xi) = \E^{(\log n + R(q^{-1})+\gamma)(\E^{\I\xi-1})}\,\left(1 - \frac{\zeta(2)+I(q^{-2})}{2}(\E^{\I\xi}-1)^2 + \frac{\zeta(3)+I(q^{-3})}{3}(\E^{\I\xi}-1)^3 \right).$$
One has in this situation
$$\dtv(X_n,\sigma_n) \simeq \frac{1}{(\log n)^2}\left(\frac{(\zeta(2)+I(q^{-2}))^2}{8} + \frac{\zeta(4)+I(q^{-4})}{4}\right)\left(\int_{\R} |G_4(\alpha)|\,\frac{d\alpha}{\sqrt{2\pi}}\right).$$
\begin{remark}
A central limit theorem for the random variables $Y_n$ has been proved by Rhoades in \cite{Rhoa09}. Our results immediately  yield the speed of convergence of this analogue on function fields of the Erdös--Kac central limit theorem.
\end{remark}
\bigskip

\subsection{Number of distinct prime divisors of a random integer} \label{subsec:erdoskac}
As pointed out in the introduction, number theory provides another area where the phenomenon of mod-Poisson convergence is prominent; see in particular the discussion of \cite[\S7.2]{FMN16}. Consider as in Example \ref{ex:omegan} the number $\omega(N_n)$ of distinct prime divisors of a random integer chosen uniformly in $\lle 1,n\rre$. This quantity satisfies the celebrated Erd\"os--Kac central limit theorem:
$$\frac{\omega(N_n) - \log \log n}{\sqrt{\log\log n}} \rightharpoonup \mathcal{N}_{\R}(0,1) ,$$
where $\mathcal{N}_{\R}(0,1)$ denotes a standard one-dimensional Gaussian distribution of mean $0$ and variance $1$ (\emph{cf.} \cite{EK40}). This Gaussian approximation actually comes from a more precise Poissonian approximation, which can be stated in the framework of mod-Poisson convergence:

\begin{theorem}[Rényi--Turán]
Locally uniformly in $z \in \C$,
\begin{align*}
&\esper[\E^{z\omega(N_n)}] \\
&= \E^{(\log\log n + \gamma)(\E^{z}-1)}\,\left(\prod_{n \in \N} \left(1+\frac{\E^{z}-1}{n}\right)\E^{-\frac{\E^{z}-1}{n}} \prod_{p \in \P} \left(1+\frac{\E^{z}-1}{p}\right)\E^{-\frac{\E^{z}-1}{p}} +O\!\left(\frac{1}{\log n}\right)\right).
\end{align*}
\end{theorem}

\noindent This formula first appeared in \cite{RT58}, and it can be obtained by an application of the Selberg--Delange method, see \cite[Chapter II.5]{Ten95}. It is actually stronger than the prime number theorem, as it requires a larger zero-free region of the Riemann zeta function than the line $\Re(s) = 1$. The problem of the rate of convergence in the Erdös--Kac central limit theorem was already studied by Harper in \cite{Har09} for the Kolmogorov distance to the Gaussian approximation; and by Barbour, Kowalski and Nikeghbali in \cite{BKN09} for the three distances to the Poissonian (or signed measure) approximation.\bigskip

Note that the residue $\psi(\xi) = \prod_{n \in \N^*} \left(1+\frac{\E^{\I\xi}-1}{n}\right)\E^{-\frac{\E^{\I\xi}-1}{n}} \prod_{p \in \P} \left(1+\frac{\E^{\I\xi}-1}{p}\right)\E^{-\frac{\E^{\I\xi}-1}{p}}$ can be rewritten as
$$\exp\left(\sum_{k \geq 2} (-1)^{k-1}\,\frac{\frakp_k({\N^*}^{-1} + \P^{-1})}{k}\,(\E^{\I\xi}-1)^k\right),$$
where $\frakp_k({\N^*}^{-1} + \P^{-1}) = \sum_{n=1}^\infty \frac{1}{n^k} + \sum_{p \in \P} \frac{1}{p^k}$. Therefore:

\begin{theorem}\label{thm:poissonerdos}
For any $r \geq 1$, we introduce the scheme of approximation $(\sigma_n^{(r)})_{n \in \N}$, defined by the Fourier transforms
$$\widehat{\sigma_n^{(r)}}(\xi) = \E^{(\log\log n+\gamma)(\E^{\I \xi}-1)}\,\left(\sum_{k=0}^r \frake_k({\N^{*}}^{-1}+\P^{-1})\,(\E^{\I\xi}-1)^k\right),$$
where the parameters $\frake_{k}({\N^{*}}^{-1}+\P^{-1})$ correspond to the specialisation of $\Sym$:
$$\frakp_1({\N^{*}}^{-1}+\P^{-1}) = 0\qquad;\qquad \frakp_{k\geq 2}({\N^{*}}^{-1}+\P^{-1}) = \zeta(k) + \sum_{p \in \P} \frac{1}{p^k}.$$
One has the asymptotic formulas:
\begin{align*}
\dloc(\omega(N_n),\sigma_n^{(r)}) &= \frac{|\frake_{r+1}({\N^{*}}^{-1}+\P^{-1})\,G_{r+1}(z_{r+2})|}{\sqrt{2\pi}\,(\log \log n)^{\frac{r}{2}+1}} + o\!\left(\frac{1}{(\log\log n)^{\frac{r}{2}+1}}\right);\\
\dkol(\omega(N_n),\sigma_n^{(r)}) &= \frac{|\frake_{r+1}({\N^{*}}^{-1}+\P^{-1})\,G_{r}(z_{r+1})|}{\sqrt{2\pi}\,(\log \log n)^{\frac{r+1}{2}}} + o\!\left(\frac{1}{(\log\log n)^{\frac{r+1}{2}}}\right);\\
\dtv(\omega(N_n),\sigma_n^{(r)}) &= \frac{\int_{\R}|\frake_{r+1}({\N^{*}}^{-1}+\P^{-1})\,G_{r+1}(\alpha)|\,d\alpha}{\sqrt{2\pi}\,(\log \log n)^{\frac{r+1}{2}}} + o\!\left(\frac{1}{(\log \log n)^{\frac{r+1}{2}}}\right).
\end{align*}
\end{theorem}
\bigskip

This theorem allows one to quantify how much better the Poisson approximation is in comparison to the Gaussian approximation (Erd\"os--Kac theorem). Indeed, since the case $r=1$ of Theorem \ref{thm:poissonerdos} is the basic scheme of approximation of $\omega(N_n)$ by a Poisson random variable $Y_n$ of parameter $\l_n=\log\log n + \gamma$, one has:
\begin{align*}
&\dkol\!\left(\frac{\omega(N_n)-\log\log n - \gamma}{\sqrt{\log\log n + \gamma}},\, \frac{\mathcal{P}_{(\log\log n +\gamma)} - \log\log n -\gamma}{\sqrt{\log\log n + \gamma}}\right) \\
&= \dkol(\omega(N_n),\mathcal{P}_{(\log\log n +\gamma)}) = O\!\left(\frac{1}{\log\log n}\right).
\end{align*}
On the other hand, the classical Berry--Esseen estimate for sums of identically distributed random variables ensures that 
$$\dkol\!\left(\frac{\mathcal{P}_{(\log\log n +\gamma)} - \log\log n -\gamma}{\sqrt{\log\log n + \gamma}},\,\mathcal{N}_{\R}(0,1)\right) = O\!\left(\frac{1}{\sqrt{\log\log n}}\right).$$
Thus, in terms of Kolmogorov distance, the Poisson approximation of the sequence $(\omega(N_n))_{n\in \N}$ is at distance $O(\frac{1}{\log\log n})$, whereas the Gaussian approximation is at distance $O(\frac{1}{\sqrt{\log\log n}})$. Similar remarks can be made for the statistics of random combinatorial objects previously studied. 
\bigskip
\bigskip

\section{The multi-dimensional case}\label{sec:multidimensional}

In this last section, we extend the theoretical results of Sections \ref{sec:wiener} and \ref{sec:asymptoticdistance} to the multi-dimensional case, that is to say that we are going to approximate the laws of random variables with values in $\Z^{d\geq 2}$. Thus, we consider a sequence of random variables $(X_n)_{n\in \N}$ in $\Z^d$ that converges mod-$\phi$ with parameters $\l_n$, where $\phi$ is an infinitely divisible distribution with minimal lattice $\Z^d$:
$$\mu_n(\xi) = \esper\!\left[\E^{\I\sum_{i=1}^d \xi_i X_{n,i}}\right] = \E^{\l_n\,\phi(\xi)}\,\psi_n(\xi_1,\ldots,\xi_d),$$
with $\lim_{n\to \infty}\psi_n = \psi$ uniformly on the torus. In all the examples that we shall look at, the law with exponent $\phi$ enjoys the factorization property:
$$\phi(\xi) = \phi(\xi_1,\ldots,\xi_d)=\sum_{i=1}^d \phi_i(\xi_i), $$
where the $\phi_j$'s are infinitely divisible laws on $\Z$. In other words, if $Y$ is a reference random variable with law on $\Z^d$ with L\'evy--Khintchine exponent $\phi$, then its coordinates are independent random variables on the lattice $\Z$. This does not mean at all that the $X_n$'s will have independent coordinates. Hence, though we shall see independence of coordinates by looking at the first order asymptotics of $X_n$ (captured by the infinitely divisible law $\phi$), the higher order asymptotics will shed light on the non-independence of the coordinates (this being captured by the limiting residue $\psi(\xi)$, which does not factorize in general).\medskip

As in Section \ref{sec:asymptoticdistance}, we consider a general approximation scheme $(\nu_n)_{n \in \N}$ of $(X_n)_{n\in \N}$ with Fourier transforms
$$\widehat{\nu_n}(\xi) = \E^{\l_n \phi(\xi)}\,\chi_n(\xi_1,\ldots,\xi_d),$$
and $\lim_{n \to \infty}\chi_n = \chi$. The goal will be to compute the asymptotics of 
\begin{align*}
\dloc(\mu_n,\nu_n)&=\sup_{k \in \Z^d} |\mu_{n}(\{k\}) - \nu_n(\{k\})| ;\\
\dtv(\mu_n,\nu_n)&=\sum_{k \in \Z^d} |\mu_{n}(\{k\}) - \nu_n(\{k\})|.
\end{align*}
In dimension $d \geq 2$, there is no interesting analogue of the Kolmogorov distance (at least for discrete measures; see \cite{FMN17b} for the case of continuous measures). In \S\ref{subsec:multiwiener}, we study the multi-dimensional Wiener algebra $\Ac(\T^d)$, and we establish an estimate of norms that is similar to our Theorem \ref{thm:concentrationinequality} (see Theorem \ref{thm:multiconcentrationinequality}). In \S\ref{subsec:multidistances}, we deduce from this theorem the multi-dimensional analogue of the results of Section \ref{sec:asymptoticdistance}. Finally, in \S\ref{subsec:colouredpermutation} and \S\ref{subsec:residueclass}, we study two examples which generalize the ones of Section \ref{sec:oneexample}, and stem respectively from the combinatorics of coloured permutations, and from arithmetic progressions in number theory. 
\medskip

Let us fix a few standard notation relative to the multi-dimensional setting. The coordinates of a parameter $\xi \in \T^d = (\R/2\pi\Z)^d$ (respectively, $k \in \Z^d$) will be denoted $(\xi_1,\ldots,\xi_d)$ (resp., $(k_1,\ldots,k_d)$). If $\alpha \in \N^d$ is a multi-index, we denote $|\alpha| = \sum_{i=1}^d \alpha_i$ and 
$$\partial^\alpha f = \frac{\partial^{|\alpha|} f(\xi_1,\ldots,\xi_d)}{\partial \xi_1^{\alpha_1}\cdots \partial \xi_d^{\alpha_d}}.$$
In the same setting, if $n = |\alpha|$, then the multinomial coefficient $\binom{n}{\alpha}$ is $\frac{n!}{(\alpha_1)!\cdots (\alpha_d)!}$. On the other hand, if $\alpha$ and $\beta$ are two $d$-tuples, we write
$$\scal{\alpha}{\beta} = \sum_{i=1}^d \alpha_i \beta_i.$$
We also write $\alpha \leq \beta$ if $\alpha_i \leq \beta_i$ for all $i \in \lle 1,d\rre$.
\bigskip

\subsection{The multi-dimensional Wiener algebra}\label{subsec:multiwiener}
\begin{definition}
The multi-dimensional Wiener algebra $\Ac = \Ac(\T^d)$ is the algebra of continuous functions on the $d$-torus whose Fourier series converges absolutely. It is a Banach algebra for the norm 
$$ \normA{f} := \sum_{n \in \Z^d} |c_n(f)|,$$
where $c_{n}(f)$ denotes the Fourier coefficient $$c_n(f)=\int_{\theta_1=0}^{2\pi}\cdots \int_{\theta_d=0}^{2\pi} f(\E^{\I \theta_1}, \ldots, \E^{\I \theta_d})\,\E^{-\I \sum_{i=1}^d n_i\theta_i }\,\frac{d\theta}{(2\pi)^d},$$
where $n=(n_1,\ldots,n_d)$.
\end{definition}
As in the case $d=1$, if $\mu$ is a (signed) measure on $\Z^d$, then its total variation norm is equal to $\|\mu\|_{\mathrm{TV}} = \normA{\widehat{\mu}}$. On the other hand, one has the analogue of Proposition \ref{prop:H_in_A}:

\begin{proposition}\label{prop:multiH_in_A}
There is a constant $C_{H,d}$ such that, for any $f \in \Ac(\T^d)$, 
$$\normA{f} \leq C_{H,d} \sup_{|\alpha| \leq K_d} \|\partial^\alpha f\|_{\leb^2},$$
where $K=K_d = \lfloor\frac{d}{2}\rfloor+1$, and $\leb^2 = \leb^2(\T^d)$ is the space of square-integrable functions on the torus, endowed with the norm 
$$\|f\|_{\leb^2} = \sqrt{\int_{\theta_1=0}^{2\pi}\cdots \int_{\theta_d=0}^{2\pi} |f(\E^{\I\theta_1},\ldots,\E^{\I\theta_d})|^2\,\frac{d\theta}{(2\pi)^d}}.$$
\end{proposition}

\begin{proof}
Note that $K_d \geq \frac{d+1}{2}$ for any $d\geq 1$. We use again the Cauchy--Schwarz inequality and the Parseval identity:
\begin{align*}
&\normA{f} = \sum_{n \in \Z^d} |c_n(f)| = \sum_{n \in \Z^d} \left(\frac{1+\sum_{i=1}^d (n_i)^2}{1+\sum_{i=1}^d (n_i)^2}\right)^{\!\frac{K_d}{2}}\, |c_n(f)| \\
&\leq \sqrt{\sum_{n \in \Z^d} \left(\frac{1}{1+\sum_{i=1}^d (n_i)^2}\right)^{\!K_d}} \,\sqrt{\sum_{n \in \Z^d} \left(1+\sum_{i=1}^d (n_i)^2\right)^{\!K_d} |c_n(f)|^2} \\
&\leq \sqrt{\sum_{n\in \Z^d} \left(\frac{1}{1+\sum_{i=1}^d (n_i)^2}\right)^{\!\frac{d+1}{2}}}\,\sqrt{\sum_{\alpha_0+\alpha_1+\cdots+\alpha_d = K_d} \binom{K_d}{\alpha_0,\ldots,\alpha_d} \sum_{n \in \Z^d} |c_n(\partial^{(\alpha_1,\ldots,\alpha_d)}f)|^2}\\
&\leq \sqrt{\sum_{n \in \Z^d} \left(\frac{1}{1+\sum_{i=1}^d (n_i)^2}\right)^{\!\frac{d+1}{2}}\,\sum_{\alpha_0+\alpha_1+\cdots+\alpha_d=K_d} \binom{K_d}{\alpha_0,\ldots,\alpha_d}}\,\,\sup_{|\alpha|\leq K_d} \|\partial^\alpha f\|_{\leb^2},
\end{align*}
hence the result since the series $(C_{H,d})^2$ under the square root sign is finite.
\end{proof}
\bigskip

This inequality leads to the analogue of Theorem \ref{thm:concentrationinequality} in a multi-dimensional setting. We consider two measures $\mu$ and $\nu$ on $\Z^d$ with Fourier transforms
\begin{align*}
\widehat{\mu}(\xi) &= \E^{\l \phi(\xi)}\,\psi(\xi);\\
\widehat{\nu}(\xi)  &= \E^{\l \phi(\xi)}\,\chi(\xi)
\end{align*}
where $\phi(\xi) = \sum_{i=1}^d \phi_i(\xi_i)$. We assume $\phi$ to have moments at least up to order $K_d$, and denote
$$\phi_i(\xi_i) = \I\, m_i \xi_i - \frac{\sigma_i^2\,\xi_i^2}{2}+o(\xi_i^2) $$
the Taylor expansion around $0$ of each law on $\Z$ with L\'evy exponent $\phi_i$. We fix an integer $r\geq \lfloor \frac{d}{2}\rfloor$ such that \vspace{2mm}
\begin{enumerate}
	\item the residues $\psi$ and $\chi$ are $(r+1)$ times continuously differentiable on $\T^d$;\vspace{2mm}
	\item their Taylor expansions at $0$ coincide up to order $r$:
	$$\forall |\alpha|\leq r,\,\,\,\left(\partial^\alpha(\psi-\chi)\right)(0) = 0 .$$
\end{enumerate}
A parameter $\eps \in (0,1)$ being fixed, we also introduce the non-negative quantities
\begin{align*}
\beta_{r+1}(\eps)&= \sup_{|\alpha|=r+1} \sup_{\xi \in [-\eps,\eps]^d} |\partial^{\alpha}(\psi-\chi)(\eps)|; \\
\gamma(\eps) &= \sup_{i \in \lle 1,d\rre}\sup_{\theta \in [-\eps,\eps]} |\phi_i''(\theta)+\sigma_i^2|;\\
M &= -\sup_{i \in \lle 1,d\rre}\sup_{\theta \in [-\pi,\pi]} \left(\frac{\Re(\phi_i(\theta))}{\theta^2}\right).
\end{align*}
The quantity $\beta_{r+1}(\eps)$ allows one to bound any derivative up to order $(r+1)$ of $\psi-\chi$ on $[-\eps,\eps]^d$. Indeed, this is obvious for exponents $\alpha$ with $|\alpha|=r+1$, and otherwise,
\begin{align*}
\left|\partial^{\alpha}(\psi-\chi)(\xi)\right| = \left|\int_{t=0}^1 \frac{d}{dt}(\partial^{\alpha}(\psi-\chi)(t\xi))\,dt \right| \leq \sum_{i=1}^d |\xi_i| \left|\int_{t=0}^1 \frac{\partial}{\partial \alpha_i}(\partial^{\alpha}(\psi-\chi)(t\xi))\,dt\right|,
\end{align*}
which leads by descending induction on $|\alpha|$ to the bound
$$ \left|\partial^{\alpha}(\psi-\chi)(\xi)\right| \leq \left(\sum_{i=1}^d|\xi_i|\right)^{\!r+1-|\alpha|}\,\beta_{r+1}(\eps)\leq (d\eps)^{r+1-|\alpha|}\,\beta_{r+1}(\eps)$$
for any $\xi \in [-\eps,\eps]^d$.

\begin{theorem}\label{thm:multiconcentrationinequality}
Fix $\eps \in (0,1)$ such that $\gamma(\eps) \leq \min_{i \in \lle 1,d\rre} \frac{\sigma_i^2}{2}$. There exist some positive constants $C_1(d,c,\phi)$ and $C_2(d,c,\phi)$, that depend only on $d$, $c$ and $\phi$, such that
\begin{align*}
&\|\widehat{\mu}-\widehat{\nu}\|_{\Ac} \\
&\leq C_1(d,c,\phi)\left(\normA{\psi-\chi}\,\frac{\l^K\,\E^{-\frac{\l M \eps^2}{4}}}{\eps} + \frac{\eps^{r+1}\,\beta_{r+1}(\eps)}{\l^{\frac{d}{4}}}\,\left(\sqrt{\l}+\frac{1}{\eps}\right)^{\!K} \right) \left(1+\frac{C_2(d,c,\phi)}{\sqrt{\l}}\right)
\end{align*}
for any $\eps \in (0,1)$ and any $\l > 0$.
\end{theorem}
\bigskip

Since Proposition \ref{prop:multiH_in_A} involves higher order derivatives than Proposition \ref{prop:H_in_A}, in order to prove Theorem \ref{thm:multiconcentrationinequality}, we shall use a smoother cut-function than in the proof of Theorem \ref{thm:concentrationinequality}. Thus, in the following, $c_{d,\eps}(\xi)=\prod_{i=1}^d c(\frac{\xi_i}{\eps})$, where $c(x)$ is a function of one variable with values in $[0,1]$, that is of class $\mathscr{C}^{K_d}$ on $\R$, with $c(x)=0$ outside $[-1,1]$ and $c(x)=1$ on $[-\frac{1}{2},\frac{1}{2}]$. In particular, $c_{d,\eps}$ (respectively, $1-c_{d,\eps}$) vanishes on $[-\pi,\pi]^d \setminus [-\eps,\eps]^d$ (resp., on $[-\frac{\eps}{2},\frac{\eps}{2}]^d$). \medskip

\begin{lemma}
Under the assumptions of Theorem \ref{thm:multiconcentrationinequality}, there exists a constant $C(d,c,\phi)$ such that
$$\normA{(1-c_{d,\eps})\,\E^{\l \phi}} \leq C(d,c,\phi)\,\frac{\lambda^K}{\eps}\,\E^{-\frac{\lambda M\eps^2}{4}}\,\left(1+O\!\left(\frac{1}{\l}\right)\right),$$
where the constant hidden in the $O(\cdot)$ also depends only on $d$, $c$ and $\phi$.
\end{lemma}

\begin{proof}
We combine Proposition \ref{prop:multiH_in_A} with the rules of differentiation of functions of several variables:
\begin{align*}
&\normA{(1-c_{d,\eps})\,\E^{\l\phi}} \leq C_{H,d}\,\sup_{|\alpha|\leq K_d} \|\partial^\alpha((1-c_{d,\eps})\,\E^{\l\phi})\|_{\leb^2([-\pi,\pi]^d \setminus [-\frac{\eps}{2},\frac{\eps}{2}]^d } \\
&\leq C_{H,d} \sup_{|\alpha| \leq K_d} \sum_{\beta \leq \alpha} \prod_{i=1}^d \binom{\alpha_i}{\beta_i} \|\partial^{\alpha-\beta}(1-c_{d,\eps})\|_{\leb^2}\,\| \partial^{\beta}(\E^{\l \phi})\|_{\leb^\infty([-\pi,\pi]^d \setminus [-\frac{\eps}{2},\frac{\eps}{2}]^d )} \\
& \leq 2^{K_d}\,C_{H,d}\,\left(\sup_{|\gamma|\leq K_d} \|\partial^\gamma (1-c_{d,\eps})\|_{\leb^2}\right) \left(\sup_{|\beta| \leq K_d} \|\partial^{\beta}(\E^{\l \phi})\|_{\leb^\infty([-\pi,\pi]^d \setminus [-\frac{\eps}{2},\frac{\eps}{2}]^d )}\right).
\end{align*}
Since the L\'evy--Khintchine exponent $\phi$ is assumed to have the factorization property, for any multi-index $\beta$,
$$\partial^\beta(\E^{\lambda\phi}) = \prod_{i=1}^d \frac{\partial^{\beta_i}\E^{\lambda\phi_i(\xi_i)}}{\partial \xi_i^{\beta_i}},$$
and 
$$\frac{\partial^n \E^{\lambda\phi_i(\theta)}}{\partial \theta^n} = \E^{\lambda\phi_i(\theta)}\,\sum_{k=1}^n B_{n,k}(\lambda\phi_i'(\theta),\lambda\phi_i''(\theta),\ldots,\lambda\phi_i^{(n-k+1)}(\theta)),$$
where $B_{n,k}(x_1,\ldots,x_{n+k-1})$ is the incomplete Bell polynomial of parameters $n$ and $k$, see \cite{Bell27}. More precisely,
\begin{align*}
B_{n,k} (x_1,\ldots,x_{n+k-1}) &= \!\!\!\!\!\!\sum_{\substack{m_1+m_2+\cdots+m_{n-k+1}=k \\ m_1+2m_2+\cdots + (n-k+1)m_{n-k+1}=n}} \!\!\!\!\!\!\frac{n!}{ \prod_{i=1}^{n-k+1} (i!^{m_i}\,m_i!)}\, (x_1)^{m_1}\cdots (x_{n-k+1})^{m_{n-k+1}} \\
|B_{n,k}(x_1,\ldots,x_{n+k-1})| &\leq \left(\sup_{i \in \lle 1,n+k-1\rre}|x_i|\right)^{\!k}\,S(n,k)
\end{align*}
where $S(n,k)=B_{n,k}(1,1,\ldots,1)$ is the Stirling number of the second kind. As a consequence,
$$\left|\frac{\partial^n \E^{\lambda\phi_i(\theta)}}{\partial\theta^n}\right| \leq |\E^{\lambda\phi_i(\theta)}|\,T_n(\lambda \|\phi_i\|_{\mathscr{C}^{n}(\T)}) \leq |\E^{\lambda\phi_i(\theta)}|\,T_n(\lambda \|\phi\|_{\mathscr{C}^{n}(\T^d)}),$$
where $T_n(x) = \sum_{k=1}^n S(n,k)\,x^k$ is the $n$-th Touchard polynomial, with leading term $S(n,n)\,x^n = x^n$. Thus,
\begin{align*}
|\partial^\beta\E^{\lambda\phi}| &\leq |\E^{\lambda\phi}|\,\prod_{i=1}^d \left|T_{\beta_i}(\lambda \|\phi\|_{\mathscr{C}^{\beta_i}(\T^d)})\right| \\
&\leq |\E^{\lambda\phi}| \,(\l \|\phi\|_{\mathscr{C}^{K}(\T^d)})^{|\beta|}\,\left(1+O\!\left(\frac{1}{\l \|\phi\|_{\mathscr{C}^{K}(\T^d)}}\right)\right)
\end{align*}
if $|\beta|\leq K_d$. Now, if $\xi \notin [-\frac{\eps}{2},\frac{\eps}{2}]^d$, then one of the coordinates $\xi_i$ has modulus larger than $\frac{\eps}{2}$, so $|\E^{\phi(\xi)}| \leq \E^{-\frac{M\eps^2}{4}}$. We conclude that
\begin{align*}
 &\normA{(1-c_d)\,\E^{\l\phi}} \\
 &\leq 2^{K_d}\,C_{H,d}\,\left(\sup_{|\gamma|\leq K_d} \|\partial^\gamma (1-c_{d,\eps})\|_{\leb^2}\right) \,(\l\,\|\phi\|_{\mathscr{C}^{K}(\T^d)})^{K}\,\E^{-\frac{\lambda\,M\eps^2}{4}}\,\left(1+O\!\left(\frac{1}{\l}\right)\right),
\end{align*}
where the constant in the $O(\cdot)$ only depends on $d$ and the exponent $\phi$. Finally,  as soon as $\gamma \neq 0$, $\partial^\gamma (1-c_{d,\eps}) = -\frac{1}{\eps^{|\gamma|}} (\partial^\gamma c_{d,1})(\frac{\xi}{\eps})$, so,
\begin{align*}
\|\partial^\gamma(1-c_{d,\eps})\|_{\leb^2} &=\frac{1}{\eps^{|\gamma|}}\,\sqrt{\int_{[-\eps,\eps]^d} \left|(\partial^\gamma c_{d,1})\left(\frac{\xi}{\eps}\right)\right|^2 \,\frac{d\xi}{(2\pi)^d} } \\
&= \frac{1}{\eps^{|\gamma|-\frac{d}{2}}}\sqrt{\int_{[-1,1]^d} \left|(\partial^\gamma c_{d,1})(\theta)\right|^2 \,\frac{d\theta}{(2\pi)^d} }\\
&= O\!\left(\frac{1}{\eps^{|\gamma|-\frac{d}{2}}}\right).
\end{align*}
We conclude by taking the maximal possible value $|\gamma|=K_d$, which gives $K_d-\frac{d}{2}=1$ if $d$ is even, and $\frac{1}{2}$ if $d$ is odd.
\end{proof}
\medskip

\begin{lemma}
Under the assumptions of Theorem \ref{thm:multiconcentrationinequality}, there exists another constant $C(d,c,\phi)$ such that
$$ \normA{c_{d,\eps}\,(\psi-\chi)\,\E^{\l\phi}} \leq C(d,c,\phi)\,\frac{\eps^{r+1}\,\beta_{r+1}(\eps)}{\l^{\frac{d}{4}}}\,\left(\sqrt{\l}+\frac{1}{\eps}\right)^{\!K}\,\left(1+O\!\left(\frac{1}{\sqrt{\l}}\right)\right).$$
Again, the constant hidden in the $O(\cdot)$ depends only on $d$, $c$ and $\phi$.
\end{lemma}

\begin{proof}
The hypothesis on $\gamma(\eps)$ ensures that, for any index $i \in \lle 1,d\rre$ and any $\theta \in [-\eps,\eps]$,
$\Re(\phi_i''(\theta)) \leq -\sigma_i^2/2$ and $|\phi_i''(\theta)| \leq 2\sigma_i^2$. On the other hand, as in the case $d=1$ dealt with by Lemma \ref{lem:concentration2}, we can shift phases by $S(\xi) = \I\sum_{i=1}^d \xi_i\lfloor \l m_i\rfloor$:
\begin{align*}
&\normA{c_{d,\eps}\,(\psi-\chi)\,\E^{\l \phi}} = \normA{c_{d,\eps}\,(\psi-\chi)\,\E^{\l \phi-S}}\\
&\leq C_{H,d}\,\sup_{|\alpha|\leq K_d} \|\partial^\alpha(c_{d,\eps}\,(\psi-\chi)\,\E^{\l \phi-S})\|_{\leb^2([-\eps,\eps]^d)} \\
&\leq C_{H,d}\,\sup_{|\alpha|\leq K_d} \left(\sum_{\beta+\gamma+\delta = \alpha} \prod_{i=1}^d \binom{\alpha_i}{\beta_i,\gamma_i,\delta_i}\,\|\partial^{\beta}c_{d,\eps}\,\partial^{\gamma}(\psi-\chi)\|_{\leb^\infty([-\eps,\eps]^d)} \|\partial^\delta\E^{\l \phi -S}\|_{\leb^2([-\eps,\eps]^d)}\right).
\end{align*}
On the last line, we can bound $\|\partial^{\beta}c_{d,\eps}\|_{\leb^\infty([-\eps,\eps]^d)}$ by $C(d,c)\,\eps^{-|\beta|}$, and $\|\partial^{\gamma}(\psi-\chi)\|_{\leb^\infty([-\eps,\eps]^d)}$ by $\beta_{r+1}(\eps)\,(d\eps)^{r+1-|\gamma|}$. Thus,
\begin{align*}
&\normA{c_{d,\eps}\,(\psi-\chi)\,\E^{\l \phi}} \\
&\leq C(d,c)\,\eps^{r+1}\,\beta_{r+1}(\eps) \left(\sup_{|\alpha| \leq K_d}\sum_{\rho+\delta=\alpha} \prod_{i=1}^d \binom{\alpha_i}{\rho_i,\delta_i}\,\eps^{-|\rho|}\,\|\partial^\delta\E^{\l \phi -S}\|_{\leb^2([-\eps,\eps]^d)}\right)
\end{align*}
where $C(d,c)$ is some positive constant depending only on $d$ and the cut function $c$.\bigskip

To bound the norms $\|\partial^\delta\E^{\l \phi -S}\|_{\leb^2([-\eps,\eps]^d)}$, we use the same combinatorics of Bell polynomials as in the previous lemma. Specifically, we have
\begin{align*}
\left|\frac{\partial^n(\E^{\l\phi_i(\theta)-\I\lfloor \l m_i\rfloor\theta})}{\partial\theta^n} \right| &\leq \left|\E^{\l \phi_i(\theta)}\right|\,\left|\sum_{k=1}^n B_{n,k}(\l \phi_i'(\theta)-\I\lfloor \l m_i\rfloor,\l \phi_i''(\theta),\ldots,\l \phi_i^{(n-k+1)}(\theta))\right|,
\end{align*}
and the variables of the incomplete Bell polynomials that appear are bounded as follows:
\begin{align*}
|\l\phi_i'(\theta) -\I\lfloor \l m_i \rfloor| &\leq \l |\phi_i'(\theta)-\phi_i'(0)| + |\l m_i - \lfloor \l m_i \rfloor| \\
&\leq \l\,\|\phi_i''\|_{\leb^\infty([-\eps,\eps])}\,|\theta| + 1 \leq 2\l\,\sigma_i^2\,|\theta|+1;\\
|\l \phi_i^{(r \geq 2)}(\theta)| &\leq \l\,\|\phi_i^{(r)}\|_\infty \leq \l\,\|\phi\|_{\mathscr{C}^n(\T^d)},
\end{align*}
Rewriting the complete Bell polynomial $B_n=\sum_{k=1}^n B_{n,k}$ as a sum over integer partitions of size $n$, we thus obtain
\begin{align*}
&\left|\frac{\partial^n(\E^{\l\phi_i(\theta)-\I\lfloor \l m_i\rfloor\theta})}{\partial\theta^n} \right| \\
&\leq \E^{-\frac{\l\,\sigma_i^2\xi_i^2}{4}}\,\sum_{\substack{L\text{ partition}\\\text{of size }n}} \frac{n!}{\prod_{i\geq 1}(i!^{m_i(L)} \,m_i(L)!)}\, \left(2\l\,\sigma_i^2\,|\theta|+1\right)^{m_1(L)}\,\left(\l\,\|\phi\|_{\mathscr{C}^n(\T^d)}\right)^{\ell(L)-m_1(L)}.
\end{align*}
Consequently,
\begin{align*}
&\sqrt{\int_{[-\eps,\eps]} \left|\frac{\partial^n(\E^{\l\phi_i(\theta)-\I\lfloor \l m_i\rfloor\theta})}{\partial\theta^n} \right|^2 \,\frac{d\theta}{2\pi}} \\
&\leq \sum_{\substack{L\text{ partition}\\\text{of size }n}} \frac{n!\,(\l\,\|\phi\|_{\mathscr{C}^n(\T^d)})^{\ell(L)-m_1(L)}}{\prod_{i\geq 1}(i!^{m_i(L)} \,m_i(L)!)} \sqrt{\int_{[-\eps,\eps]} \E^{-\frac{\l\,\sigma_i^2\xi_i^2}{2}}\, \left(2\l\,\sigma_i^2\,|\theta|+1\right)^{2m_1(L)}\frac{d\theta}{2\pi}},
\end{align*}
and each integral under the square root sign is equal to 
\begin{align*}
&\sqrt{\frac{1}{\sqrt{2\pi\l\, \sigma_i^2}} \sum_{k=0}^{2m_1(L)}\binom{2m_1(L)}{k} (2\sqrt{\l}\,\sigma_i)^k \int_{\R} \E^{-\frac{x^2}{2}}\, |x|^k\,\frac{dx}{\sqrt{2\pi}}}\\
&=\frac{2^{m_1(L)}\,\sqrt{(2m_1(L)-1)!!}}{(2\pi)^{1/4}}\,(\l\,\sigma_i^2)^{\frac{m_1(L)}{2}-\frac{1}{4}}\,\left(1+O\!\left(\frac{1}{\sqrt{\l}}\right)\right).
\end{align*}
Thus, the term corresponding to an integer partition $L$ in the previous bound on the norm 
$$\left\|\frac{\partial^n(\E^{\l\phi_i(\theta)-\I\lfloor \l m_i\rfloor\theta})}{\partial\theta^n}\right\|_{\leb^2([-\eps,\eps])}$$
is of order $O(\l^{\ell(L) - \frac{m_1(L)}{2}-\frac{1}{4}})$, where the constant in the $O(\cdot)$ can depend only on $n$ and $\phi$. The integer partitions that maximise the exponent $\ell(L)-\frac{m_1(L)}{2}$ among those of fixed size $n$ are those with parts $1$ or $2$, of the form
$$L=(\underbrace{2,2,\ldots,2}_{j\text{ parts of size }2},\underbrace{1,1,\ldots,1}_{n-2j\text{ parts of size }1}).$$
Therefore,
\begin{align*}
&\left\|\frac{\partial^n(\E^{\l\phi_i(\theta)-\I\lfloor \l m_i\rfloor\theta})}{\partial\theta^n}\right\|_{\leb^2([-\eps,\eps])}\\
&\leq \frac{1}{(2\pi\l\,\sigma_i^2)^{\frac{1}{4}}}\sum_{j=0}^{\lfloor \frac{n}{2}\rfloor} \frac{n!\,(\l\,\|\phi\|_{\mathscr{C}^n(\T^d)})^{j}}{2^j\,j!\,(n-2j)!}\,2^{n-2j}\,\sqrt{(2n-4j-1)!!}\,(\l\,\sigma_i^2)^{\frac{n}{2}-j}\,\left(1+O\!\left(\frac{1}{\sqrt{\l}}\right)\right)\\
&\leq C(n,\phi)\,\l^{\frac{n}{2}-\frac{1}{4}}\,\left(1+O\!\left(\frac{1}{\sqrt{\l}}\right)\right),
\end{align*}
where $C(n,\phi)$ and the constant hidden in the $O(\cdot)$ depend only on $n$ and $\phi$. Finally,
\begin{align*}
&\normA{c_{d,\eps}\,(\psi-\chi)\,\E^{\l \phi}} \\
&\leq C(d,c,\phi)\,\frac{\eps^{r+1}\,\beta_{r+1}(\eps)}{\l^{\frac{d}{4}}}\,\left(1+O\!\left(\frac{1}{\sqrt{\l}}\right)\right)\, \left(\sup_{|\alpha| \leq K_d} \left(\sqrt{\l} + \frac{1}{\eps}\right)^{\!|\alpha|}\right) \\
&\leq C(d,c,\phi)\,\frac{\eps^{r+1}\,\beta_{r+1}(\eps)}{\l^{\frac{d}{4}}}\,\left(\sqrt{\l}+\frac{1}{\eps}\right)^{\!K_d}\,\left(1+O\!\left(\frac{1}{\sqrt{\l}}\right)\right).\qedhere
\end{align*}
\end{proof}\medskip

\begin{proof}[Proof of Theorem \ref{thm:multiconcentrationinequality}]
As in the one dimensional case, it suffices now to write
$$\normA{\widehat{\mu}-\widehat{\nu}} \leq \normA{\psi-\chi}\,\normA{(1-c_{d,\eps})\,\E^{\l \phi}} + \normA{c_{d,\eps}\,(\psi-\chi)\,\E^{\l \phi}}, $$
and to use the two previous lemmas.
\end{proof}
\bigskip

\begin{corollary}\label{cor:multinormestimate}
With the assumptions of Theorem \ref{thm:multiconcentrationinequality}, for any $r \geq \lfloor \frac{d}{2}\rfloor$, there exists a positive constant $C_3(r,d,c,\phi)$ such that, if $\eps>0$ is fixed and $\l \geq 2$ is sufficiently large (chosen according to $\eps$), then
$$\|\widehat{\mu}-\widehat{\nu}\|_{\Ac} \leq C_3(r,d,c,\phi)\, \left(\normA{\psi-\chi}+ \beta_{r+1}(\eps)\right)\,\frac{(\log \l)^{r+1}}{\l^{\frac{r}{2}+\frac{1}{4}}}.$$
\end{corollary}

\begin{proof}
Set $\eps=\frac{t\,\log \l}{\sqrt{\l}}$ with $t>0$. We then have, up to a multiplicative remainder $(1+o(1))$, a bound on the norm $\normA{\widehat{\mu}-\widehat{\nu}}$ equal to the sum of the two terms:
$$
C_1(d,c,\phi)\,\normA{\psi-\chi}\, \frac{\l^{K+\frac{1}{2} - \frac{tM}{4}}}{t\log \l} \qquad\text{and}\qquad
C_1(d,c,\phi)\,\beta_{r+1}(\eps)\,(t\,\log \l)^{r+1}\,\l^{\frac{K}{2}-\frac{r+1}{2}- \frac{d}{4}}.$$
We set $t = \frac{1}{M}(d+2r+2K+4)$, so that the two powers of $\l$ agree. Note that 
 $$\frac{K_d}{2}-\frac{r+1}{2}- \frac{d}{4} = \frac{\frac{d}{2}+1}{2} - \frac{r+1}{2}- \frac{d}{4} = \begin{cases}
 	-\frac{r}{2} &\text{ if $d$ is even},\\
 	-\frac{r}{2}-\frac{1}{4}&\text{ if $d$ is odd}.
 \end{cases}$$
 So, 
 $$\normA{\widehat{\mu}-\widehat{\nu}}\leq \begin{cases}
 	C_3(r,d,c,\phi)\, (\normA{\psi-\chi}+ \beta_{r+1}(\eps))\,\frac{(\log \l)^{r+1}}{\l^{\frac{r}{2}}} &\text{ if $d$ is even},\\
 	C_3(r,d,c,\phi)\, (\normA{\psi-\chi}+ \beta_{r+1}(\eps))\,\frac{(\log \l)^{r+1}}{\l^{\frac{r}{2}+\frac{1}{4}}}&\text{ if $d$ is odd}.
 \end{cases}$$
If $d$ is odd, the claim is proven. Otherwise, notice that if $\mu$ and $\nu$ are two measures satisfying the hypotheses of Theorem \ref{thm:multiconcentrationinequality} in even dimension $d$, then
\begin{align*}
\mu^+(k_1,\ldots,k_d,k_{d+1}) = \mu(k_1,\ldots,k_d)\,\frac{\E^{-\l}\,\l^{k_{d+1}}}{(k_{d+1})!}; \\
\nu^+(k_1,\ldots,k_d,k_{d+1}) = \nu(k_1,\ldots,k_d)\,\frac{\E^{-\l}\,\l^{k_{d+1}}}{(k_{d+1})!}
\end{align*}
are signed measures on $\Z^{d+1}$ that again satisfy the hypotheses of Theorem \ref{thm:multiconcentrationinequality}:
\begin{align*}
\widehat{\mu^+}(\xi_1,\ldots,\xi_{d+1}) = \widehat{\mu}(\xi_1,\ldots,\xi_{d+1}) \,\E^{\l(\E^{\I \xi_{d+1}}-1)} = \E^{\l \Phi(x_1,\ldots,x_{d+1})}\,\psi(\xi_1,\ldots,\xi_{d}); \\
\widehat{\nu^+}(\xi_1,\ldots,\xi_{d+1}) = \widehat{\nu}(\xi_1,\ldots,\xi_{d+1}) \,\E^{\l(\E^{\I \xi_{d+1}}-1)} = \E^{\l \Phi(x_1,\ldots,x_{d+1})}\,\chi(\xi_1,\ldots,\xi_{d}),
\end{align*}
where $\Phi(\xi_1,\ldots,\xi_{d+1}) = \phi(\xi_1,\ldots,\xi_d) + \E^{\l (\E^{\I \xi_{d+1}}-1)}$. Since $$\normA{\widehat{\mu^+} - \widehat{\nu^+}}=\normA{\widehat{\mu}-\widehat{\nu}}, $$ 
we can apply our result in dimension $d+1$, assuming $r \geq \lfloor \frac{d+1}{2} \rfloor$:
$$\normA{\widehat{\mu}-\widehat{\nu}} \leq C_3(r,d+1,c,\phi)\, (\normA{\psi-\chi}+ \beta_{r+1}(\eps))\,\frac{(\log \l)^{r+1}}{\l^{\frac{r}{2}+\frac{1}{4}}}.$$
However, for $d$ even, $\lfloor \frac{d+1}{2} \rfloor=\lfloor \frac{d}{2} \rfloor$, so the minimum $r$ that one can take is again $\lfloor \frac{d}{2} \rfloor$.
\end{proof}
\medskip

\noindent In comparison to the one dimensional case, we lost a logarithmic factor $(\log \l)^{r+1}$ in our estimate of norms (\emph{cf.} Theorem \ref{thm:concentrationinequality}). This will not be a problem for the calculation of asymptotics of distances.
\bigskip

\subsection{Asymptotics of distances in the multi-dimensional setting}\label{subsec:multidistances}
As in the one-dimensional case, a combination of the Laplace method and of the norm estimates given by Corollary \ref{cor:multinormestimate} yields under appropriate hypotheses the asymptotics of $\dloc(\mu_n,\nu_n)$ and of $\dtv(\mu_n,\nu_n)$, where $(\nu_n)_{n \in \N}$ is a general scheme of approximation of the laws $\mu_n$ of the mod-$\phi$ convergent random variables $X_n$ in $\Z^d$. The only difference is that, for the computation of the total variation distance, we shall need to assume the order of approximation $r$ to be larger than some minimal value $\lfloor\frac{d}{2}\rfloor$.\bigskip

Until the end of this paragraph, $(X_n)_{n \in \N}$ is a sequence of random variables with values in $\Z^d$, which converges mod-$\phi$ with parameters $\l_n \to \infty$ and limit residue $\psi(\xi_1,\ldots,\xi_d)$. In the sequel, we shall write $(\nu_n)_{n \in \N}$ a sequence of signed measures on $\Z^d$ such that $\widehat{\nu_n}(\xi)=\E^{\l_n \phi(\xi)}\,\chi_n(\xi_1,\ldots,\xi_d)$, with $\lim_{n \to \infty} \chi_n(\xi_1,\ldots,\xi_d) = \chi(\xi_1,\ldots,\xi_d)$. The multi-dimensional analogue of the hypothesis \eqref{eq:H1} is:
\begin{align*}
\forall n\in \N,\,\,\,&\psi_n(\xi)-\chi_n(\xi) = \sum_{\alpha_1+\cdots+\alpha_d = r+1} \beta_n^{(\alpha_1,\ldots,\alpha_d)}\,(\I\xi_1)^{\alpha_1}\cdots (\I\xi_d)^{\alpha_d}\,(1+o_{\xi}(1));\\
\text{and}\,\,\,&\psi(\xi)-\chi(\xi) = \sum_{\alpha_1+\cdots+\alpha_d = r+1} \beta^{(\alpha_1,\ldots,\alpha_d)}\,(\I\xi_1)^{\alpha_1}\cdots (\I\xi_d)^{\alpha_d}\,(1+o_{\xi}(1))\tag{H1d}\label{eq:H1d}
\end{align*}
with $\lim_{n \to \infty} \beta_n^{\alpha} = \beta^{\alpha}$ for any choice of index $\alpha=(\alpha_1,\ldots,\alpha_d)$ with $|\alpha|=r+1$. On the other hand, we write in the following $\|x\|^2 = \sum_{i=1}^d (x_i)^2$, $H_{\alpha}(x) = \prod_{i=1}^d H_{\alpha_i}(x_i)$, $\sigma^{\alpha} = \prod_{i=1}^d (\sigma_i)^{\alpha_i}$, and finally
$$\frac{k-\l_n m}{\sqrt{\l_n}\,\sigma} = \left(\frac{k_i-\l_n m_i}{\sqrt{\l_n}\,\sigma_i}\right)_{i \in \lle 1,d\rre}$$
if $k \in \Z^d$.

\begin{proposition}
Under the assumption \eqref{eq:H1d} with any $r\geq 0$, one has
\begin{align*}
&\mu_n(k_1,\ldots,k_d) - \nu_n(k_1,\ldots,k_d) = \\
&\left(\sum_{|\alpha|=r+1} \frac{\beta^{\alpha}}{\sigma^\alpha} \,\frac{1}{\prod_{i=1}^d (\sqrt{2\pi}\,\sigma_i)}\,\E^{-\frac{1}{2}\left\|\frac{k - \l_n m}{\sqrt{\l_n} \sigma}\right\|^2} \,H_{\alpha}\!\left(\frac{k-\l_n m}{\sqrt{\l_n}\,\sigma}\right)\right) \frac{1}{(\l_n)^{\frac{r+d+1}{2}}} + o\!\left(\frac{1}{(\l_n)^{\frac{r+d+1}{2}}}\right)
\end{align*}
with a remainder that is uniform over $\Z^d$.
\end{proposition}

\begin{proof}
The same arguments as in dimension $1$ ($d$-dimensional Fourier inversion formula and Laplace method) yield
\begin{align*}
&\mu_n(k_1,\ldots,k_d) - \nu_n(k_1,\ldots,k_d) \\
&= \sum_{|\alpha|=r+1} \frac{\beta_n^{\alpha}}{(2\pi)^d} \int_{[-\eps,\eps]^d} \left(\prod_{i=1}^d (\I\xi_i)^{\alpha_i}\,\E^{\l_n(\phi_i(\I\xi_i)-\I m_i\xi_i)}\,\E^{-\I\,(k-\l_n m_i)\xi_i} \right)\frac{d\xi}{(2\pi)^d} + o\!\left(\frac{1}{(\l_n)^{\frac{r+d+1}{2}}}\right) \\
&= \sum_{|\alpha|=r+1} \frac{\beta_n^{\alpha}}{(2\pi)^{\frac{d}{2}}} \prod_{i=1}^d \left(\frac{1}{(\sqrt{\l_n}\,\sigma_i)^{\alpha_i+1}} \int_{\R} (\I x)^{\alpha_i}\,\E^{-\frac{x^2}{2}}\,\E^{-\I \left(\frac{k_i-\l_n m_i}{\sqrt{\l_n}\,\sigma_i}\right)x}\frac{dx}{\sqrt{2\pi}}\right)  + o\!\left(\frac{1}{(\l_n)^{\frac{r+d+1}{2}}}\right),
\end{align*}
and the same Hermite polynomials as in Proposition \ref{prop:localestimate} appear.
\end{proof}
\bigskip

Using the same argument of density of the sequence of lattices $\{\frac{k-\l_n m}{\sqrt{\l_n}\,\sigma},\,\,k \in \Z^d\}$ as in the proof of Theorem \ref{thm:mainloc}, we conclude that:

\begin{theorem}\label{thm:multiloc}
Under the hypothesis \eqref{eq:H1d}, one has
\begin{align*}
&\dloc(\mu_n,\nu_n) \\
&= \sup_{x \in \R^d}\left( \E^{-\frac{\|x\|^2}{2}}\,\left|\sum_{|\alpha|=r+1} \frac{\beta^{\alpha}\,H_{\alpha}(x)}{\sigma^\alpha}\right|\right) \times \frac{1}{(\prod_{i=1}^d \sqrt{2\pi}\,\sigma_i)\,(\l_n)^{\frac{r+d+1}{2}}}\,+\,o\!\left(\frac{1}{(\l_n)^{\frac{r+d+1}{2}}}\right).
\end{align*}
\end{theorem}
\bigskip

The proof of the multi-dimensional analogue of Theorem \ref{thm:maindtv} is a bit more involved, but the main difficulty was lying in the proof of the norm estimate (Corollary \ref{cor:multinormestimate}). In the following we fix $r \geq \lfloor \frac{d}{2} \rfloor$; then, $r+1 \geq K_d$. The multi-dimensional version of the hypothesis \eqref{eq:H2} is the following. We assume the residues $\psi$, $\psi_n$, $\chi$ and $\chi_n$ to be of class $\mathscr{C}^{K_d}$ on $\T^d$, with $\partial^{\delta}(\psi_n-\chi_n)(0)=\partial^{\delta}(\psi-\chi)(0)=0$ for any $|\delta|<K_d$. We also assume that there exist families of coefficients $(\beta_n^\alpha)_{|\alpha|=r+1,n\in \N}$ and $(\gamma_n^\alpha)_{|\alpha|=r+2,n\in \N}$, such that for any multi-index $\delta$ with $|\delta|=K_d$, 
\begin{align*}
\partial^\delta(\psi_n-\chi_n)(\xi) &= \I^{|\delta|}\left(\sum_{|\rho|=r+1-|\delta|} \!\!\!\!\beta_n^{\rho+\delta}\,\frac{(\rho+\delta)!}{\rho!}\, (\I\xi)^{\rho} \,+\!\!\! \sum_{|\rho|=r+2-|\delta|} \!\!\!\!\gamma_n^{\rho+\delta}\,\frac{(\rho+\delta)!}{\rho!}\, (\I\xi)^{\rho}\,(1+o_\xi(1))\right)\\
\partial^\delta(\psi-\chi)(\xi) &= \I^{|\delta|}\left(\sum_{|\rho|=r+1-|\delta|} \!\!\!\!\beta^{\rho+\delta}\,\frac{(\rho+\delta)!}{\rho!}\, (\I\xi)^{\rho} \,+\!\!\! \sum_{|\rho|=r+2-|\delta|} \!\!\!\!\gamma^{\rho+\delta}\,\frac{(\rho+\delta)!}{\rho!}\, (\I\xi)^{\rho}\,(1+o_\xi(1))\right)\tag{H2d}\label{eq:H2d}
\end{align*}
with $\lim_{n \to \infty}\beta_n^\alpha = \beta^\alpha$ and $\lim_{n \to \infty}\gamma_n^\alpha = \gamma^\alpha$. Here, given a multi-index $\alpha$, we write $\alpha ! = \prod_{i=1} (\alpha_i)!$ and $(\I\xi)^{\alpha} = \prod_{i=1}^d (\I\xi_i)^{\alpha_i}$. These hypotheses are satisfied \emph{e.g.} if the convergences $\psi_n \to \psi$ and $\chi_n \to \chi$ occur in $\mathscr{C}^{r+2}(\T^d)$, and if $(\nu_n)_{n \in \N}$ is the scheme of approximation of order $r$ of the sequence of random variables $(X_n)_{n \in \N}$ (its definition in dimension $d \geq 2$ is similar to the definition in dimension $1$ given in Example \ref{ex:approximationorderr}, see also the two worked examples hereafter). \bigskip

Under the hypothesis \eqref{eq:H2d}, if one wants to prove a bound on $\dtv(\mu_n,\nu_n)$ of order $O((\l_n)^{-\frac{r+1}{2}})$, then one can replace $\nu_n$ with the measure $\rho_n$ defined by the Fourier transform
$$\widehat{\rho_n}(\xi) = \E^{\l_n\phi(\xi)} \left(\psi_n(\xi)-\sum_{|\alpha|=r+1} \beta_n^\alpha \prod_{i=1}^d(\E^{\I\xi_i}-1)^{\alpha_i}\right).$$
Indeed, denoting $\widetilde{\psi_n}(\xi)$ the residue associated to $\rho_n$, the hypothesis \eqref{eq:H2d} implies that, for any $|\delta|\leq K_d$,  
$$\left|\partial^{\delta}(\widetilde{\psi_n}-\chi_n)(\xi)\right| \leq C\,|\xi|^{r+2-|\delta|}$$
with a constant $C$ that is uniform, and also valid for the derivatives $|\partial^{\delta}(\widetilde{\psi}-\chi)(\xi)|$. This bound is the only one useful in the proof of Theorem \ref{thm:multiconcentrationinequality} and its Corollary \ref{cor:multinormestimate}, and therefore, one can apply Corollary \ref{cor:multinormestimate} with parameter $r+1$ instead of $r$. We thus get
$$\dtv(\rho_n,\nu_n) = O\!\left(\frac{(\log \l_n)^{r+2}}{(\l_n)^{\frac{r+1}{2}+\frac{1}{4}}}\right) = o\!\left(\frac{1}{(\l_n)^{\frac{r+1}{2}}}\right).$$
From there, the proofs are essentially identical to the one dimensional case (Theorem \ref{thm:maindtv}), and one obtains:
\begin{theorem}\label{thm:multidtv}
Consider a general approximation scheme $(\nu_n)_{n \in \N}$ of a mod-$\phi$ convergent sequence $(X_n)_{n \in \N}$ in $\Z^d$. We assume that $r \geq \lfloor \frac{d}{2}\rfloor$ and that the hypothesis \eqref{eq:H2d} is satisfied. Then,
$$\dtv(\mu_n,\nu_n) = \frac{1}{(2\pi)^{\frac{d}{2}}\,(\l_n)^{\frac{r+1}{2}}}\left(\int_{\R^d} \E^{-\frac{\|x\|^2}{2}}\,\left|\sum_{|\alpha|=r+1} \frac{\beta^\alpha\,H_{\alpha}(x)}{\sigma^\alpha}\right|\,dx\right) + o\!\left(\frac{1}{(\l_n)^{\frac{r+1}{2}}}\right).$$
\end{theorem}

\noindent Last, the discussion of \S\ref{subsec:derived} holds \emph{mutatis mutandis} in the multi-dimensional setting, and if $\|\chi_n-\chi\|_{\mathscr{C}^{K_d}}$ is asymptotically smaller than any negative power of $\l_n$, then one can replace $(\nu_n)_{n \in \N}$ by the derived scheme $(\sigma_n)_{n\in \N}$ with constant residue and keep the conclusions of Theorems \ref{thm:multiloc} and \ref{thm:multidtv}.
\bigskip

\subsection{Cycle-type of random coloured permutations}\label{subsec:colouredpermutation}
In this section, we study random elements of the group of $d$-coloured permutations:
$$ G_n = \sym_n \wr (\Z/d\Z)= \sym_n \ltimes \left( \Z / d \Z \right)^n.$$
It can be viewed as the group of complex matrices of size $n$ times $n$, with one non-zero coefficient on each row and each column, and these non-zero coefficients belonging to the group of $d$-th roots of unity $\Z/d\Z=\{1,\omega,\omega^2,\ldots,\omega^{d-1}\}$, with $\omega = \E^{\frac{2\I \pi}{d}}$. Another alternative definition is the following:
$$G_n = \{\sigma \in \sym_{nd} = \sym(\Z/n\Z \times \Z/d\Z) \,\,|\,\,\forall (a,b) \in \Z/n\Z \times \Z/d\Z,\,\,\,\sigma(a,b+1) = \sigma(a,b)+(0,1) \}.$$
From this definition, it appears that a $d$-coloured permutation is entirely determined by the images of the elements $(a,0)$ with $a \in \Z/n\Z$:
$$\sigma(a,0)=(\rho(a),k(a))\quad \text{with }\rho \in \sym_n=\sym(\Z/n\Z).$$
One then has $\sigma(a,b)=(\rho(a),k(a)+b)$. In particular, $\card G_n = n!\,d^n$, the factor $n!$ coming from the choice of $\rho \in \sym_n$, and the factor $d^n$ from the choice of the elements $k(a)$.\bigskip

Let us identify the conjugacy classes of elements of $G_n$ (we refer to \cite[\S2.3]{CSST14}). Suppose that $\sigma_2 = \tau\,\sigma_1\,\tau^{-1}$ in $G_n$, with $\sigma_i$ associated to the pair $(\rho_i,k_i) \in \sym_n \times (\Z/d\Z)^n$, and $\tau$ associated to the pair $(\nu,l)$. Note that $\tau^{-1} = (\nu^{-1},-l\circ \nu^{-1})$, this identity coming from the product rule $(\rho_1,k_1)(\rho_2,k_2)=(\rho_1\rho_2,k_1\circ \rho_2+k_2)$. Then,
\begin{align*}
(\rho_2,k_2) &= (\nu,l) (\rho_1,k_1)(\nu^{-1},-l\circ \nu^{-1}) \\
&=(\nu,l) (\rho_1\nu^{-1},(k_1-l)\circ\nu^{-1}) \\
&=(\nu\rho_1\nu^{-1},(k_1+l\circ\rho_1-l)\circ \nu^{-1}),
\end{align*}
so in particular $\rho_1$ and $\rho_2$ are conjugated in $\sym_n$, so they have the same cycle-type (an integer partition $L$ of size $|L|=n$). Then, to determine the conjugacy class of $(\rho_1,k_1)$, taking the conjugate of $(\rho_2,k_2)$ by $(\nu^{-1},0)$, we can assume that $\rho_1=\rho_2=\rho$, in which case
$$\sigma_1 = (\rho,k_1)\qquad;\qquad \sigma_2=(\rho, k_2=k_1 + (l\circ \rho-l)),$$
that is to say that $k_2$ and $k_1$ differ by a \emph{cocyle} $l\circ \rho -l$. This is possible if and only if, for every cycle $c = (a=\rho^r(a),\rho(a),\rho^2(a),\ldots,\rho^{r-1}(a))$ of $\rho \in \sym_n$, one has
$$k_1(a)+k_1(\rho(a)) + \cdots + k_1(\rho^{r-1}(a)) = k_2(a)+k_2(\rho(a)) + \cdots + k_2(\rho^{r-1}(a)).$$
Indeed, one has
$$\sum_{j=0}^{r-1}(l\circ \rho - l)(\rho^{j}(a)) = l(\rho^r(a))-l(a) = l(a)-l(a)=0. $$
In the previous setting, denote $\varSigma(c,k) = \sum_{j=0}^{r-1} k(\rho^{j}(a))$, which is an element of $\Z/d\Z$. This quantity depends only on the cycle $c$, and not on the choice of a representative $a$ of the corresponding orbit; it allows to one to ``color'' each cycle of $\rho$. Then (\emph{cf.} \cite[Theorem 2.3.5]{CSST14}, applied to the abelian case):

\begin{theorem}
Two coloured permutations $\sigma_1=(\rho_1,k_1)$ and $\sigma_2=(\rho_2,k_2)$ are conjugated in $G_n = \sym_n \wr (\Z/d\Z)$ if and only if:\vspace{2mm}
\begin{enumerate}
	\item the permutations $\rho_1$ and $\rho_2$ have the same lengths of cycles (encoded by an integer partition $L$ of size $n$);\vspace{2mm}
	\item and there is a size-preserving bijection $c \mapsto \psi(c)$ between the cycles of $\rho_1$ and $\rho_2$, such that the colors are also preserved:
	$$\varSigma(c,k_1) = \varSigma(\psi(c),k_2).$$
\end{enumerate}
\end{theorem}
\noindent Consequently, the conjugacy classes of $G_n$ are labelled by the $d$-uples of integer partitions $L=(L^{(1)},L^{(2)},\ldots,L^{(d)})$, such that $\sum_{i=1}^d|L^{(i)}| = n$ . If $\sigma \in G_n$ is of type $L$, then the parts of $L^{(i)}$ are the sizes of the cycles of $\sigma$ with color $i \in \Z/d\Z$.

\begin{example}
Consider the following $2$-coloured permutation in $G_8 = \sym_8 \wr \Z/2\Z$:
$$\sigma = \big(38762415,11010001\big),$$
where the first part is the word $\rho(1)\rho(2)\ldots\rho(8)$ of the permutation $\rho$ of size $8$ underlying $\sigma$, and the second part is the word $k(1)k(2)\ldots k(8)$. The three disjoint cycles of $\rho$ are $(1,3,7)$, $(2,8,5)$ and $(4,6)$, and the associated colors are $1+0+0=1$, $1+1+0=0$ and $1+0=1$. Therefore, the cycle-type of $\sigma$ is
$$L^{(1)}=(3,2),L^{(2)}=(3)$$
since the sizes of the cycles of color $1$ are $3$ and $2$, and the size of the unique cycle of color $0$ is $3$. We leave the reader check that the other $2$-coloured permutation
$$\sigma'=\big(73864512,10011100\big)$$
has the same cycle-type $((3,2),(3))$. Therefore, $\sigma$ and $\sigma'$ are conjugated in $G_n$.
\end{example}
\bigskip

In the following, we denote $\sigma_n$ a random uniform element of the wreath product $G_n$, and $\ell_n = (\ell_n^{(1)},\ell_n^{(2)},\ldots,\ell_n^{(d)})$ the random $d$-uple of integers that consists in the lengths of the partitions $L^{(1)},\ldots,L^{(d)}$ of the cycle-type of $\sigma_n$. Our goal is to apply the results of the previous paragraph to these random vectors. To this purpose, it is convenient to construct a coupling of all the random permutations $(\sigma_n)_{n\in \N}$, which generalizes the Feller coupling in the case $d=1$ of the Feller coupling. Algebraically, this amounts to the following:

\begin{lemma}
We consider $G_n$ as a subgroup of $G_{n+1}$, by sending a pair 
$$(\rho \in \sym_n,(k_1,\ldots,k_n) \in (\Z/d\Z)^n)\quad\text{to the pair}\quad (\rho \in \sym_{n+1},(k_1,\ldots,k_n,0)\in (\Z/d\Z)^{n+1}).$$
Let $\sigma_{n+1}$ be an element of $G_{n+1}$. There exists a unique element $\sigma_n \in G_{n}$, and a unique coloured transposition
$$\tau_n=((j,n+1),(0,\ldots,0,\alpha,0,\ldots,0,-\alpha))\quad \text{with }j\in \lle 1,n+1\rre,$$
and the $\alpha \in \Z/d\Z$ at position $j$, such that $\sigma_{n+1}=\tau_n\sigma_n$. Here we agree that the trivial transposition $(n+1,n+1)$ is the identity permutation; in this case we put $\alpha$ in position $n+1$, and there is no term $-\alpha$.
\end{lemma}

\begin{proof}
Denote $\sigma_{n+1}=(\rho_{n+1},(k_1,\ldots,k_{n+1}))$. If $\rho_n = \rho_{n+1} \circ (j=\rho_{n+1}^{-1}(n+1),n+1)$, then $\rho_n$ sends $n+1$ to $n+1$, so it can be considered as a permutation in $\sym_n$. We then have $\rho_{n+1} = \rho_n \circ (j,n+1)$. Set
$$\sigma_n = (\rho_n,(k_1,\ldots,k_{j-1},k_j+k_{n+1},k_{j+1},\ldots, k_n)) \in G_n.$$
We then have
\begin{align*}
\sigma_n\tau_n&=(\rho_n,(k_1,\ldots,k_j+k_{n+1},\ldots,k_n,0)) ((j,n+1),(0,\ldots,k_j,\ldots,-k_j)) \\
&=(\rho_n \circ (j,n+1),(k_1,\ldots,0,\ldots,k_n,k_j+k_{n+1})+(0,\ldots,k_j,\ldots,0,-k_j))\\
&=(\rho_{n+1},(k_1,\ldots,k_{n+1})) = \sigma_{n+1}.
\end{align*}
The unicity of $\tau_n$ comes from a cardinality argument: $\frac{\card G_{n+1}}{\card G_n} = (n+1)d$, and this is the number of $d$-coloured transpositions $((j,n+1),(0,\ldots,0,\alpha,0,\ldots,0,-\alpha))$.
\end{proof}
\bigskip

We denote $T(j,\alpha)$ the coloured transposition of the previous lemma. Then, in order to construct a random coloured permutation $\sigma_n \in G_n$, it suffices to take random independent positive integers $j_1 \leq 1,j_2 \leq 2,\ldots,j_n \leq n$, and random independent elements $\alpha_1,\ldots,\alpha_n \in \Z/d\Z$, and to consider the product
$$\sigma_n=T(j_1,\alpha_1)\circ T(j_2,\alpha_2)\circ \cdots \circ T(j_n,\alpha_n).$$
We also denote $(\mathcal{F}_n)_{n \in \N}$ the filtration of probability spaces associated to the sequence of random coloured permutations $(\sigma_n)_{n \in \N}$. The interest of our construction is that one can easily follow the evolution of the cycle type of $\sigma_n$. More precisely, if $j_n=n$, then to construct $\sigma_n$, one adds a cycle of length $1$ to $\sigma_{n-1}$, with color $\alpha_n$. Thus, for every color $i \in \Z/d\Z$, there is a probability $\frac{1}{nd}$ to increase the length $\ell^{(i)}$ by one unit:
$$\forall i \in \lle 1,d\rre,\,\,\,\proba[\ell_n = \ell_{n-1}+ (0,\ldots,0,1_i,0,\ldots,0)|\mathcal{F}_{n-1}] = \frac{1}{nd}.$$
On the other hand, if $j_n \neq n$, then the multiplication by $T(j_n,\alpha_n)$ increases by $1$ the size of the cycle containing $j_n$ in $\sigma_{n-1}$, but it does not change its color, because of the terms $\alpha_n$ and $-\alpha_n$ that compensate one another. So,
$$\proba[\ell_n = \ell_{n-1}|\mathcal{F}_{n-1}]= 1-\frac{1}{n}.$$
We have therefore proven that $\ell_n$ is a sum of independent increments in $\Z^d$, with
$$\proba[\ell_n-\ell_{n-1} = 0] = 1-\frac{1}{n}\quad;\quad \proba[\ell_n-\ell_{n-1}=e_i] = \frac{1}{nd},$$
where $(e_i)_{i \in \lle 1,d\rre}$ is the canonical basis of the lattice $\Z^d$.\bigskip

As an immediate consequence of the previous discussion, we have:
\begin{theorem}
The sequence of random vectors $(\ell_n)_{n \in \N}$ converges modulo a $d$-dimens\-ion\-al Poisson law of exponent $\phi(\xi)=\frac{1}{d}\sum_{i=1}^d \,(\E^{\I \xi_i}-1)$, with parameters $H_n$ and limiting function
\begin{align*}
\psi(\xi) &= \prod_{n=1}^\infty \left(1+\frac{1}{nd}\sum_{i=1}^d (\E^{\I\xi_i}-1)\right)\E^{-\frac{1}{nd}\,\sum_{i=1}^d (\E^{\I\xi_j}-1)} \\
&= \E^{\gamma\,\phi(\xi)}\,\frac{1}{\Gamma\!\left(\frac{1}{d}\sum_{i=1}^d \E^{\I\xi_i}\right)}.
\end{align*}
\end{theorem}

\begin{proof}
Since the increments of $\ell_n$ are independent,
\begin{align*}
\esper[\E^{\I\scal{\xi}{\ell_n}}] &= \prod_{j=1}^n \left(1+\frac{1}{jd}\sum_{i=1}^d (\E^{\I\xi_i}-1)\right) =\E^{H_n\,\phi(\xi)}\, \psi_n(\xi)
\end{align*}
where $\psi_n(\xi)$ is the partial product over indices in $\lle 1,n\rre$ coming from the infinite product $\psi(\xi)$.
\end{proof}
\bigskip

In particular, since the Fourier transform of $\ell_n$ does not factorize over the coordinates $\xi_i$, the coordinates of $\ell_n$ are not independent, though in the asymptotics $n \to \infty$ and at first order, $\ell_n$ looks like a $d$-uple of independent Poisson variables of parameters $\frac{H_n}{d}$. In the following, in order to simplify a bit the discussion, we assume $d=2$, which already contains all the subtleties of the general case. We denote $(x,y)$ the coordinates in $\R^2$, and $(\xi,\zeta)$ the coordinates in the Fourier space. We can then write:
\begin{align*}
\psi_n(\xi,\zeta) &= \prod_{j=1}^n \left(1+\frac{1}{j} \left(\frac{\E^{\I\xi}+\E^{\I\zeta}}{2}-1\right)\right)\,\E^{-\frac{1}{j} \left(\frac{\E^{\I\xi}+\E^{\I\zeta}}{2}-1\right)} \\
&= \sum_{k=0}^\infty \frake_{k,n}\,\left(\frac{\E^{\I\xi}+\E^{\I\zeta}}{2}-1\right)^k = 1-\frac{\frakp_{2,n}}{2} \left(\frac{\E^{\I\xi}+\E^{\I\zeta}}{2}-1\right)^2 + o(\|(\xi,\zeta)\|^2)
\end{align*}
where the $\frake_{k,n}$'s are obtained from the parameters 
$$\frakp_{1,n} = 0\qquad;\qquad \frakp_{k\geq 2,n} = \sum_{j=1}^n \frac{1}{j^k}$$
by the same recipe as in Section \ref{sec:oneexample}. The scheme of approximation of $(\ell_n)_{n \in \N}$ of order $r=1$ is the Poissonian approximation
$$\widehat{\nu_n}(\xi,\zeta) = \E^{H_n\left(\frac{\E^{\I \xi}+\E^{\I \zeta}}{2} -1\right)}\qquad;\qquad\nu_n(k,l)=\E^{-H_n}\,\frac{(H_n)^{k+l}}{2^{k+l}\,k!\,l!}.$$
It satisfies the hypotheses \eqref{eq:H1d} and \eqref{eq:H2d}, with $\beta^{(2,0)}_n = \beta^{(0,2)}_n=-\frac{\frakp_{2,n}}{8}$ and $\beta^{(1,1)}_n = -\frac{\frakp_{2,n}}{4}$. Since $H_{(2,0)}(x,y)=x^2-1$, $H_{(1,1)}(x,y)=xy$
and $H_{(0,2)}(x,y)=y^2-1$, Theorems \ref{thm:multiloc} and \ref{thm:multidtv} ensure that
\begin{align*}
\dloc(\mu_n,\nu_n) &= \frac{\pi}{6\,(\log n)^2}\,\sup_{(x,y)\in \R^2}\left|\E^{-\frac{x^2+y^2}{2}}\,(2-(x+y)^2)\right| + o\!\left(\frac{1}{(\log n)^2}\right);\\
\dtv(\mu_n,\nu_n) &= \frac{\pi}{24\,\log n}\,\left(\int_{\R^2} \E^{-\frac{x^2+y^2}{2}}\,|2-(x+y)^2|\,dx\,dy\right)+ o\!\left(\frac{1}{\log n}\right)
\end{align*}
since $\sum_{|\alpha|=2} \frac{\beta^\alpha\,H_\alpha(x)}{\sigma^\alpha} = -\frac{\frakp_2}{2}((x+y)^2-2) = \frac{\pi^2}{12}\,(2-(x+y)^2)$. 
\begin{center}
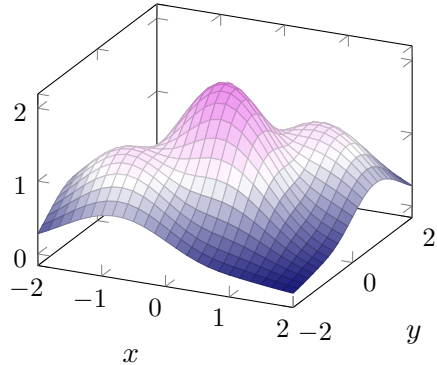
\begin{figure}[ht]
\pgfplotsset{width=10cm} 
\begin{tikzpicture} 
\begin{axis}[
	colormap/violet,
    title={},
    xlabel=$x$, ylabel=$y$,
    small,
] 

\addplot3[
    surf,
    domain=-2:2,
    domain y=-2:2,
]
    {exp(-(x^2+y^2)/2)*sqrt(4-2*(x+y)^2+(x+y)^4)};
\end{axis}
\end{tikzpicture}
\caption{The function $\E^{-\frac{x^2+y^2}{2}}\,|2-(x+y)^2|$.\label{fig:function2variables}}
\end{figure}
\end{center}

For the local distance, one checks at once that the maximum of the function $f(x,y)=\E^{-\frac{x^2+y^2}{2}}\,|2-(x+y)^2|$ is obtained when $x=y=0$, so
$$\dloc(\mu_n,\nu_n) = \frac{\pi}{3\,(\log n)^2}\,(1+o(1)),$$
see Figure \ref{fig:function2variables}.
For the total variation distance, a computer algebra system (\texttt{Sage}) yields the approximate value $12.162\dots$ for the integral. As in the one-dimensional case (Example \ref{ex:sigman} and \S\ref{subsec:disjointcycles}), one can on the other hand construct better schemes of approximations, which yield distances smaller than any arbitrary negative power of $\log n$. It is important to notice that the dependence between the different coordinates of $\ell_n$ is directly involved in these asymptotics of distances.
\bigskip

\subsection{Distinct prime divisors counted according to their residue classes}\label{subsec:residueclass}
Similar to the generalisation of \S\ref{subsec:disjointcycles} by the study of random coloured permutations, there is a natural generalisation of the discussion of \S\ref{subsec:erdoskac} on distinct prime divisors of a random integer, which involves \emph{residue classes} of prime numbers.\bigskip

Fix an integer $a \geq 2$, and denote $d=\varphi(a)$ the cardinality of the multiplicative group $(\Z/a\Z)^*$. This quantity is the usual Euler $\varphi$-function, and it is equal to the number of integers in $\lle 1,a\rre$ that are coprime to $a$. We label the elements of $(\Z/a\Z)^*$ as $b_1,b_2,\ldots,b_d$. If $n \in \N$ and $i \in \lle 1,d\rre$, we denote $\omega_i(n)$ the number of distinct prime divisors of $n$ that have residue class $b_i$ modulo $a$:
$$\omega_i(n) = \card\{p \in \P,\,\,p\mathrel{|}n \text{ and }p\equiv b_i \bmod a\} .$$
One then has
$$\sum_{i\in (\Z/a\Z)^*} \omega_i(n) = \omega(n) - \card\{p \in \P,\,\,p\mathrel{|}a \text{ and }p\mathrel{|}n\}.$$
We are interested in the random vector $\Omega(N_n)=(\omega_1(N_n),\omega_2(N_n),\ldots,\omega_d(N_n))$, where $N_n$ is a random integer chosen uniformly in $\lle 1,n\rre$. In view of the theory developed in the previous sections, the asymptotics of these vectors are encoded in the multiple generating series
$$\frac{1}{n}\sum_{i=1}^n (z_1)^{\omega_1(i)}(z_2)^{\omega_2(i)}\cdots (z_d)^{\omega_d(i)}.$$
In number theory, the asymptotics of such quantities are classically related to the behavior of the Dirichlet series
$$F(z,s) = \sum_{n=1}^\infty \frac{(z_1)^{\omega_1(n)}(z_2)^{\omega_2(n)}\cdots (z_d)^{\omega_d(n)}}{n^s}.$$
In a moment we shall precise these relations, which amount to the so-called Selberg--Delange method, see \cite[Chapter II.5]{Ten95}. Note that, though we want to study a random vector in $\Z^{d \geq 2}$, the Dirichlet series written above is not a multiple Dirichlet series.\bigskip	

The main algebraic tool that is required in order to use Selberg--Delange method is the theory of Dirichlet characters and their $L$-series, see for instance Chapter II.8 in \cite{Ten95}. In the following we recall the basics of this theory. The space of functions $(\Z/a\Z)^* \to \C$ is endowed with a Hilbert structure
$$\scal{f}{g} = \frac{1}{\varphi(a)}\,\sum_{i \in (\Z/a\Z)^*} f(i)\overline{g(i)},$$
and a Hilbert basis consists of the so-called \emph{Dirichlet characters} $\chi_{a,1},\ldots,\chi_{a,d}$, which are the morphisms of (multiplicative) groups $\chi : (\Z/a\Z)^* \to \C^*$. In particular, there are as many Dirichlet characters as elements of the group $(\Z/a\Z)^*$, \emph{i.e.}, $d$ distinct characters. In the following, we denote $\chi_{a,1}$ the trivial character: $\chi_{a,1}(i)=1$ for all $i \in (\Z/a\Z)^*$.\medskip

If $\chi$ is a character of $(\Z/a\Z)^*$, introduce its $L$-function 
$$L(\chi,s) = \sum_{n=1}^\infty \frac{\chi(n)}{n^s}, $$
where the function $\chi$ is extended to $\Z$ by 
$$\chi(n) = \begin{cases}
	\chi(m) &\text{ if }(n,a)=1\, \text{ and }m\equiv n \bmod a,\\
	0 &\text{ if }(n,a)>1.
\end{cases}$$
These functions admit an Euler product representation:
$$L(\chi,s) = \prod_{p \in \P_a} \left(1-\frac{\chi(p)}{p^s}\right)^{-1},$$
where $\P_a$ is the set of prime numbers that do not divide $a$. In particular, $L(\chi_{a,1},s)$ is almost the same as Riemann's $\zeta$-function:
$$L(\chi_{a,1},s) = \prod_{p \in \P_a} \left(1-\frac{1}{p^s}\right)^{-1} = \zeta(s)\,\prod_{p\mathrel{|} a}\left(1-\frac{1}{p^s}\right).$$
Therefore, the $L$-function associated to the trivial character has abscissa of convergence $1$, and can be extended to a meromorphic function on the half-plane $\Re(s)>0$, with a single pole at $s=1$. On the other hand, for the other characters $\chi_{a,2},\ldots,\chi_{a,d}$, the corresponding $L$-functions converge simply towards holomorphic functions on the same half-plane $\Re(s)>0$.\bigskip	

We now form the Dirichlet series
$$F(z,s) = \sum_{n \geq 1} \frac{(z_1)^{\omega_1(n)} \cdots (z_d)^{\omega_d(n)}}{n^s}.$$
For any choice of parameters $z_1,\ldots,z_d \in \C$, the series $F(z,s)$ converges absolutely on the half-plane $\Re(s)>1$, as can be seen from the inequality
$$\sum_{n \geq 1} \left|\frac{(z_1)^{\omega_1(n)} \cdots (z_d)^{\omega_d(n)}}{n^s}\right| \leq \prod_{p \in \P_a} \left(1+\frac{\max_{i\in \lle 1,d\rre}|z_i|}{p^s-1}\right).$$
Given parameters $y_1,\ldots,y_d$, we set
$$G_{\chi}(z,s) = F(z,s)\,\prod_{j=1}^d L(\chi_{a,j},s)^{-y_i}.$$
Let us choose the parameters $y_1,\ldots,y_d$ so that the series $G_\chi(z,s)$ is an holomorphic function on $\Re(s)>\frac{1}{2}$. If $\Re(s)>1$, we can write without ambiguity
\begin{align*}
F(z,s) &= \prod_{i=1}^d \prod_{p \equiv b_i\bmod a} \left(1+\frac{z_i}{p^s-1}\right); \\
G_{\chi}(z,s) &= \prod_{i=1}^d \prod_{p \equiv b_i\bmod a} \left(1+\frac{z_i}{p^s-1}\right)\,\,\prod_{j=1
}^d \left(1-\frac{\chi_{a,j}(p)}{p^s}\right)^{y_j}\\
&=\prod_{p \in \P_a} \left(\left(1+\frac{z_{i(p)}}{p^s-1}\right)\prod_{j=1}^d \left(1-\frac{\chi_{a,j}(b_{i(p)})}{p^s}\right)^{y_j} \right)\end{align*}
where $i(p)$ is the unique $i \in \lle 1,d\rre$ such that $p \equiv b_i\bmod a$. The Taylor expansion of the term corresponding to $p \in \P_a$ is the previous product is
$$1 + \frac{1}{p^s} \left( z_{i(p)}- \sum_{j=1}^d y_j \,\chi_{a,j}(b_{i(p)}) \right) + O\!\left(\frac{1}{p^{2s}}\right).$$
Set $y_j = \frac{1}{d}\sum_{k=1}^d z_k\,\overline{\chi_{a,j}(b_k)}$. Then, by orthogonality of the Dirichlet characters, one has:
\begin{align*}
\sum_{j=1}^d y_j\,\chi_{a,j}(b_i) &= \frac{1}{d}\sum_{j,k=1}^d z_k \overline{\chi_{a,j}(b_k)}\,\chi_{a,j}(b_i) \\
&=d\,\sum_{j=1}^d \scal{\sum_{k=1}^d z_kb_k}{\chi_{a,j}}\scal{\chi_{a,j}}{b_i}\\
&= d\,\scal{\sum_{k=1}^d z_k b_k}{b_i} = z_i.
\end{align*}
Hence, for this choice of parameters, $G_{\chi}(z,s) = \prod_{p \in \P_a} \left(1+O(p^{-2s})\right)$, so 
$G_{\chi}(z,d)$ is an holomorphic function on $\Re(s)>\frac{1}{2}$. Note that $y_1 = \frac{z_1+\cdots+z_d}{d}$. Now, by the previous discussion on $L$-series, one can remultiply $G_{\chi}(z,s)$ by $\prod_{j \neq 1} L(\chi_{a,j},s)^{y_j}$, hence:

\begin{proposition}
For any choice of complex parameters $z_1,\ldots,z_d$, the series
$$G(z,s) = F(z,s)\,(\zeta(s))^{-\frac{z_1+\cdots+z_d}{d}}$$
has an holomorphic extension on the half-plane $\Re(s)>\frac{1}{2}$.
\end{proposition}

\begin{proof}
If $L(\chi_{a,1},s)=\zeta_a(s)$ is the partial $\zeta$-function associated to the integer $a$, then we shown that $F(z,s)\,(\zeta_a(s))^{-\frac{z_1+\cdots+z_d}{d}}$ is an holomorphic function on $\Re(s)>\frac{1}{2}$. It suffices then to multiply by the missing terms
\begin{equation*}
\prod_{p\mathrel{|} a} \left(1-\frac{1}{p^s}\right)^{\frac{z_1+\cdots+z_d}{d}}.\qedhere
\end{equation*}
\end{proof}
\medskip

We can now apply Theorem 3 in \cite[Chapter II.5]{Ten95}:

\begin{theorem}\label{thm:multiselbergdelange}
When $n$ goes to infinity,
$$\frac{1}{n}\sum_{i=1}^n (z_1)^{\omega_1(n)}\cdots (z_d)^{\omega_d(n)} = (\log n)^{\frac{z_1+\cdots+z_d}{d}-1}\,\frac{G((z_1,\ldots,z_d),1)}{\Gamma\!\left(\frac{z_1+\cdots+z_d}{d}\right)}\,\left(1+O\!\left(\frac{1}{\log n}\right)\right),$$
with a remainder that is locally uniform in the parameters $z_1,\ldots,z_d$.
\end{theorem}

\noindent As a corollary, the random vectors $(\Omega(N_n))_{n \in \N}$ of numbers of distinct prime divisors in each residue class of $(\Z/a\Z)^*$ converge mod-$\phi$, where $\phi(\xi) = \frac{1}{d}\sum_{j=1}^d(\E^{\I\xi_j}-1)$. The parameters of this mod-$\phi$ convergence are $\l_n = \log\log n$, and the limiting residue is
$$\psi(\xi) = \frac{G((\E^{\I\xi_1},\ldots,\E^{\I\xi_d}),1)}{\Gamma\!\left(\frac{\E^{\I\xi_1}+\cdots+\E^{\I \xi_d}}{d}\right)} .$$
Moreover, the convergence happens at speed $O(\frac{1}{\log n})$. As a consequence, one can construct explicit schemes of approximations of the laws of the vectors $(\Omega(N_n))_{n\in \N}$, which yield distances that are $O((\log\log n)^{-\frac{p}{2}})$ with $p\geq 1$ arbitrary (see Theorems \ref{thm:multiloc} and \ref{thm:multidtv}). The first of these schemes is the Poisson approximation:
$$\widehat{\nu_{n}}(\xi) = \E^{(\log\log n)\frac{1}{d}\sum_{i=1}^d (\E^{\I\xi_i}-1)}.$$
Unfortunately, it is then difficult to calculate the constants involved in these asymptotics of distances. Indeed, there are no simple expression for the values of $G((\E^{\I\xi_1},\ldots,\E^{\I\xi_d}),1)$ and its partial derivatives around $\xi=0$.
\bigskip
\bigskip

\bibliographystyle{alpha}
\bibliography{wiener}

\end{document}